\documentclass[10pt,hidelinks]{article}
\usepackage[T1]{fontenc}
\usepackage{amsmath,amsfonts,amsthm,mathrsfs,amssymb}
\usepackage{subcaption}
\usepackage{cite}
\usepackage{multirow}
\usepackage{algorithmic,algorithm}
\usepackage{graphicx,float}
\usepackage{placeins}
\usepackage{color}
\usepackage{indentfirst}
\usepackage[bookmarks=true]{hyperref}
\numberwithin{equation}{section}
\newcommand{\R}{{\mathbb R}}

\newcommand{\N}{{\mathbb N}}
\newcommand{\C}{{\mathbb C}}

\newcommand{\be}{\begin{eqnarray}}
\newcommand{\ben}{\begin{eqnarray*}}
\newcommand{\en}{\end{eqnarray}}
\newcommand{\enn}{\end{eqnarray*}}

\newcommand{\pa}{\partial}

\newcommand{\ov}{\overline}

\def\CC{{C\nolinebreak[4]\hspace{-.05em}\raisebox{.4ex}{\tiny\bf ++}} }
\newtheorem{theorem}{Theorem}[section]
\newtheorem{lemma}[theorem]{Lemma}

\newtheorem{definition}[theorem]{Definition}
\newtheorem{remark}[theorem]{Remark}

\newtheorem{condition}[theorem]{Condition}

\definecolor{Taor1}{rgb}{0.000,0.000,0.000}
\definecolor{Taor25}{rgb}{1.000,0.000,0.000}

\begin{document}
\renewcommand{\theequation}{\arabic{section}.\arabic{equation}}
\begin{titlepage}
  \title{Multi-patch/multiple-scattering frequency-time hybrid solver for interior and exterior wave equation problems}

\author{
Shuai Pan\thanks{School of Mathematical Sciences, University of Electronic Science and Technology of China, Chengdu, Sichuan 611731, China. Email:{\tt pans@std.uestc.edu.cn}},\;
Gang Bao\thanks{School of Mathematical Sciences, Zhejiang University, Hangzhou 310027, China. Email:{\tt baog@zju.edu.cn}},\;
Tao Yin\thanks{State Key Laboratory of Mathematical Sciences and Institute of Computational Mathematics and Scientific/Engineering Computing, Academy of Mathematics and Systems Science, Chinese Academy of Sciences, Beijing 100190, China. Email:{\tt yintao@lsec.cc.ac.cn}},\;
and Oscar P. Bruno\thanks{Department of Computing \& Mathematical Sciences, California Institute of Technology, 1200 East California Blvd., CA 91125, United States. Email:{\tt obruno@caltech.edu}}}
\end{titlepage}
\maketitle
%%\vspace{.2in}
%
\begin{abstract}
  This paper proposes a new multiple-scattering frequency-time hybrid
  (FTH-MS) integral equation solver for problems of wave scattering by
  obstacles in two dimensional space, including interior problems in
  closed cavities and problems exterior to a set of disconnected open
  or closed scattering obstacles. The multiple-scattering FTH-MS
  method is based on a partition of the domain boundary into a
  user-prescribed set of overlapping open arcs, along with a
  corresponding sequence of multiple-scattering problems that
  effectively decompose the interior problem into a series of open-arc
  wave equation subproblems. The new strategy provides a significant
  extension of the original FTH-MS algorithm originally presented in
  [{\it O.P. Bruno, T. Yin, Math. Comput. 93(346) (2024) 551-587}], in
  that (1)~By allowing for use of an arbitrary of number of component
  arcs, and not just two as in the previous contribution, the new
  approach affords (1a) A significantly increased geometric
  flexibility, as well as, (1b)~The use of partitions for which each
  open arc leads to small numbers of iterations if iterative
  linear-algebra solvers are employed; and, (2)~It facilitates
  parallelization---as the subproblem solutions that are needed at
  each multiple scattering step can be evaluated in an embarrassingly parallel 
  fashion. Utilizing a suitably-implemented Fourier transformation,
  each sub-problem is reduced to a Helmholtz frequency-domain problem
  that is tackled via a uniquely-solvable boundary integral
  equation. Similar FTH-MS methods are also presented for problems
  exterior to a number of bounded obstacles. All of the algorithms
  considered incorporate the previously introduced ``time-windowing
  and recentering'' methodology (that enables both treatment of
  incident signals of long duration and long time simulation), as well
  as a high-frequency Fourier transform algorithm that delivers
  numerically dispersionless, spectrally-accurate time evolution for
  arbitrarily long times.
\end{abstract}
  {\bf Keywords:} Wave equation, multiple scattering, Fourier transform, integral equation

\section{Introduction}
\label{sec:1}
We present efficient multi-patch/multiple-scattering frequency-time
hybrid algorithms (FTH-MS) for the solution of the time-dependent wave
equation in both interior and exterior two-dimensional domains.  The
proposed methods utilize a multiple scattering strategy that
decomposes a given interior time-domain problem into a sequence of
limited-duration time-domain problems of scattering by overlapping
open arcs, each one of which is reduced (by means of the Fourier
transform) to a sequence of Helmholtz frequency-domain problems. As
discussed and illustrated in the previous
contributions~\cite{ABL20,BY23}, the combined use of
frequency-domain boundary-integral solvers and frequency-to-time
high-frequency Fourier transform methods results in algorithms that
compare well with existing approaches---including methods based on use of
volumetric discretizations and temporal time-stepping
strategies~\cite{T00,FP96,XCS13,GSS06,SX21,BMPS21,LSZ21}; as well as
methods based on spatio-temporal Green functions and integral
equations; and methods based on a combination of a time-discretization
and integral equations in the Laplace frequency domain.  Indeed,
(a)~Like other boundary-integral methods, the boundary integral
schemes in the FTH-MS algorithm only require discretization of the lower-dimensional
domain boundary (thus enabling significant efficiency gains over
volumetric methods); and, (b)~The combined FTH-MS use of
boundary-integral frequency domain solutions and fast and effective
high-frequency Fourier transform methods results in essentially
spatio-temporally dispersionless algorithms. In particular, (c)~Unlike
other frequency-time hybrid methods, the FTH-MS algorithm does not
incur accumulating temporal dispersion errors and it avoids a certain
``infinite-tail'' difficulty mentioned below.

The volumetric type approaches, in contrast, generally suffer from
spatial and temporal numerical dispersion errors~\cite{BS97} (a
problem which can be mitigated~\cite{melenk2011wavenumber} if
high-order methods are utilized). Thus, for the types of volumetric
solvers often used in practice, applications to high-frequency and/or
large-time problems typically require the use of fine spatial and temporal
meshes---with associated large requirements of computer memory and
long computing times. Challenges have also been found in the context
of integral equation solvers, including the time-domain boundary
integral equation method (TDBIE) \cite{ADG11,BGH20,D03,SU22,SUZ21},
which proceeds on the basis of the retarded-potential Green's
function; and the ``Convolution Quadrature''
method~\cite{BK14,L94,S16} (CQ) which, utilizing a finite-difference
time discretization followed by a Laplace-like transformation in
discrete time, reduces the original time-domain problem to modified
Helmholtz problems over a range of frequencies that can be solved by
means of appropriate frequency-domain boundary integral equations (BIE). In
particular, the TDBIE method requires integration in challenging
domains given by the intersection of the light cone with the overall
scattering surface, and it has been found to suffer from numerical
instability~\cite{BGH20}. As mentioned in~\cite[Section 2.2.1]{ABL20}, on the other hand, the CQ method inherits
the dispersive character of the finite-difference approximation that
underlies its time discretization, and it gives rise to a challenging
infinite tail leading to unbounded memory requirements and increased
time-per-timestep as times grow.

The recently introduced frequency-time hybrid (FTH)
approach~\cite{ABL20} for exterior problems, in contrast, utilizes the
Fourier transform to reduce the time-domain problem to sequences of
frequency-domain problems for the Helmholtz equation. Relying on a
certain ``windowing and time-recentering'' procedure together with
high-frequency Fourier transformation algorithms, the FTH solver can
provide highly accurate numerical solutions at low computing costs,
with spectral accuracy in time, for problems involving complex
scatterers, for incident fields applied over long periods of time, and
with extremely low dispersion errors.

The FTH method~\cite{ABL20} is not applicable to either interior spatial
domains---in which the waves are perpetually trapped, and for which
the corresponding interior-domain Helmholtz problems are not uniquely
solvable at any frequency $k$ for which $-k^2$ an eigenvalue of the
Laplace operator. The recent contribution~\cite{BY23} provides an
extension of the FTH applicable to interior spatial domains. The
extended method proceeds by exploiting a novel multiple-scattering
technique that re-expresses the full time-evolution as a ``ping-pong'' problem of
multiple scattering between two overlapping patches (open-arcs) that cover the
domain boundary. As a result, the method~\cite{BY23},
which relies on a generalized Huygens-like domain-of-influence
condition related to amount of overlap between the two
patches, only requires solution of certain  {\em open-arc} frequency-domain scattering
problems  which are uniquely solvable at all real
frequencies. However, since it is limited to decomposing the domain boundary into only two patches, this ping-pong solver may  encounter significant difficulties with challenging wave-trapping patches---where iterative frequency-domain solvers often require many iterations.
Additionally,  the present FTH-MS generalization of the ping-pong multiple scattering algorithm~\cite{BY23}  includes a capability to enable the solution of problems of scattering by open
arcs---to which the ping-pong solver~\cite{BY23} is not applicable on account of its endpoint singularity-handling approach. 

Relying on decompositions of the domain boundary as a union of
arbitrary numbers of overlapping arcs, the new FTH-MS method proceeds
by adequately accounting for the multiple scattering among all pairs
of component arcs: it re-expresses an interior or exterior time-domain
problem in terms of a sequence of open-arc wave equation sub-problems,
each of which is solved by means of the FTH solver~\cite{ABL20}. We
emphasize that, unlike the previous method~\cite{BY23}, the new
approach can be applied to advantage to open-arc exterior problems as
well. By utilizing partitions in terms of sufficiently simple and
small open arcs, the new method reduces the overall problem to a
number of relatively small frequency domain problems which can be
treated effectively and with high accuracy by means of iterative or
direct linear-algebra solvers, and which can be parallelized
effectively in an embarrassingly parallel fashion.

In a major further improvement, the present contribution eliminates
the need for the ``smooth extension along normals'' singularity-handling technique
introduced in~\cite{BY23} to bypass singularities induced by the
endpoints of one patch at interior points of neighboring patches;
see~\cite[Fig. 3]{BY23}, the accompanying discussion, and the
introductory paragraph of Section~\ref{FTH-op-arc} below.  As detailed
in Section~\ref{FTH-op-arc}, the simplification is achieved by
leveraging the structure of the endpoint-induced singularities to
design appropriate ``smoothing'' changes of variables. As a result,
the overall approach is significantly streamlined, it is applicable to
open-arc exterior problems in addition to the previously considered
closed-arc configurations, and is well-suited for potential
generalizations to three-dimensional settings.

This paper is organized as follows. Section~\ref{sec:2} reviews the
basic elements of the FTH solver~\cite{ABL20}, introduces the
necessary Huygens-like domain-of-influence condition, and mentions
relevant solution regularity results. Section~\ref{sec:3.1} then
introduces the main multi-patch multiple scattering concept, and
presents the key theoretical result that validates the method:
Theorem~\ref{equivalence-int}.  Relying on these ideas,
Section~\ref{sec:3.2} introduces the multi-patch multiple-scattering
FTH-MS algorithm for interior domains, and Section~\ref{sec:3.3}
extends the method to exterior wave equation problems involving both
closed and open scattering obstacles.  The computational implementation used is described in
Section~\ref{sec:4}, including, in particular, a study in
Section~\ref{FTH-op-arc} which examines the aforementioned end-point
and induced singularities, and which introduces variable
transformations that remove these singularities and restore high-order
accuracy.  Numerical examples are presented in Section~\ref{sec:5}
that illustrate the accuracy and efficiency provided by the proposed
methodology.

\section{Preliminaries}
\label{sec:2}

As a preliminary to the presentation of the multi-patch
multiple-scattering FTH-MS solver for interior and exterior wave
equation problems, this section briefly lays down the wave problems
considered, and it reviews theoretical concepts underlying the
exterior and interior FTH-based methods~\cite{ABL20,BY23}---including
the windowing-and-recentering approach~\cite{ABL20}, the Huygens-like
domain-of-influence properties~\cite{BY23} and the regularity
properties~\cite{BH86,CHLM10} of the underlying solutions to wave
equation problems.

\subsection{Interior and exterior wave equation problems}
Let $D \subset \mathbb{R}^2$ denote an open domain with a smooth
boundary $\Gamma$, call
$D^\mathrm{e} = \mathbb{R}^2 \setminus \overline{D}$ its exterior, and
let $u^i(x,t)$ be a given smooth incident field defined for
$(x,t) \in \Gamma\times\mathbb{R}$ which vanishes for $t \leq 0$. In
this paper we consider the initial and boundary-value problems for the
wave equation in both $E = D$ and $E = {D}^\mathrm{e}$,
subject to Dirichlet boundary conditions $f$ on $\Gamma$,
\begin{eqnarray}
\label{waveeqn0}
\begin{cases}
\pa_t^2u(x,t)-c^2\Delta u(x,t)=0, & (x,t)\in E\times\R_+, \cr
u(x,0)=\pa_t u(x,0)=0, & x\in \textcolor{Taor1}{E}, \cr
u(x,t)= f(x,t),  & (x,t)\in\Gamma\times\R_+,
\end{cases}
\end{eqnarray}
where $\R_+:=\{t\in\R: t>0\}$, $f=-u^i$, and where $c>0$ and $u$
denote the wave-speed and the scattered field, respectively. Since
$u^i(x,t)=0$ for $t\leq 0$, problem (\ref{waveeqn}) can be equivalently
written in the form
\begin{eqnarray}
\label{waveeqn}
\begin{cases}
\pa_t^2u(x,t)-c^2\Delta u(x,t)=0, & (x,t)\in E\times\R, \cr
u(x,t)= f(x,t),  & (x,t)\in\Gamma\times\R,
\end{cases}
\end{eqnarray}
by invoking the causality condition $u(x,t)=0$ for
$x\in\textcolor{Taor1}{E}$ and $t\le0$. The Dirichlet problem of
scattering by an open arc $ \Gamma $, which is also considered in this
paper, can similarly be formulated using equations~\eqref{waveeqn0}
and~\eqref{waveeqn} by defining $E$ as the complement of $\Gamma$:
$E = \Gamma^\mathrm{c}$.

Throughout this paper a number of wave-equation problems will be
considered which, assuming vanishing boundary data for $t\leq 0$, will
be expressed in a form similar to (\ref{waveeqn}) in lieu of the
equivalent formulation---of the form~\eqref{waveeqn0}---in terms of
vanishing initial conditions at $t=0$. Throughout this paper a time
dependent function will be said to be {\em causal} if it vanishes for
$t\leq 0$.

\subsection{FTH solver on the exterior domain $E = D^\mathrm{e}$}
\label{sec:2.1}
The FTH solver~\cite{ABL20} produces solutions $u(x,t)$ for the
exterior domain $D^\mathrm{e}$ as indicated in what follows; as mentioned in
Section~\ref{sec:1}, a different approach must be used for the
interior domain $D$ on account of a lack of unique solvability for
interior Helmholtz problems. For a given time-domain function
$g\in L^2(\R)$, its Fourier transform, denoted by $G$, is given by \be
\label{forwardFT}
G(\omega)=\mathbb{F}(g)(\omega):=\int_\R
g(t)e^{i\omega t}\,dt,\quad \omega\in\R,
\en
whereas the corresponding inverse Fourier transform $\mathbb{F}^{-1}$ of a frequency-domain function $G\in L^2(\R)$ is given by
\be
\label{backwardFT}
g(t)=\mathbb{F}^{-1}(G)(\omega):=\frac{1}{2\pi}\int_\R
G(\omega)e^{-i\omega t}\,d\omega.
\en
Letting $U$ and $F$ denote the Fourier transforms of the solution $u$
and boundary values $f$, respectively, it follows that $U$ satisfies
the Helmholtz problem with wavenumber
$\kappa=\kappa(\omega)=\omega/c\in\R$: \be
\label{Helmholtzpro}
\begin{cases}
\Delta U(x,\omega)+\kappa^2 U(x,\omega)=0, & x\in D^\mathrm{e}, \cr
U(x,\omega)=F(x,\omega),  & x\in\Gamma.
\end{cases}
\en

Calling $\Phi_\omega:=i/4 H_0^{(1)}(\kappa|x-y|)$  the fundamental
solution of the Helmholtz equation in $\R^2$,  letting
$\mathcal{S}_{\Gamma}$ and $\mathcal{D}_{\Gamma}$ denote the
frequency-domain single-layer and double-layer potentials
\begin{eqnarray}
&& {\mathcal S}_{\Gamma}[\Psi](x,\omega)=\int_{\Gamma} \Phi_\omega(x,y) \Psi(y,\omega) ds_y,\quad (x,\omega)\in D^\mathrm{e}\times\R,\quad\mbox{and}\label{SL_rep}\\
&& {\mathcal D}_{\Gamma}[\Psi](x,\omega)=\int_{\Gamma}
\pa_{\nu_y}\Phi_\omega(x,y) \Psi(y,\omega)ds_y,\quad (x,\omega)\in
D^\mathrm{e}\times\R,
\end{eqnarray}
and introducing the corresponding frequency-domain
boundary integral operators
\begin{eqnarray}
&& V_{\Gamma}[\Psi](x,\omega)=\int_{\Gamma} \Phi_\omega(x,y) \Psi(y,\omega) ds_y,\quad (x,\omega)\in{\Gamma}\times\R,\label{SL_bnd_op}\\
&& K_{\Gamma}[\Psi](x,\omega)=\int_{\Gamma}
\pa_{\nu_y}\Phi_\omega(x,y) \Psi(y,\omega)ds_y,\quad
(x,\omega)\in{\Gamma}\times\R,
\end{eqnarray}
the frequency-domain solution $U$ can be expressed in the form
\be
\label{FDsol}
U(x,\omega)=\left({\mathcal D}_{\Gamma} -i\eta {\mathcal
    S}_{\Gamma}\right)[\Psi](x,\omega),\quad (x,\omega)\in
D^\mathrm{e}\times\R, \en where $\Psi$ is an unknown density and where
$\eta$ is a given parameter satisfying $\mathrm{Re}\,\eta\ne 0$. In
the limit as $x\rightarrow {\Gamma}$ the classical
frequency-domain combined field integral equation (CFIE) \be
\label{FDBIE}
\left(\frac{1}{2}I+K_{\Gamma}-i\eta V_{\Gamma} \right)[\Psi](x,\omega)= F(x,\omega),\quad (x,\omega)\in{\Gamma}\times\R
\en
results, which is uniquely solvable for all real wavenumbers $\kappa$.

The FTH wave equation solver~\cite{ABL20} for exterior domains $D^\mathrm{e}$
proceeds via effectively performing the following sequence of operations
\be
\label{process}
f(x,t)\xrightarrow{\mathbb{F}} F(x,\omega)
\xrightarrow{(\ref{FDBIE}),(\ref{FDsol})} U(x,\omega)
\xrightarrow{\mathbb{F}^{-1}} u(x,t).
\en
However, for incident fields $u^i$ and associated boundary data $f$ of
long duration, the Fourier transform $F(x,\omega)$ is generally a
highly oscillatory function of $\omega$, as a result of the rapidly
oscillating  factor $e^{i\omega t}$ that appears in the Fourier transform integrand
for large values of $t$---see e.g. \cite[Fig. 1]{ABL20}. Under such a
scenario a very fine frequency-discretization, requiring
$\mathcal{O}(T)$ frequency points, and, thus, a number
$\mathcal{O}(T)$ of evaluations of the frequency-domain boundary
integral equation solver, is required to obtain the time-domain
solution $U(x,t)$. This makes the overall algorithm unacceptably
expensive for long-time simulations. To overcome these difficulties,
the FTH algorithm~\cite{ABL20} relies on a certain ``windowing and
time-recentering'' procedure  proposed in \cite[Sec. 3.1]{ABL20},
that decomposes a scattering problem involving an incident time signal
of long duration into a sequence of problems with smooth incident
field of a limited duration, all of which can be solved in terms of a
fixed set of solutions of the corresponding frequency-domain problems
for arbitrarily large values of $T$.

For a given final time $T$, the windowing-and-recentering approach is
based on use of a smooth partition of unity (POU)
$\mathcal{P}=\{\sqcap_q(t)\ |\ q\in\mathcal{Q}\},
\mathcal{Q}=\{1,\cdots,Q\}$, where the functions $\sqcap_q$ satisfy
the relation $\sum_{q\in\mathcal{Q}}\sqcap_q(t)\ = 1$ for $t\in[0,T]$
and, for a certain sequence $s_q$ ($q\in\mathcal{Q}$), each
$\sqcap_q(t)$ is a non-negative, smooth windowing function of $t$,
supported in the interval $[s_q-H,s_q+H]$ of duration $2H$. Then
utilizing the POU $\mathcal{P}$, any smooth long-time
signal $g(t)$, $t\in[0,T]$, can be expressed in the form
\begin{equation}\label{eq:windowing}
  g(t)=\sum_{q\in\mathcal{Q}} g_q(t), \quad g_q(t)=g(t)\sqcap_q(t),
\end{equation}
where $g_q$ is compactly supported in $[s_q-H,s_q+H]$. The
corresponding Fourier transform is then given by
\begin{equation}
G(\omega)=\sum_{q\in\mathcal{Q}} G_q(\omega),\quad
G_q(\omega)=\int_{0}^{T} g_q(t)e^{i\omega t}dt= e^{i\omega
  s_q}G_{q,slow}(\omega),
\end{equation}
where, defining by
\be
\label{Fourierk}
\mathbb{F}_{q,slow}(g)(\omega):=\int_{-H}^H
  g(t+s_q)\sqcap_q(t)e^{i\omega t}dt,
\en
the $s_q$-centered slow Fourier-transform operator, we call $G_{q,slow}(\omega)= \mathbb{F}_{q,slow}(g)(\omega)$; note that, as suggested by the notation used, $G_{q,slow}$ is a slowly-oscillatory function of $\omega$.

For $q\in\mathcal{Q}$ we now call
$F_{q,slow}(x,\omega)= \mathbb{F}_{q,slow}(f)(x,\omega)$ the slow
Fourier-transform of the boundary data $f$, and we let $U_{q,slow}$ be the
solution of the Helmholtz equation problem (\ref{Helmholtzpro}) with
boundary data $F_{q,slow}$. The solution $U_{q,slow}$ can be expressed in the form
\be
\label{FDsolk}
U_{q,slow}(x,\omega)=\left({\mathcal D}_{\Gamma} -i\eta {\mathcal
    S}_{\Gamma}\right)[\Psi_{q,slow}](x,\omega),\quad (x,\omega)\in
D^\mathrm{e}\times\R, \en where $\Psi_{q,slow}$ is the
unique solution to the frequency-domain CFIE \be
\label{FDBIEk}
\left(\frac{1}{2}I+K_{\Gamma}-i\eta V_{\Gamma} \right)[\Psi_{q,slow}](x,\omega)= F_{q,slow}(x,\omega),\quad (x,\omega)\in{\Gamma}\times\R.
\en
It then easily follows  that
\be
\label{TDsol}
u(x,t)=\mathbb{F}^{-1}\left(\sum_{q\in\mathcal{Q}} e^{i\omega
s_q}U_{q,slow}\right)(x,t)= \sum_{q\in\mathcal{Q}}\mathbb{F}^{-1} \left(U_{q,slow}\right)(x,t-s_q),\quad (x,t)\in D^\mathrm{e}\times\R.
\en
On the basis of an appropriate high-frequency Fourier transform
algorithm to produce the Fourier transforms on the right-hand side
of~\eqref{TDsol}, see~\cite[Section 4]{ABL20}, the FTH solver proceeds
via the following sequence of operations \be
\label{processnew}
f(x,t)\xrightarrow{\mathbb{F}_{q,slow}} F_{q,slow}(x,\omega) \xrightarrow{(\ref{FDBIEk}),(\ref{FDsolk})} U_{q,slow}(x,\omega) \xrightarrow{(\ref{TDsol})} u(x,t).
\en

\subsection{Huygens-like domain-of-influence}
\label{sec:2.2}

Like the algorithm~\cite{BY23}, the multiple scattering algorithm
proposed in this paper depends in an essential manner on a variant of
the well known Huygens principle---in a form that is applicable to the
problem of scattering by obstacles and open arcs. A Huygens-like
domain-of-influence result, which follows by consideration of the
properties of the retarded Green function, may be stated as follows.
\begin{theorem}\cite[Proposition 3.6.2]{S16}
\label{Huygens1}
Let $u$ be the solution to the wave equation problem (\ref{waveeqn}).
If $f(x,t)=0$ for all $x\in\Gamma$ and $t\le t_0$, then for $x\in\R^2\backslash\Gamma$,
\ben
u(x,t)=0\quad \mbox{for all}\quad t\le t_0+c^{-1}\ \mathrm{dist}(x,\Gamma).
\enn
\end{theorem}
\begin{remark}\label{Huy_OA}
  A version of Theorem~\ref{Huygens1} can also be established for the
  solution to the wave equation problem
\begin{eqnarray}
\label{waveeqn-arc}
\begin{cases}
\pa_t^2v(x,t)-c^2\Delta v(x,t)=0, & (x,t)\in\R^2\backslash\mathcal{C}\times\R,\cr
v(x,t)= g(x,t),  & (x,t)\in\mathcal{C}\times\R,
\end{cases}
\end{eqnarray}
on the exterior of a bounded Lipschitz open-arc $\mathcal{C}$.
\end{remark}

Theorem~\ref{Huygens1} and its open-arc analog mentioned in
Remark~\ref{Huy_OA} do not provide sharp time estimates, as they
only ensure that the field propagates away from the complete boundary
(with speed equal to the speed of sound), but they do not account for
propagation along the scattering boundary.  In particular, for
incident fields illuminating a subset of the boundary of a scatterer,
these theoretical results do not establish that the field propagates
at the speed of sound along the scattering boundary. To determine the
relation between the delay in the arrival of the solution and
compactly supported boundary data, a generalized Huygens-like
domain-of-influence condition was introduced in~\cite{BY23}, which is
believed to be valid, which was established rigorously for a flat arc
and was demonstrated numerically for curved arcs and closed curves in
that reference, but for which a rigorous mathematical proof in the
case of general arcs or closed curves is unavailable as
yet. In the context of the wave equation problem~(\ref{waveeqn-arc})
outside the open arc $\mathcal{C}$, for example, the generalized
Huygens-like domain-of-influence condition may be stated as follows.
\begin{condition}
\label{huygens-conj}
We say that an open Lipschitz curve $\mathcal{C}$ with endpoints $e_1$ and $e_2$ satisfies the restricted Huygens condition iff for every Lipschitz curve $\mathcal{C}^{\mathrm{inc}}\subseteq\mathcal{C}$ satisfying $\mathrm{dist}(\mathcal{C}^{\mathrm{inc}},\{e_1,e_2\})> 0$, and for every $g$ defined in $\mathcal{C}$ such that \be
\label{DoI_ball_ass}
\{x\in\mathcal{C}\ |\ g(x,t)\ne 0 \}\subseteq \mathcal{C}^{\mathrm{inc}}\quad \mbox{for all}\quad t>0,
\en
the solution $v(x,t)$ to the wave equation problem (\ref{waveeqn-arc}) satisfies
\be
\label{DoI_ball_res1}
\{x\in\R^2 \ |\ v(x,t)\ne 0\} \subseteq \Lambda^s(t)\quad \mbox{for all}\quad t\le c^{-1}\mathrm{dist}(\mathcal{C}^{\mathrm{inc}},\{e_1,e_2\}),
\en
where
\ben
\Lambda^s(t)=\{x\in\R^2 \ |\ \mathrm{dist}(x,\mathcal{C}^{\mathrm{inc}})\leq ct\}.
\enn
\end{condition}

As in \cite{BY23}, we conjecture that Condition~\ref{huygens-conj}
holds for arbitrary open and closed Lipschitz curves $\mathcal{C}$ and
for all $t>0$ (even without the restriction
$t\le c^{-1}\mathrm{dist}(\mathcal{C}^{\mathrm{inc}},\{e_1,e_2\})$).
The proof is left for future work. Condition~\ref{huygens-conj}
expresses a well accepted principle in wave physics, namely, that
solutions of the wave equation propagate at the speed of sound, and
that the wave field vanishes identically before the arrival of a
wavefront. This boundary-propagation character provides a crucial
element in the proposed multiple scattering algorithm and, indeed,
throughout this paper Condition~\ref{huygens-conj} is assumed to be
valid.

\subsection{Solution regularity}
\label{sec:2.3}

The regularity of the solutions of the wave equation may be described
in terms of certain function spaces, as discussed in what follows. For
given $\sigma>0$ and $\alpha,p\in\R$, and for a given Hilbert space
$K$, the spatio-temporal Sobolev spaces $H_{\sigma,\alpha}^p(\R,K)$ of
functions with values in $K$ which vanish for $t\leq \alpha$ are
defined by~\cite{BH86,CHLM10}
\begin{equation}
H_{\sigma,\alpha}^p(\R,K):=\left\{ f\in\mathcal{L}'_{\sigma,\alpha}(K): \int_{-\infty+i\sigma}^{\infty+i\sigma} |s|^{2p}\|\mathcal{L}[f](s)\|_K^2ds<\infty \right\},
\end{equation}
together with the norm
\begin{equation}
\label{norm}
\|f\|_{H_{\sigma,\alpha}^p(\R,K)}:=\left(\int_{-\infty+i\sigma}^{\infty+i\sigma} |s|^{2p}\|\mathcal{L}[f](s)\|_K^2ds \right)^{1/2},
\end{equation}
where $\mathcal{L}[f]$ denotes the Fourier-Laplace transform of $f$ given by
\begin{equation}
\mathcal{L}[f](s):=\int_{-\infty}^\infty f(t)e^{is t}dt,\quad s\in \C_\sigma:=\{\omega\in\C: \mbox{Im}(s)>\sigma>0\},
\end{equation}
and where
$\mathcal{L}'_{\sigma,\alpha}(K):=\{\phi\in\mathcal{D}_\alpha'(K):
e^{-\sigma t}\phi\in \mathcal{S}_\alpha'(K)\}$ is defined in terms of
the sets $\mathcal{D}_\alpha'(K)$ and $\mathcal{S}_\alpha'(K)$ of
$K$-valued distributions and $K$-valued tempered distributions that
vanish for $t\le \alpha$, respectively. We also call \ben
H_{\alpha}^p(\beta,K)=\{f(x,t)|_{t\in(-\infty,\beta]}: f\in
H_{\sigma,\alpha}^p(\R,K)\} \enn the set of all restrictions of
functions $f\in H_{\sigma,\alpha}^p(\R,K)$ to the interval
$-\infty < t\leq \beta$. It can be checked that, as suggested by the notation used, the space $H_{\alpha}^p(\beta,K)$ does not depend on $\sigma$~\cite{BY23}.

Relevant well-posedness results for the problems (\ref{waveeqn}) and
\eqref{waveeqn-arc} are summarized in the following theorem.
\begin{theorem}\cite{CHLM10,YMC22}
\label{welltime}
Let $p\in\R$, $\sigma>0$ and $\alpha\ge 0$ be given.
\begin{itemize}
\item For $f\in H_{\sigma,\alpha}^p(\R,H^{1/2}(\Gamma))$, the wave equation problem (\ref{waveeqn}) admits a unique solution $u\in H_{\sigma,\alpha}^{p-3/2}(\R,H^1(D))$ and $u\in H_{\sigma,\alpha}^{p-3/2}(\R,H_{loc}^1(D^\mathrm{e}))$.
\item For $g\in H_{\sigma,\alpha}^p(\R,H^{1/2}(\mathcal{C}))$, the wave equation problem \eqref{waveeqn-arc} admits a unique solution $v\in H_{\sigma,\alpha}^{p-3}(\R,H_{loc}^1(\R^2\backslash\mathcal{C}))$.
\end{itemize}
\end{theorem}

\section{FTH-MS solver}
\label{sec:3}

Reference~\cite{BY23} proposed a novel ``ping-pong'' multiple
scattering algorithm for the solution of time-domain wave scattering
problems~\eqref{waveeqn} within a closed curve $\Gamma$. The ping-pong algorithm
proceeds by re-expressing the full time-evolution of the solution $u$
of a boundary-value problem for the wave equation as a problem of
multiple scattering between two overlapping subsets (open-arc patches)
$\Gamma_1$ and $\Gamma_2$ of the domain boundary $\Gamma$. By
appropriately accounting for multiple scattering between $\Gamma_1$
and $\Gamma_2$ it was shown in~\cite{BY23} that solutions to such
open-arc subproblems can be combined to produce a full solution for
the given interior domain problem for arbitrarily long time. The
present contribution proposes a generalized version of the previous
multiple scattering algorithm. In the new setting the full
time-evolution of the solution is obtained as the result of a set of
multiple scattering processes among arbitrary (user-prescribed)
numbers $N$ of overlapping arcs $\Gamma_j$ (called patches in what
follows) whose union equals the domain boundary:
$\Gamma =\cup_{j=1}^N \Gamma_j$. Unlike the method~\cite{BY23}, the
new approach utilizes a POU
$\{\chi_j\ :\ j=1,\dots,N\}$ of the domain boundary which is
subordinated to the covering of the boundary by the patches
$\Gamma_j$. (Without loss of generality, this setup assumes a POU in
which $\chi_j$ is the only POU element that vanishes outside
$\Gamma_j$.)  While POUs are not used in the algorithm~\cite{BY23},
the use of a POU is a key element in the method presented here---which
makes it possible to take advantage of the Huygens
Condition~\ref{huygens-conj} without relying on naturally occurring
vanishing boundary values near arc endpoints---a requirement which can
easily be ensured in the case of two-arc decompositions, but for which
a reasonable generalization to decompositions containing $N>2$ patches
which would additionally be applicable to 2D boundaries in 3D space
has not as yet been found. The ability to utilize arbitrary number of
patches provides a number of significant benefits, as it enables
(i)~The development of embarrassingly parallel implementations;
(ii)~Desirable large reductions in the frequency-domain patch-wise
problem sizes required; and, if iterative linear algebra solvers are
used for the solution of patch-wise frequency-domain problems,
(iii)~Significant reductions in the associated number of
iterations---as adequately multi-patch decompositions may be designed
to avoid wave-trapping patch-wise problems whose solution would
typically requires large numbers of iterations.

The basic principles of the new multiple-scattering methods are
presented in Section~\ref{sec:3.1} in the context of wave problems in
an interior domain $D$. Algorithmic details, once again in the
interior-domain context, are presented in Section~\ref{sec:3.2}, and
the theoretical and algorithmic modifications necessary for the
treatment of exterior problems are then presented in
Section~\ref{sec:3.3}.

%
%\subsubsection{Overlapping-arc geometry and windowing functions}
%\label{sec:3.4.1}

\subsection{Multiple-scattering algorithm in the interior domain $D$}
\label{sec:3.1}

The proposed multiple scattering algorithm for the wave equation
problem (\ref{waveeqn0}) relies on a decomposition of the boundary
$\Gamma$ into a number of overlapping patches. Thus, as
illustrated in Figure~\ref{MSmodel-int}, the scattering surface
$\Gamma$ is covered by a number $N\ge 2$ of overlapping patches
$\Gamma_j\subset\Gamma$, $j=1,\cdots,N$, in such a way that
$\Gamma=\cup_{j=1}^N \Gamma_j$ and $\Gamma_j$ has only a non-empty
overlap with $\Gamma_{j-1}$ and $\Gamma_{j+1}$ (where, for notational
convenience, we define $\Gamma_{N+1}=\Gamma_1$ and
$\Gamma_{0}=\Gamma_N$). In what follows we call
$\Gamma_{j,j+1}^{\mathrm{ov}}=\Gamma_j\cap\Gamma_{j+1}$ the overlap of
$\Gamma_j$ and $\Gamma_{j+1}$, and we call
$\Gamma_j^{\mathrm{tov}}:=\Gamma_j\backslash
(\Gamma_{j-1,j}^{\mathrm{ov}}\cup \Gamma_{j,j+1}^{\mathrm{ov}})$,
$j=1,\cdots,N$, the portion of $\Gamma_j$ that remains once the regions
of overlap with other patches are ``truncated'', see
Figure~\ref{MSmodel-patchwindow}(a).
\begin{figure}[htb]
\centering
\includegraphics[scale=0.3]{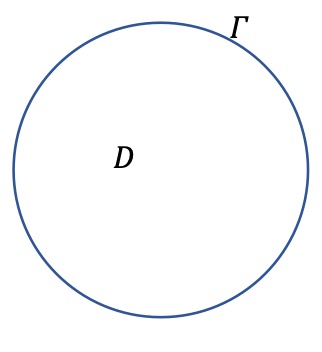} 
\includegraphics[scale=0.3]{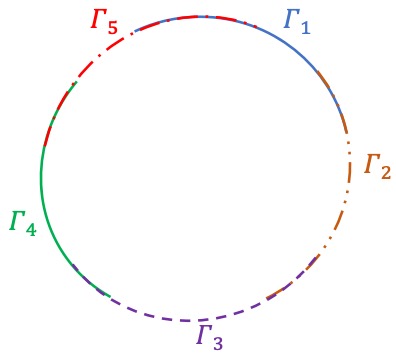} 
\caption{Illustration of the overlapping-patch decomposition of $\Gamma$ into a number $N=5$ of overlapping patches.}
\label{MSmodel-int}
\end{figure}

For each patch $\Gamma_j$ we introduce a smooth windowing function
$\chi_j:\Gamma\to [0,1]$ given by
\begin{equation}
  \label{eq:POU1}
\chi_j(x)=\begin{cases} 1, &
  x\in\Gamma_j^{\mathrm{tov}}, \cr
  \eta(A_{j,j+1}^{\mathrm{ov}}-d_{j,j+1}^{\mathrm{ov}}(x);c_0A_{j,j+1}^{\mathrm{ov}},c_1A_{j,j+1}^{\mathrm{ov}}),
  & x\in\Gamma_{j,j+1}^{\mathrm{ov}}, \cr
  1-\eta(A_{j-1,j}^{\mathrm{ov}}-d_{j-1,j}^{\mathrm{ov}}(x);c_0A_{j-1,j}^{\mathrm{ov}},c_1A_{j-1,j}^{\mathrm{ov}}),
  & x\in\Gamma_{j-1,j}^{\mathrm{ov}},\cr 0, & x\notin\Gamma_j,
\end{cases} \quad j=1,\cdots,N,
\end{equation}
where $0<c_0<1/2<c_1<1$;  where $\eta:\mathbb{R}\to [0,1]$ is a
smooth function equal to $1$ for $|t|<t_0$ and equal to $0$ for
$|t|>t_1$ ($t_1>t_0$), such as the one given in~\eqref{eta} below; and
where, for $j=1,\cdots,N$,
$A_{j,j+1}^{\mathrm{ov}}:={\rm diam}(\Gamma_{j,j+1}^{\mathrm{ov}})$
denotes the diameter of $\Gamma_{j,j+1}^{\mathrm{ov}}$ and
$d_{j,j+1}^{\mathrm{ov}}(x)$ denotes the diameter of the arc that
connects the point $x\in \Gamma_{j}$ and the endpoint $x_{j,j+1}$ of
$\Gamma_j$ that is located in the overlap region
$\Gamma_{j,j+1}^{\mathrm{ov}}$, see
Figure~\ref{MSmodel-patchwindow}(b).  Clearly
$\{\chi_j, j=1,\cdots,N\}$ satisfies the POU condition
\begin{equation}
  \label{eq:POU2}
  \sum_{j=1}^N\chi_j(x)=1\quad\mbox{for}\quad x\in\Gamma.
\end{equation}
The $C^\infty$ function
\be
\label{eta}
\eta(t;t_0,t_1)=\begin{cases}
1, & |t|\le t_0, \cr
e^{\frac{2e^{-1/s}}{s-1}}, & t_0<|t|<t_1, s=\frac{|t|-t_0}{t_1-t_0}, \cr
0, & |t|\ge t_1;
\end{cases}
\en
was utilized to produce all of the numerical results presented in
this paper---although, of course, other smooth window functions could
be used instead.

We additionally let
$\Gamma_j^0=\{x\in\Gamma_j| \chi_j(x)=0\}$ and we call
$\Gamma_{j}^{\mathrm{tr}}=\Gamma_j\backslash\Gamma_j^0$ the
corresponding truncated portion of $\Gamma_j$. The choice of the
parameters $0<c_0<1/2<c_1<1$ ensures that \ben
\mbox{dist}\{\Gamma_{j}^{\mathrm{tr}},\Gamma\backslash\Gamma_j\}>0,\quad j=1,\cdots,N; \enn the values $c_0=1/3$
and $c_1=2/3$ are used throughout this paper. Finally we define the
crucial parameter
\begin{equation}\label{dist}
\delta_{\mathrm{min}}=\min_{j=1,\cdots,N}\left\{\mbox{dist}\{\Gamma_{j}^{\mathrm{tr}},\Gamma\backslash\Gamma_j\}\right\} > 0.
\end{equation}

\begin{figure}[htb]
\centering
\begin{tabular}{cc}
\includegraphics[scale=0.3]{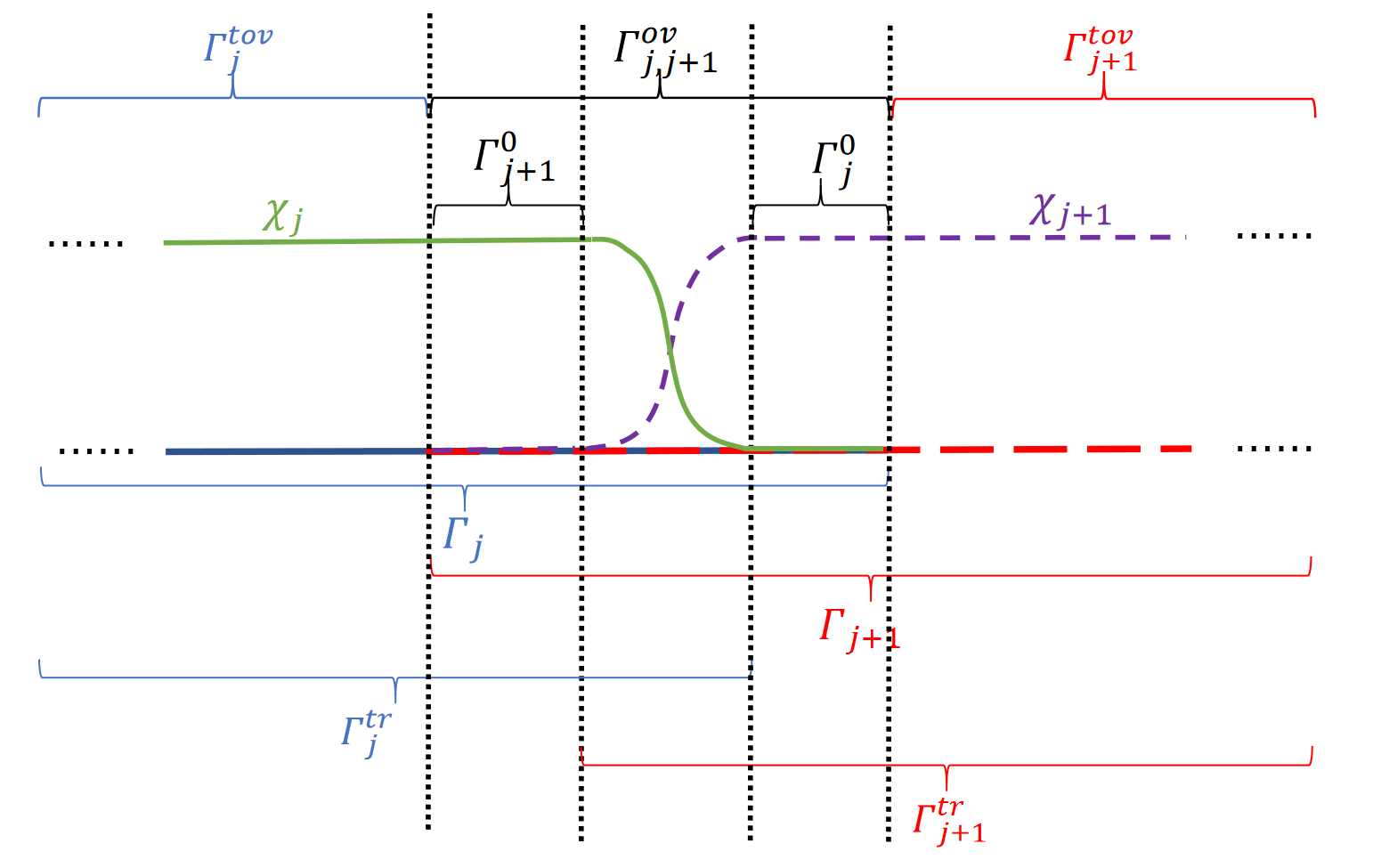} &
\includegraphics[scale=0.3]{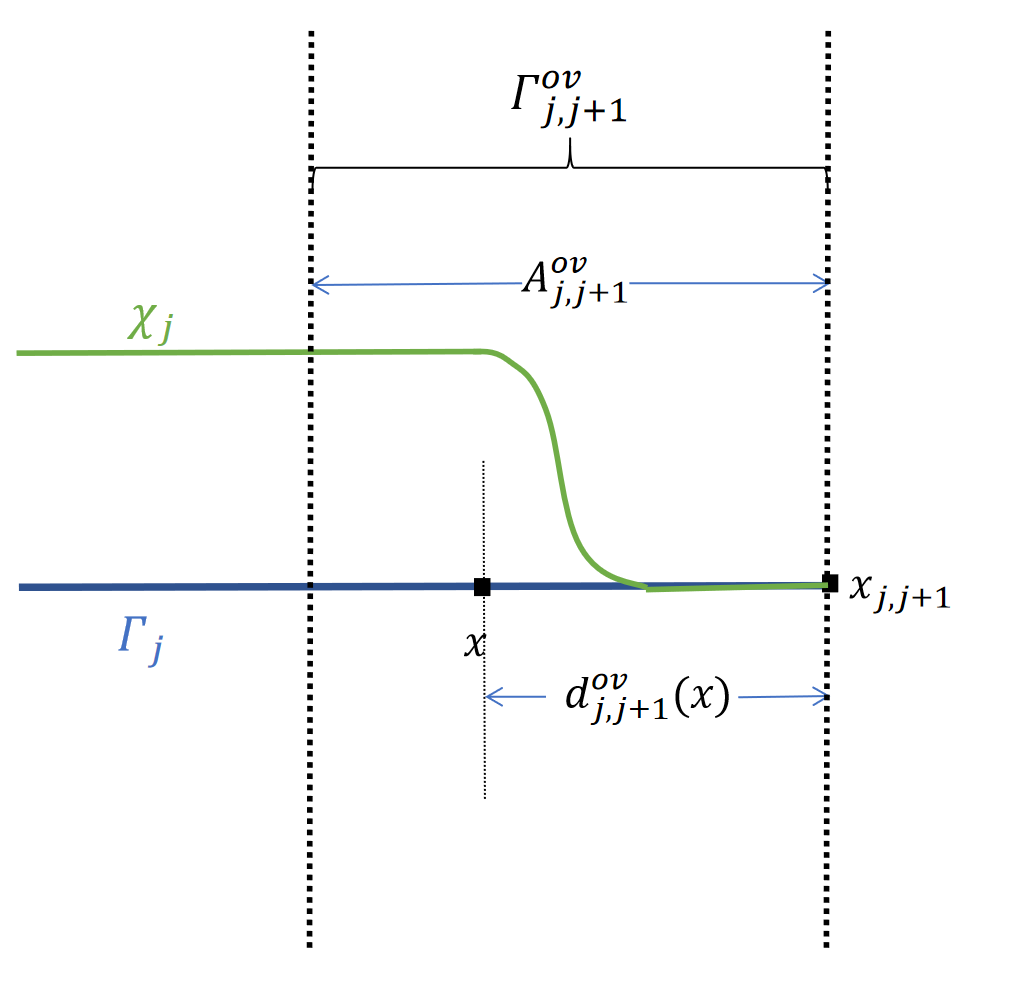} \\
(a) & (b)
\end{tabular}
\caption{Notations used for the boundary partitioning (a) and windowing functions (b).}
\label{MSmodel-patchwindow}
\end{figure}

%
%\subsubsection{Multiple scattering strategy}
%\label{sec:3.4.2}

Taking into account the finite propagation speed that characterizes
the solutions of the wave equation, or, more precisely, the
``Huygens-principle'' Condition~\ref{huygens-conj} introduced in
Section~\ref{sec:2.2}, we propose to produce the solution of the
original problem (\ref{waveeqn0}) in the interior domain $D$ as the
sum of wave-equation solutions multiply scattered by the arcs
$\Gamma_j, j=1,\cdots,N$. We thus consider the wave equation problems
exterior to each arc $\Gamma_j$, $j=1,\cdots,N$,
\begin{eqnarray}
\label{waveeqn-Gj2}
\begin{cases}
\pa_t^2u_j(x,t)-c^2\Delta u_j(x,t)=0, & (x,t)\in \Gamma_j^\mathrm{c}\times\R, \cr
u_j(x,t)= f_j(x,t),  & (x,t)\in\Gamma_j\times\R,
\end{cases}
\end{eqnarray}
where $\Gamma_j^\mathrm{c}=\R^2\backslash\Gamma_j$. Theorem~\ref{welltime}
tells us that, for given $p\in\R$, $\sigma>0$, $\alpha\ge0$ and for
$f_j\in H_{\sigma,\alpha}^p(\R,H^{1/2}(\Gamma_j))$, the wave equation
problem \eqref{waveeqn-Gj2} admits a unique solution
$u_j\in H_{\sigma,\alpha}^{p-3}(\R,H_{loc}^1(\Gamma_j^c))$. The
following lemma, whose proof is straightforward, expresses the
Condition~\ref{huygens-conj} in terms related to the curve
$\mathcal C = \Gamma_j$ in~(\ref{waveeqn-Gj2}) and the complete
boundary $\Gamma$---and thus presents that Huygens condition in
the form in which it will be used in the proof of
Theorem~\ref{equivalence-int}.
\begin{lemma}
\label{localpro}
Let $p\in\R$, $\sigma>0$ and $T_0>0$, and assume that, for all
$j=1,\cdots,N$,
\begin{itemize}
\item[(a)] $f_j\in H_{\sigma,T_0}^{p}(\R,H^{1/2}(\Gamma_j))$ satisfies
    \begin{equation}\label{ini_set}
    f_j(x,t)=0\quad\mbox{for}\quad (x,t)\in\Gamma_{j}^0\times \R;
    \end{equation}
\item[(b)] $\Gamma_j$ satisfies Condition~\ref{huygens-conj}.
\end{itemize}
Then, letting $t_0=\delta_{\mathrm{min}}/c>0$, and
calling $u_j\in H_{\sigma,T_0}^{p-3}(\R,H_{loc}^1(\Gamma_j^c))$ the unique solution of the wave equation problem \eqref{waveeqn-Gj2}, we have
\begin{equation}\label{ini_set_res}
u_j(x,t)=0\quad\mbox{for}\quad
  (x,t)\in (\Gamma\backslash\Gamma_j) \times (-\infty,T_0+t_0].
\end{equation}
\end{lemma}

To describe the proposed multiple-scattering scheme, we use the
notational conventions 
\be\label{sum_conv}
\sum_{k\backslash \{j\}}:=\sum_{k=1, k\ne
  j}^N\qquad\mbox{and} \qquad \sum_{k\backslash\{j,j+1\}}:=\sum_{k=1,
  k\ne j, k\ne (j+1)}^N, 
\en 
and, for $j=1,\cdots,N$, we inductively
define boundary-condition functions $g_{j,m}(x,t)$ ($m\geq 1$) and
associated wave-equation solutions $v_{j,m}(x,t)$ ($m\geq 1$) to the
problems (\ref{waveeqn-Gj2}) as detailed in what follows.

\begin{definition}\label{inductive_int}
  For $m\in\mathbb{N}$, we define $v_{j,m}(x,t)$ to equal
  the solution $u_j(x,t)$ of the wave equation
  problem~(\ref{waveeqn-Gj2}) with boundary-condition function
  $f_j = g_{j,m}$---so that $v_{j,m}$ satisfies the boundary
  conditions
\begin{equation}\label{bdry_int}
v_{j,m}(x,t) = g_{j,m}(x,t)\quad \mbox{for}\quad (x,t)\in \Gamma_j\times \R,\quad m \in\mathbb{N},
\end{equation}
for all $j=1,\cdots,N$, where
$g_{j,m}(x,t):\Gamma_{j}\times \R\to\mathbb{C}$ denotes the causal
functions defined inductively via the relations
\begin{equation}
  \label{criterion_int1}
  g_{j,1}(x,t)= -\chi_j(x)u^i(x,t)\quad (x,t)\in \Gamma_j\times \R,
\end{equation}
and, for $m\ge 1$,
\begin{equation}
  \label{criterion_int2}
  g_{j,m+1}(x,t)=-\chi_j(x) \sum_{k\backslash \{j\}}\widetilde v_{k,j,m}(x,t), \quad (x,t)\in  \Gamma_j\times \R.
\end{equation}
Here, for $j=1,\dots,N$, $(x,t)\in\Gamma_{j}\times \R$,
$k=1,\cdots,N, k\ne j$, and $m\ge 1$, $\widetilde v_{k,j,m}(x,t)$ denotes
the function
\begin{equation}\label{vtilde}
  \widetilde v_{k,j,m}(x,t)=
\begin{cases}
v_{k,m}(x,t), & x\in\Gamma_j, k\notin\{j-1, j+1\}  ,\cr
v_{k,m}(x,t), & x\in\Gamma_j\backslash \Gamma_k, k=j-1\; \mathrm{and}\; k=j+1  , \cr
0, & x\in\Gamma_{j-1,j}^{\mathrm{ov}}, k=j-1, \cr
0, & x\in\Gamma_{j,j+1}^{\mathrm{ov}}, k=j+1.
\end{cases}
\end{equation}
\end{definition}

\begin{remark}
\label{remarkg}
  The definition of the windowing function $\chi_j$ ($j=1,\cdots,N$)
  ensures that
  \be
\label{vanishing}
g_{j,m}(x,t)=0\quad \mbox{for}\quad (x,t)\in \Gamma_{j}^0\times
\R,\quad j=1,\cdots,N,\;m \in\mathbb{N}, \en which is a crucial
property in the proposed multiple scattering algorithm---in that it
ensures that the solution of the problem of scattering by the open arc
$\Gamma_j$ with boundary conditions given by the incident field
$g_{j,m}$ coincides, at least for a certain time interval, with the
solution of the problem of scattering by the complete surface
$\Gamma$, with incident field equal to $g_{j,m}$ for $x\in \Gamma$.
Moreover, the function $g_{j,m}(x,t)$ is continuous and vanishes at
the endpoints of $\Gamma_j^{\mathrm{tov}}$, despite the seemingly
discontinuous nature of the closely related definition~\eqref{vtilde}
at those points. The claimed continuity of $g_{j,m}(x,t)$ holds since,
for $m \geq 2$ and $k = j\pm 1$, $g_{k,m}(x,t)$ is continuous and
vanishes at the endpoints of $\Gamma_j^{\mathrm{tov}}$ for all
$t \in \mathbb{R}$ (as it follows from the
definitions~\eqref{criterion_int1}--\eqref{criterion_int2} with
$j=k$), which implies that $v_{k,m}$, and, thus,
$\widetilde{v}_{k,j,m}(x,t)$, tend to zero, for all
$t \in \mathbb{R}$, as $x$ approaches the endpoints of
$\Gamma_j^{\mathrm{tov}}$.
\end{remark}

We now consider the $M\times N$-th order multiple-scattering sum
\begin{equation}\label{mult-scatt-int}
  u_M(x,t):=\sum_{m=1}^M\sum_{j=1}^N
  v_{j,m}(x,t),
\end{equation}
which includes contributions from all $M\times N$ ``ping-pong''
scattering iterates $v_{j,m}(x,t)$ with $j=1,\cdots,N$,
$m=1,\dots, M$. Theorem~\ref{equivalence-int} below establishes that,
throughout $D$, the multiple-scattering sum $u_M$ coincides with the
exact solution $u$ of the original wave equation problem
(\ref{waveeqn0}) for $t\leq T(M)$, where, for $M\in\mathbb{N}$ we
define
\begin{equation}
  \label{eq:TM}
  T(M) = M \delta_{\mathrm{min}}/c.
\end{equation}
The proof of Theorem~\ref{equivalence-int} relies on certain
relations, established in the following lemma, concerning the
solutions $v_{j,m}$ and the boundary value functions $g_{j,m}$
introduced in Definition~\ref{inductive_int}.
\begin{lemma}\label{localrels} With reference to
  equation~\eqref{eq:TM}, for each $M\in\mathbb{N}$ and each
  integer $j$, $1\leq j\leq N$, the relations
\begin{align}
&  v_{j,M}(x,t)+\sum_{m=1}^{M-1}\sum_{k=1}^N v_{k,m}(x,t)+u^i(x,t)=0, \quad (x,t)\in\Gamma_j^{\mathrm{tov}}\times \R,\label{A1}\\
& v_{j,M}(x,t)+v_{j\pm 1,M}(x,t)+\sum_{m=1}^{M-1}\sum_{k=1}^N v_{k,m}(x,t)+u^i(x,t)=0, \quad (x,t)\in\left(\Gamma_j\cap \Gamma_{j\pm 1}\right)\times \R,  \quad \mbox{and}\label{A2}\\
& g_{j,M}(x,t)=0,\quad (x,t)\in\Gamma_{j}\times (-\infty,T(M-1)]\phantom{\sum_{k=1}^N}\label{A3}
\end{align}
are satisfied.
\end{lemma}
\begin{proof} We establish~\eqref{A1} through~\eqref{A3} by induction
  in $M$. In the case $M=1$, \eqref{A1} follows directly
  from~\eqref{bdry_int} and~\eqref{criterion_int1};
  equation~\eqref{A2} results once again from~\eqref{bdry_int}
  and~\eqref{criterion_int1} together with~\eqref{eq:POU1}
  and~\eqref{eq:POU2}; and~\eqref{A3} follows
  from~\eqref{criterion_int1} since
  $u^i\in H_{\sigma,0}^{p}(\R,H^{1/2}(\pa D))$ and, thus, vanishes for
  $t\leq 0$ (Section~\ref{sec:2.3}). Having established~\eqref{A1}
  through~\eqref{A3} for $M=1$, the corresponding inductive steps are
  considered next.

  Let $L\in \mathbb{N}$ and assume~\eqref{A1} holds for $M= L$ and
  $1\leq j\leq N$. Then, the validity of~\eqref{A1} with $M=L+1$ is
  equivalent to the validity of the relation
  \begin{equation}
    \label{eq:A1diff1}
    v_{j,L+1}(x,t)+\sum_{k\backslash \{j\}} v_{k,L}(x,t) = 0,\quad (x,t)\in\Gamma_j^{\mathrm{tov}}\times \R,
  \end{equation}
  which results as the differnce between the $M = L+1$ and the $M=L$
  versions of~\eqref{A1}.  But, in view of~\eqref{bdry_int}
  and~\eqref{vtilde} and since $\chi_j(x) = 1$ for
  $x\in \Gamma_j^{\mathrm{tov}}$,~\eqref{eq:A1diff1} coincides
  with~\eqref{criterion_int2} for
  $(x,t)\in \Gamma_j^{\mathrm{tov}}\times\mathbb{R}$ and therefore
  holds true. The proof of~\eqref{A1} for $M\in\mathbb{N}$ is thus
  complete.

  We next consider the inductive step for~\eqref{A2} with a ``$+$''
  sign; the relation with a ``$-$'' sign follows similarly. Let us
  thus assume the $+$-sign version of~\eqref{A2} holds for $M = L$ for
  a given $L\in\mathbb{N}$ and for $1\leq j\leq N$. The validity of
  the $+$-sign version of~\eqref{A2} for $M = L+1$ is equivalent to
  the validity of the relation
  \begin{equation}
    \label{eq:A1diff}
    v_{j,L+1}+v_{j+1,L+1}+\sum_{k\backslash\{j,j+1\}}v_{k,L}(x,t)
    = 0,\quad (x,t)\in \Gamma_{j,j+1}^{\mathrm{ov}}\times \R,
 \end{equation}
 that results as the difference between the $M = L+1$ and the $M=L$
 $+$-sign relations. But, in view of~\eqref{bdry_int}
 and~\eqref{criterion_int2} and using the notations~\eqref{sum_conv}, the relation~\eqref{eq:A1diff} may be
 re-expressed in the form
 $-\chi_j(x) \sum_{k\backslash \{j\}}\widetilde v_{k,j,L}(x,t)
 -\chi_{j+1}(x) \sum_{k\backslash \{j+1\}}\widetilde
 v_{k,j+1,L}(x,t)+\sum_{k\backslash \{j,j+1\}}v_{k,L}(x,t) =
 0$. Using all four expressions in~\eqref{vtilde} for the case
 $x\in\Gamma_{j,j+1}^{\mathrm{ov}}$
 considered presently, in turn, the latter relation is equivalent to
 $-\chi_j(x) \sum_{k\backslash \{j,j+1\}} v_{k,j,L}(x,t)
 -\chi_{j+1}(x) \sum_{k\backslash \{j,j+1\}}
 v_{k,j+1,L}(x,t)+\sum_{k\backslash \{j,j+1\}}v_{k,L}(x,t) =
 0$---which clearly holds true in view of~ \eqref{eq:POU1}
 and~\eqref{eq:POU2}. The proof of~\eqref{A2} is thus complete.

 We finally consider the inductive step for~\eqref{A3}: assuming this
 relation holds for $M = L\in\mathbb{N}$ and for $1\leq j\leq N$, we
 seek to establish its $M=L+1$ version which, in view
 of~\eqref{criterion_int2}, may equivalently be expressed in the form
\begin{equation}
  \label{eq:A3Lp1}
  -\chi_j(x) \sum_{k\backslash \{j\}}\widetilde v_{k,j,L
  }(x,t) =0  \quad\mbox{on}\quad\Gamma_j\times (-\infty,T(L)].
\end{equation}
To establish~\eqref{eq:A3Lp1}, using~\eqref{bdry_int},~\eqref{vtilde}
and the inductive hypothesis we show that each term
$\widetilde v_{k,j,L}(x,t)$ in the sum~\eqref{eq:A3Lp1} vanishes for
$(x,t)\in \Gamma_j\times (-\infty,T(L)]$, and we do this for each
one of the two cases $k=j\pm 1$ and $k\ne j\pm1$. In the first case,
and restricting attention to the subcase $k=j+1$ (since the subcase
$k=j-1$ follows similarly) we see from~\eqref{vtilde} that
$\widetilde v_{j+1,j,L}(x,t)=0$ for
$(x,t)\in\Gamma_{j,j+1}^{\mathrm{ov}}\times\R$, and
$\widetilde v_{j+1,j,L}(x,t) = v_{j+1,L}(x,t)$ for
$(x,t)\in(\Gamma_j\setminus\Gamma_{j+1})\times\R$, so that, for the
$k=j+1$ subcase considered presently, it only remains to show that
$v_{j+1,L}(x,t) = 0$ for
$(x,t)\in(\Gamma_j\setminus\Gamma_{j+1})\times(-\infty,T(L)]$. But,
in view of~\eqref{bdry_int} we have
$v_{j+1,L}(x,t) = g_{j+1,L}(x,t)$ for
$(x,t)\in\Gamma_{j+1}\times\R$, and, thus (i)~$v_{j+1,L}(x,t)=0$ for
$(x,t)\in\Gamma_{j+1}\times(-\infty,T(L-1)]$ (since $g_{j+1,L}(x,t) = 0$
for $(x,t)\in\Gamma_{j+1}\times(-\infty,T(L-1)]$ by the $M=L$ inductive
hypothesis); and (ii)~ $v_{j+1,L}(x,t)=0$ for
$(x,t)\in\Gamma_{j+1}^{0}\times\R$ (since, on account of the
$\chi_{j+1}(x)$ factor in~\eqref{criterion_int2} we have
$g_{j+1,L}(x,t) = 0$ for
$(x,t)\in\Gamma_{j+1}^{0}\times\R$). The conclusions in the
points~(i) and~(ii) just considered enable us to invoke
Lemma~\ref{localpro} with $j$ replaced by $j+1$, and, thus, to
ascertain that $v_{j+1,L}(x,t)=0$ for
$(x,t)\in(\Gamma_j\setminus\Gamma_{j+1})\times(-\infty,T(L)]$---which
concludes the proof in the case $k=j\pm 1$. In the case $k\ne j\pm1$,
finally, by~\eqref{vtilde} we have
$\widetilde v_{k,j,L}(x,t)=v_{k,L}(x,t)$, and the proof follows by
invoking Lemma~\ref{localpro} once again, on the basis of the
inductive hypothesis (which tells us that
$v_{k,L}(x,t) = g_{k,L}(x,t) = 0$ for
$(x,t)\in\Gamma_k\times(-\infty,T(L-1)]$), together with the fact that
$\mathrm{dist}(\Gamma_k,\Gamma_j)\geq \delta_\mathrm{min}$, to
conclude that $\widetilde v_{k,j,L}(x,t) = v_{k,L}(x,t)=0$ for
$(x,t)\in\Gamma_j\times(-\infty,T(L)]$, as needed. The proof of the
lemma is now complete.
\end{proof}

\begin{theorem}
\label{equivalence-int}
Let $M\in\mathbb{N}$, $p\in\R$ and $\sigma>0$. Then, with reference to equation~\eqref{sum_conv}, for given
$u^i\in H_{\sigma,0}^{p}(\R,H^{1/2}(\Gamma))$ we have
$u_M \in H_{0}^{p-3/2}(T(M),H^1(D))$, and the solution $u$ of the
problem~\eqref{waveeqn} with $E=D$ (interior problem) is given by
\begin{equation}\label{thm_eqn}
  u(x,t) =u_M(x,t)\quad \mbox{for}\quad  (x,t)\in D\times (-\infty,T(M)].
\end{equation}
\end{theorem}
\begin{proof}
  The fact that $u_M \in H_{0}^{p-3/2}(T(M),H^1(D))$ follows as in the
  proof of \cite[Theorem 2.8]{BY23}. In view of the uniqueness of
  solution to equation~\eqref{waveeqn-Gj2}, to
  establish~\eqref{thm_eqn} it suffices to show that for all
  $M\in\mathbb{N}$, the function $u_M$ satisfies the boundary
  conditions
\begin{equation}\label{bound-int0}
u_M(x,t)+u^i(x,t)=0\quad\mbox{for}\quad (x,t)\in\Gamma\times (-\infty,T(M)],
\end{equation}
or, equivalently, since $\Gamma = \cup_{j=1}^N \Gamma_j$, that
\begin{equation}\label{bound-int0j}
u_M(x,t)+u^i(x,t)=0\;\mbox{ for }\; (x,t)\in\Gamma_j\times (-\infty,T(M)]\;\mbox{ and }\; 1\leq j\leq N.
\end{equation}
Noting that
\begin{equation}\label{mult-scatt-int_Lp1}
u_M(x,t)=  \sum_{k=1}^N v_{k,M}(x,t) + \sum_{m=1}^{M-1}\sum_{k=1}^N v_{k,m}(x,t),
\end{equation}
in what follows we establish~\eqref{bound-int0j} for a given $j$
($1\leq j\leq N$) by separately considering~\eqref{mult-scatt-int_Lp1}
in the cases $(x,t)\in\Gamma^\mathrm{tov}_j\times (-\infty,T(M)]$ and
$(x,t)\in\left(\Gamma_j\cap \Gamma_{j\pm 1}\right)\times (-\infty,T(M)]$.

In order to verify~\eqref{bound-int0j} for a given $j$ and for
$(x,t)\in\Gamma^\mathrm{tov}_j\times (-\infty,T(M)]$ we first note
that since we have $g_{k,M}(x,t) = 0$ for
$(x,t)\in \Gamma_k\times (-\infty,T(M-1)$ (in view of~\eqref{A3})
and $g_{k,M}(x,t) = 0$ for $(x,t)\in \Gamma_k^0\times\mathbb{R}$
(in view of~\eqref{vanishing}), and in consideration
of~\eqref{bdry_int}, Lemma~\ref{localpro} tells us that, for
$k\ne j$, $v_{k,M}(x,t) = 0$ for
$(x,t)\in (\Gamma_j\backslash\Gamma_k) \times (-\infty,T(M)]$. Applying
this relation for $k=j\pm 1$ we conclude that
\begin{equation}\label{bound-int1.31}
  v_{j\pm 1,M}(x,t) = 0\; \mbox{ for }\;
  (x,t) \in\Gamma_j^{\mathrm{tov}}\times (-\infty, T(M)],
\end{equation}
while
applying it for other values of $k$ we obtain
\begin{equation}\label{bound-int1.41}
  v_{k,M}(x,t) = 0\quad \mbox{for}\quad
  k\not\in\{j-1,j,j+1\}\quad \mbox{and}\quad
  (x,t) \in\Gamma_j\times (-\infty, T(M)].
\end{equation}
In view of~\eqref{bound-int1.31} and~\eqref{bound-int1.41},
equation~\eqref{mult-scatt-int_Lp1} reduces to
\begin{equation}\label{mult-scatt-int_tov}
u_M(x,t)=   v_{j,M}(x,t) + \sum_{m=1}^{M-1}\sum_{k=1}^N v_{k,m}(x,t) \;\mbox{ for }\; (x,t) \in\Gamma_j^{\mathrm{tov}}\times (-\infty, T(M)],
\end{equation}
and it reduces to
\begin{equation}\label{mult-scatt-int_ov}
u_M(x,t)=   v_{j,M}(x,t) + v_{j\pm 1,M}(x,t) + \sum_{m=1}^{M-1}\sum_{k=1}^N v_{k,m}(x,t) \;\mbox{ for }\; (x,t)\in \Gamma_j\cap \Gamma_{j\pm 1}\times (-\infty,
T(M)].
\end{equation}
The desired relation~\eqref{bound-int0} now follows directly
from~\eqref{mult-scatt-int_tov} and~\eqref{mult-scatt-int_ov} in view
of~\eqref{A1} and~\eqref{A2}, and the proof is thus complete.
\end{proof}

\subsection{FTH-MS solver: interior problem}
\label{sec:3.2}

The procedure embodied in Definition~\ref{inductive_int},
Lemma~\ref{localrels} and Theorem~\ref{equivalence-int} can readily be
implemented, on the basis of the methods described in
Section~\ref{sec:2.1} and \cite{BY23}, as an effective methodology for
the solution of the problem~\eqref{waveeqn} with $E=D$ (interior
problem); the resulting interior FTH-MS algorithm is described in what
follows. Using the notations introduced in Section~\ref{sec:2.1}, for
$m\in\N$, $j=1,\cdots,N$ and $q\in\mathcal{Q}$ we call
$G_{j,m,q}(x,\omega)= \mathbb{F}_{q,slow}(g_{j,m})(x,\omega)$ the slow
Fourier-transform of the boundary data $g_{j,m}$, and we let
$V_{j,m,q}$ denote the solution to the Helmholtz equation problem in
$\Gamma_j^\mathrm{c}$ with Dirichlet boundary data $G_{j,m,q}$ on $\Gamma_j$. The solution $V_{j,m,q}$ is expressed in the form
\be
\label{FDsolk-interior}
V_{j,m,q}(x,\omega)={\mathcal S}_{\Gamma_j}
[\Psi_{j,m,q}](x,\omega),\quad (x,\omega)\in \Gamma_j^c\times\R, \en
(equation~\eqref{SL_rep} but with
$(x,\omega)\in \Gamma_j^\mathrm{c}\times\R$), where,
using~\eqref{SL_bnd_op},
$\Psi_{j,m,q} \in \widetilde{H}^{-1/2}(\Gamma_j)$ denotes the unique
solution of the frequency-domain BIE \be
\label{FDBIEk-interior}
V_{\Gamma_j} [\Psi_{j,m,q}](x,\omega)= G_{j,m,q}(x,\omega),\quad
(x,\omega)\in{\Gamma_j}\times \R, \en see
  e.g.~\cite{SW84}. The solution $v_{j,m}$ is then given by \be
\label{TDsol-interior}
v_{j,m}(x,t)= \mathbb{F}^{-1}\left(\sum_{q\in\mathcal{Q}} e^{i\omega
    s_q}V_{j,m,q}\right)(x,t)= \sum_{q\in\mathcal{Q}}\mathbb{F}^{-1}
\left(V_{j,m,q}\right)(x,t-s_q),\quad (x,t)\in \Gamma_j^\mathrm{c}\times\R.
\en Using these results, notations and conventions, the proposed
FTH-MS method for the interior problem~\eqref{waveeqn} is presented in
what follows (Algorithm~\ref{alg-interior}). Here and throughout this paper, the symbol sequence ``$\mathrel{+}=$'' denotes the binary assignment operator commonly used in C++ and other
programming languages, which adds the value of the right operand to
the value of the left operand, and then assigns the result back to the left operand.

\begin{algorithm}
  \caption{FTH-MS solver for the problem (\ref{waveeqn}) in interior
    domains}
\label{alg-interior}
\begin{algorithmic}[1]
\STATE Initialize $u_M(x,t)=0, (x,t)\in E\times [0,s_Q]$.
\STATE Do $m=1,2,\cdots,M$
\STATE \ \ \ \ Do $j=1,2,\cdots,N$
\STATE \ \ \ \ \ \ \ \ Evaluate the boundary data $g_{j,m}(x,t), (x,t)\in\Gamma_j\times[0,s_Q+H]$ via relations (\ref{criterion_int1})-(\ref{criterion_int2}).
\STATE \ \ \ \ \ \ \ \ Set $v_{j,m}(x,t)=0, (x,t)\in \ov{E}\times [0,s_Q+H]$.
\STATE \ \ \ \ \ \ \ \ Do $q=1,2,\cdots,Q$
\STATE \ \ \ \ \ \ \ \ \ \ \ \ \label{seven} Evaluate the boundary data $G_{j,m,q}(x,\omega)= \mathbb{F}_{q,slow}(g_{j,m})(x,\omega), (x,\omega)\in\Gamma_j\times \R$.
\STATE \ \ \ \ \ \ \ \ \ \ \ \ Solve the integral equations (\ref{FDBIEk-interior}) with solution $\Psi_{j,m,q}$.
\STATE \ \ \ \ \ \ \ \ \ \ \ \ \label{nine} Compute $V_{j,m,q}(x,\omega), (x,\omega)\in \Gamma_j^c\times \R$ using equation (\ref{FDsolk-interior}).
\STATE \ \ \ \ \ \ \ \ \ \ \ \ \label{ten} Compute $v_{j,m}(x,t)\mathrel{+}= \mathbb{F}^{-1}(V_{j,m,q})(x,t-s_q), (x,t)\in \ov{E}\times [0,s_Q+H]$.
\STATE \ \ \ \ \ \ \ \ End Do
\STATE \ \ \ \ \ \ \ \ Compute $u_M(x,t)\mathrel{+}=v_{j,m}(x,t), (x,t)\in E\times [0,s_Q]$ per equation~\eqref{mult-scatt-int} (see Remark~\ref{infty_and_trapping_i}).
\STATE \ \ \ \ End Do
\STATE End Do
\end{algorithmic}
\end{algorithm}

\noindent

\begin{remark}\label{infty_and_trapping}
  The patches $\Gamma_j$ should be specifically chosen to be
  nontrapping, since, as detailed in
  Remarks~\ref{infty_and_trapping_i}
  and~\ref{infty_and_trapping_ii}, the use of such type of patches
  significantly impacts the efficiency of the algorithm.
\end{remark}
\begin{remark}\label{infty_and_trapping_i}
  Although Algorithm~\ref{alg-interior}, per its line 12, evaluates
  $u_M(x,t)$ as a sum of an increasing number of terms as $t$ and $M$
  grow, the selection of non-trapping patches $\Gamma_j$ ensures the
  rapid decay of the solutions $v_{j,m}(x,t)$ as $t$ increases---at
  least for the cases considered in this paper, where the 2D incident
  signal consists entirely of frequency components at nonzero
  frequencies. In such cases, integration by parts arguments similar
  to the 3D treatment~\cite{AB22} can be used to establish fast
  solution decay. As a result of this rapid decay, the total number of
  active terms $v_{j,m}(x,t)$ in the algorithm's line 12 (i.e., those
  exceeding a given error tolerance $\varepsilon^\mathrm{tol}>0$) remains asymptotically
  close to a fixed integer over time. Consequently, the computational
  cost per unit time remains uniformly bounded as time grows.  In
  cases for which the incident signal includes frequency content at
  zero frequency the frequency-asymptotic methods presented
  in~\cite{ABL25} could be applied in the present context to ensure,
  once again, uniformly bounded computing times as time grows.
  \end{remark}  
\begin{remark}\label{infty_and_trapping_ii}
  The use of non-trapping patches $\Gamma_j$ offers a significant
  computational advantage when iterative linear algebra algorithms are
  employed to obtain the required frequency-domain
  solutions. Specifically, the number of iterations required for an
  iterative algorithm to solve the integral
  equations~(\ref{FDBIEk-interior}) on non-trapping patches is
  considerably lower than that required for solving the corresponding
  equations on trapping curves; see Figure~\ref{Example5.3}.
\end{remark}

\subsection{FTH-MS solver: exterior problem}
\label{sec:3.3}

The methods introduced in Sections~\ref{sec:3.1} and~\ref{sec:3.2},
including Theorem~\ref{equivalence-int}, can be extended with minimal
modifications to apply to the wave equation~\eqref{waveeqn} in an
exterior domain $E$, including scenarios where the boundary
$\Gamma = \partial E$ consists of a combination of open and closed
curves. Using this strategy, the implementation of the exterior FTH-MS
solver involves decomposing each connected component of the curve
boundary $\Gamma$ into a union of overlapping open arcs, as described
in Section~\ref{sec:3.2}. This approach is particularly beneficial
when the curves induce trapping, as discussed in
Remark~\ref{infty_and_trapping}. As part of this treatment, the
necessary frequency-domain solutions are obtained through
equations~\eqref{FDsolk-interior} and~\eqref{FDBIEk-interior}.

If preferable, however, boundary-connected components that are closed
curves can be treated as a single patch $\Gamma_j$, and in that case
the necessary frequency-domain solutions $V_{j,m,q}$
(cf. Section~\ref{sec:3.2}) are obtained in the form \be
\label{FDsolk-closed}
V_{j,m,q}(x,\omega)=\left({\mathcal D}_{\Gamma_j} -i\eta {\mathcal
    S}_{\Gamma_j}\right)[\widehat \Psi_{j,m,q}](x,\omega),\quad (x,\omega)\in
D_j^c\times\R. \en Here, defining $G_{j,m,q}$ as in
Section~\ref{sec:3.2}, $\Psi_{j,m,q} \in H^{1/2}(\Gamma_j)$ denotes
the unique solution of the frequency-domain CFIE \be
\label{FDBIEk-closed}
\left(\frac{1}{2}I+K_{\Gamma_j}-i\eta V_{\Gamma_j} \right)[\widehat
\Psi_{j,m,q}](x,\omega)= G_{j,m,q}(x,\omega),\quad
(x,\omega)\in{\Gamma_j}\times\R; \en cf. Section~\ref{sec:2.1}. The
FTH-MS algorithm then proceeds as in the interior domain case
considered in the previous section. In all numerical examples in this
paper that include closed curves as connected boundary components,
these curves are treated using the CFIE equation, as outlined above.

\section{Computational implementation}
\label{sec:4}

The numerical implementations of both the interior-domain
Algorithm~\ref{alg-interior} and the exterior-domain counterpart
outlined in Section~\ref{sec:3.3} require the evaluation of Fourier
transforms, layer potentials, and boundary integral operators. The
next two subsections summarize the strategies employed for these
computations.

\subsection{High-frequency Fourier transform methods}\label{sec:4.1}
The Fourier transformation procedures we use coincide with those
proposed in~\cite{ABL20}. They rely on the fact that if $g(t)$ and
$G(\omega)$ are smooth functions that decay superalgebraically---i.e.,
faster than any negative power of $t$ and $\omega$, respectively, as
these variables tend to $\pm\infty$---then the error in the truncated
Fourier-transform approximation
\begin{equation}
  \label{eq:ft_appr} g(t) \approx \frac{1}{2\pi} \int_{-W}^{W}
G(\omega)e^{-i\omega t} \, d\omega
\end{equation} of the exact transform~(\ref{backwardFT}), as well as
the error in the corresponding forward transform of $g(t)$ truncated
to the interval $[-T,T]$, both decay superalgebraically as $W \to
\infty$ and $T \to \infty$, respectively. Consequently, the needed
Fourier transform~\eqref{eq:ft_appr} and related time-truncated
inverse Fourier transform can be evaluated by means of the
high-frequency quadrature methods~\cite[Secs. 3 and 4]{BY23},
respectively, with  errors that decay superalgebraically fast as both the
integration intervals and the number of quadrature points are
increased.

In the 2D context, in particular, owing to the fact that the 2D
frequency-domain solutions generally contains integrable
$\mathcal{O}(\log|\omega|)$ singularities, the evaluation of the
corresponding Fourier transforms with high-order accuracy requires
special considerations. In detail, in the 2D case the approximate
inverse Fourier transform integrals (\ref{eq:ft_appr}) are expressed
in the form
\begin{eqnarray}
\label{backwardFT1} g(t)\approx
\frac{1}{2\pi}\left(\int_{-W}^{-w_c}+\int_{-w_c}^{w_c}+\int_{w_c}^{W}\right)
G(\omega)e^{-i\omega t}d\omega,
\end{eqnarray} which may be further decomposed into 1)~Integrals of the form
\begin{eqnarray}
\label{int1} I_a^b[G](t)=\int_{a}^{b} G(\omega)e^{-i\omega t}d\omega,
\end{eqnarray} where $G$ is a smooth non-periodic function; and, 2)
Half-interval integrals
\begin{eqnarray}
\label{int2} I_0^{w_c}[G](t)=\int_{0}^{w_c} G(\omega)e^{-i\omega
t}d\omega\quad\mbox{and}\quad I_{-w_c}^0[G](t)=\int_{-w_c}^0
G(\omega)e^{-i\omega t}d\omega,
\end{eqnarray} where $G(\omega)$ contains a logarithmic singularity at
$\omega=0$. The regular integral $I_a^b[G](t)$ can be evaluated by
expressing the non-periodic function $G(\omega), \omega\in [w_c,W]$ as
a Fourier-continuation trigonometric polynomial~\cite{BL10,AB16,ABL20} and
subsequent explicit term-wise integration. To treat the integral
$I_0^{w_c}[G](t)$, which involves a logarithmic singularity at
$\omega=0$, at fixed cost for arbitrarily large times $t$, a modified
``Filon-Clenshaw-Curtis'' high-order quadrature approach was developed
in~\cite{ABL20} by introducing a graded set \ben
\left\{\mu_j=w_c\left(\frac{j}{P}\right)^p, j=1,\cdots,P\right\}, \enn
of points in the interval $(0,w_c)$ and associated integration
sub-intervals $(\mu_j,\mu_{j+1}), j=1,\cdots,P$, on each one of which
the integral $I_{\mu_j}^{\mu_{j+1}}[G](t)$ is produced by means of the
Clenshaw-Curtis quadrature rule. The Fourier transform \ben
\mathbb{F}_{q,slow}(g)(\omega):=\int_{-H}^H
g(t+s_q)\sqcap_q(t)e^{i\omega t}dt \enn for smooth boundary-values
function $g$, finally, can be obtained in a manner analogous to that
used for the evaluation of $I_{w_c}^W[G](t)$ except that, on account
of the use of compactly supported windowing function $\sqcap_q$, the
integrand may be viewed as a periodic function and, thus, a regular
Fourier expansion of the function $g(t+s_q)\sqcap_q(t), t\in[-H,H]$
may be used, instead of a the Fourier continuation used for
$I_{w_c}^W[G](t)$.

\subsection{Single- and double-layer potentials on closed curves and
open arcs
  \label{sec:4.2}} The numerical implementation of the FTH-MS
algorithm also requires evaluating and inverting specific integral
operators, as well as computing certain layer potentials, both of
which are associated with the curves $\Gamma_j$, for
$j=1,\dots,N$. These curves may be either closed or open; in the
following, we discuss the corresponding algorithms for each case.  For
each closed curve $\Gamma_j$, the relevant integral operators and
layer potentials include $\mathcal{S}_{\Gamma_j}$ and
$\mathcal{D}_{\Gamma_j}$, as well as $V_{\Gamma_j}$ and
$K_{\Gamma_j}$, which are given in equations (\ref{FDsolk-closed}) and
(\ref{FDBIEk-closed}), respectively.  For each open arc $\Gamma_j$, in
turn, the required integral operators and layer potentials include
$\mathcal{S}_{\Gamma_j}$ and $V_{\Gamma_j}$, as given in equations
(\ref{FDsolk-interior}) and (\ref{FDBIEk-interior}), respectively. All
of these operators can be expressed in the form
\begin{eqnarray}
\label{singular2} \mathcal{H}_j(x,\omega)&=&\int_{\Gamma_j}
\Psi_\omega(x,y)\psi_j(y,\omega)ds_y, \quad
x\in\Gamma_j\;\;\mbox{or}\;\; x\in\mathbb{R}^2\backslash\Gamma_j,
\end{eqnarray} where $\Psi_\omega$ equals
either $\Phi_\omega$ or $\pa_{\nu_y}\Phi_\omega$
and where $\psi_j(y,\omega)$ represents the integral density function.

In the implementations used in this paper, both closed-curve and
open-arc integrals of the form~\eqref{singular2} are evaluated
numerically using spectral representations of the density function
$\psi_j$. For closed curves, these integrals are computed with high
accuracy using the well-established Nystr\"om method~\cite{CK98}. The
treatment of open arcs, however, requires more detailed considerations;
for clarity, those details are provided separately in
Section~\ref{FTH-op-arc}. Additionally, as part of the FTH-MS method,
the singular integrals~\eqref{singular2} must be evaluated at a
sufficiently large number of frequency discretization points $\omega$
within the interval $[-W,W]$. To reduce the computational cost, an
appropriate decomposition of the singular kernel is employed, as
described in \cite[Section 3.2]{BY23}, and we refer to that source for
further details in this regard.

\subsection{Frequency-domain FTH-MS open-arc problems: additional
  singularities\label{FTH-op-arc}} The proposed FTH-MS algorithm
evaluates the integrals in~\eqref{singular2} using open-arc techniques
introduced in~\cite{BL12} and~\cite{BY21}, but in a modified form that
accounts for the unique singularity structure in the FTH-MS open-arc
density functions $\psi_j$. While the densities $\psi_j$ exhibit endpoint singularities analogous to those
observed in classical open-arc problems, they also feature additional
``induced'' singularities due to the overlaps between $\Gamma_j$ and
$\Gamma_{j\pm 1}$, as described in what follows.  (As discussed in
Section~\ref{sec:1}, the previous contribution~\cite{BY23} employs a
``smooth extension along normals'' technique that, while effectively
eliminating all induced singularities through geometric modifications,
is somewhat cumbersome, is not applicable to open-arc scatterers, and
does not generalize effectively to the three-dimensional context.)
Before considering the special FTH-MS open-arc singularities, however,
we provide some background on density singularities and numerical
methods in the context of related problems of scattering by a single
smooth open arc $\widetilde{\Gamma}$ under illumination by a smooth
incident field, with Dirichlet boundary conditions, as considered
in~\cite{BL12} and~\cite{BY21}.

\begin{figure}[htb] \centering
\includegraphics[scale=0.2]{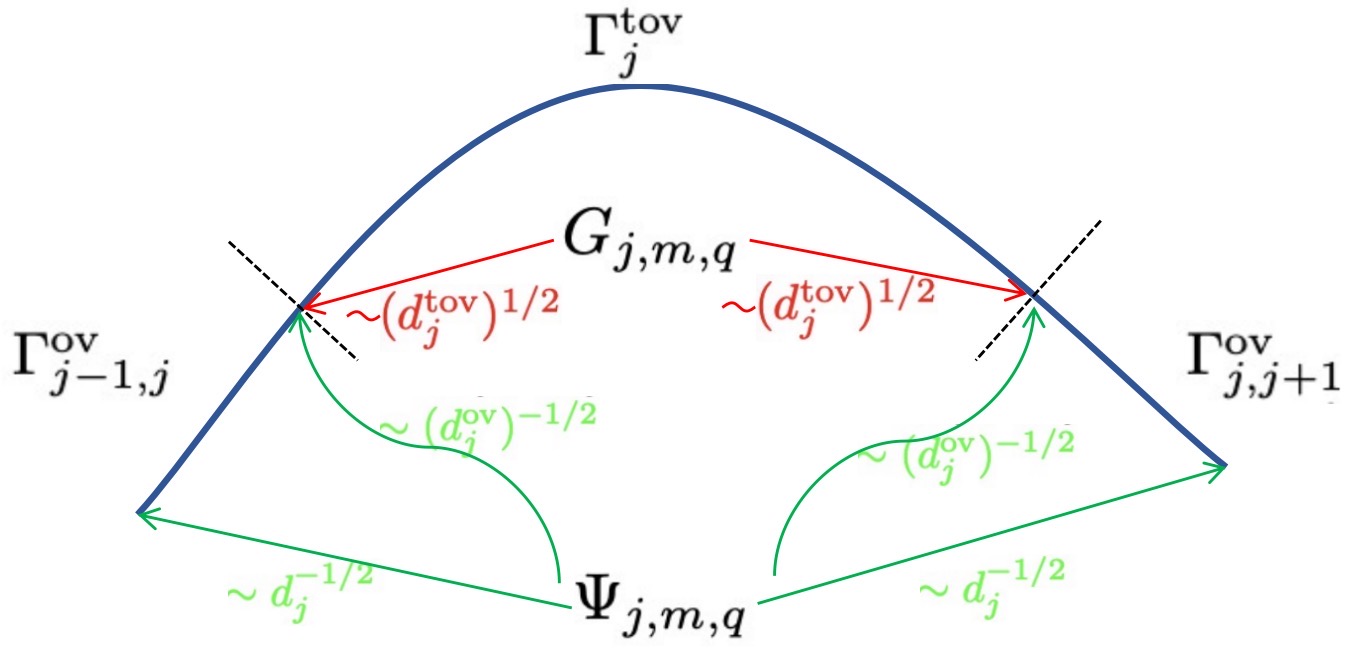}
\caption{Singularities of the functions $G_{j,m,q}$ and $\Psi_{j,m,q}$
on the arc $\Gamma_j
=\Gamma_{j-1,j}^\mathrm{ov}\cup\Gamma_j^\mathrm{tov}\cup\Gamma_{j,j+1}^\mathrm{ov}$.\label{FTH-MS_sings}}
\label{Singularity}
\end{figure}

Considering such ``smooth illumination'' Dirichlet problems, the
approach~\cite{BL12} addresses the singular nature of the associated
density $\widetilde{\psi}$ at the endpoints of $\widetilde{\Gamma}$ by
expressing the density in the form
\begin{equation}
\label{eq:singularity} \widetilde{\psi} = \frac{\alpha}{w},
\end{equation} where $\alpha$ is a smooth density function, and where
$w$ is a fixed function that, near each endpoint of
$\widetilde{\Gamma}$, equals the product of a smooth non-vanishing
function and the square-root of the distance to the endpoint.  This
formulation captures the singularity in the density $\widetilde{\psi}$
while ensuring that the density function $\alpha$ remains smooth near
the endpoints.  Reference~\cite{BL12} accurately computes the
necessary single-layer integrals by means of a spectral discretization
of the density $\alpha$ over the entire arc. In this paper we leverage
the 2D Chebyshev-based rectangular-polar spectral discretization
method~\cite{BY21,BY23} instead. This method not only explicitly
accounts for the endpoint singularity~\eqref{eq:singularity}, but it
also enables the arc to be partitioned into a prescribed number of
non-overlapping ``integration patches.'' This partitioning greatly
enhances the generality and flexibility of the overall methodology,
providing, in particular, significant advantages in the context of the
FTH-MS algorithm.

As mentioned earlier and as detailed in what follows, additional
singularities emerge in the open-arc problems associated with the
FTH-MS method. For clarity we present the description within the
framework of the interior problem considered in Section~\ref{sec:3.2},
and using the notation introduced in that section. (The corresponding concepts and
algorithms required for the  exterior problem presented
in Section~\ref{sec:3.3} are entirely analogous.)  Thus, using the
notations in Section~\ref{sec:3.2} we note that, as a result of the
overlapping-arc decompositions used as part of the FTH-MS framework,
the endpoint singularities of the arc $\Gamma_j$ introduce
corresponding singularities in the arcs $\Gamma_{j\pm 1}$ since, by
construction, the endpoints of $\Gamma_j$ are contained within
$\Gamma_{j-1}\cup \Gamma_{j+1}$.

In detail, we note that, per Remark~\ref{remarkg}, the boundary
functions $g_{j,m}$ with $m>1$, and, hence, the functions $G_{j,m,q}$
with $m>1$, are continuous and vanish at the endpoints of
$\Gamma_j^{\mathrm{tov}}$. However, it can be checked that the
singularities of the densities $\psi_{j\pm 1}$ at the endpoints of
$\Gamma_{j\pm 1}$ induce corresponding $(d_{j}^{\mathrm{tov}})^{1/2}$
singularities in the boundary-data functions $G_{j,m,q}$ near the
endpoints of $\Gamma_{j}^{\mathrm{tov}}$, where $d_{j}^{\mathrm{tov}}$
denotes the distance from the point $x\in\Gamma_j^{\mathrm{tov}}$ to
the endpoints of $\Gamma_j^{\mathrm{tov}}$ (cf.
Figure~\ref{FTH-MS_sings}). Preliminary analysis indicates that, as
expected, the edge singularity of the boundary data function
$G_{j,m,q}$ induces additional singularities of the form
$(d_j^{\mathrm{ov}})^{-1/2}$  in the solution $\Psi_{j,m,q}$
of~(\ref{FDBIEk-interior}) as $x\in \Gamma_{j-1,j}^{\mathrm{ov}}\cup
\Gamma_{j,j+1}^{\mathrm{ov}}$ approaches the endpoints of
$\Gamma_j^{\mathrm{tov}}$, where $d_j^{\mathrm{ov}}$ denotes the
distance from $x\in \Gamma_{j-1,j}^{\mathrm{ov}}\cup
\Gamma_{j,j+1}^{\mathrm{ov}}$ to the endpoints of
$\Gamma_j^{\mathrm{tov}}$. A sequence of singularities is thus set up,
one density singularity resulting in a field boundary singularity
which in turn induces a subsequent density singularity, resulting in a
sequence of singularities given by various powers of the distances to
the endpoints of $\Gamma_j^{\mathrm{tov}}$. 

These theoretical considerations together with related numerical
experiments presented below suggest that all endpoint singularities
encountered can be expressed in the form
\begin{equation}
\label{eq:sqrt_sing} \frac{f\left(d_p^{1/2}(x)\right)}{d_p^{1/2}(x)}
\quad \text{where} \quad f\quad \text{is an \emph{infinitely
differentiable function}}
\end{equation} in a neighborhood of each endpoint $p$, and where
$d_p(x)$ denotes the distance from $x$ to $p$. A singularity of this
precise form is rigorously established in~\cite{akhm_br} for the
density at a Dirichlet-Neumann junction point $p$; however, a proof
for the present setting remains an open question and is deferred to
future work.

A numerical method designed to resolve the singular behavior just described is
presented below in this section. The approach relies on a class of
cosine-type changes of variables (CoV) introduced in what
follows---which were originally put forth in~\cite[Equation
(4.12)]{BY20} to eliminate integrand singularities of a related but
different form. These CoVs require that the non-overlapping partition
$\Gamma_{jk}$ of each arc $\Gamma_j$ ($k=1,2,\dots,P_j$), as defined
by the rectangular-polar discretization method~\cite{BY21,BY23}, be
constructed in such a way that none of the endpoints of $\Gamma_j$ or
its truncated-overlap counterpart $\Gamma_j^{\mathrm{tov}}$ lie in the
interior of any of the rectangular-polar patches
$\Gamma_{jk}$. Further the CoVs defined below assume that the patch
$\Gamma_{jk}$ is parameterized by a vector function
$x = x(\theta) \in \Gamma_{jk}$ with $\theta \in [-1, 1]$.  In view of
the discussion above and the results of this and related numerical
tests, these changes variables are made an integral part of the
proposed numerical implementation, and, in particular, are used in all
of the numerical examples presented in Section~\ref{sec:5}.

Depending on the location of the singularities of the function
$\Psi_{j,m,q}$ within each rectangular-polar patch $\Gamma_{jk}$ that
contains such singularities, the cosine-type change of variable
$\theta=\theta(s)$ ($-1\leq s\leq 1$) for the patch $\Gamma_{jk}$ used
is given by \be
\label{cov} \theta(s)=
\begin{cases} s, & \mbox{If no singularities exist on
$\Gamma_{jk}$,}\cr \cos(\frac{\pi}{2}(1-s)), & \mbox{If singularities
exist at both endpoints of $\Gamma_{jk}$,} \cr
1-2\cos(\frac{\pi}{4}(1+s)), & \mbox{If a singularity exists only at
the $\theta=-1$ endpoint}\cr 2\cos(\frac{\pi}{4}(1-s))-1, & \mbox{If
singularity exists only at the $\theta=+1$ endpoint}.
\end{cases} \en
Using these change of variables we introduce the new
unknown
\begin{equation}\label{eq:psi-tilde}
  \widetilde \Psi_{j,m,q}(s) = \Psi_{j,m,q}(\theta(s))|\theta'(s)|
\end{equation}
instead of $\Psi_{j,m,q}$ in~\eqref{FDBIEk-interior} which, in
particular, incorporates the Jacobian $|\theta'(s)|$.

Relying on this CoV and associated density unknown
$\widetilde\Psi_{j,m,q}(s)$, the numerical illustration that follows
offers compelling evidence in support of the preceding conjecture
concerning the form~\eqref{eq:sqrt_sing} of the singularities
observed. This numerical example proceeds by first displaying
singularities in the numerically-obtained integral densities
associated with a simple multi-arc scattering problem, and by then
showing that the density $\widetilde \Psi_{j,m,q}(s)$ that results
upon use of the CoV~\eqref{cov} is a smooth function,
at least up to the orders of differentiation considered in the
example. This clearly illustrates the assumed
form~\eqref{eq:sqrt_sing} of the singularities of the density
$\Psi_{j,m,q}$, since, taking into account the integration Jacobian
$|\theta'(s)|$, the inverse transformations to~\eqref{cov} on a smooth
density $\widetilde\Psi_{j,m,q}(s)$ precisely induce the singularities
claimed on the density $\Psi_{j,m,q}$. The numerical method is
completed by utilizing the rectangular-polar method in the variable
$s$.

For our numerical illustration we consider a scattering problem in the
interior of the unit disc treated on the basis of $N=3$ boundary
open-arc patches $\Gamma_j$ ($j=1,2,3$), and we utilize the
associated integral equation (\ref{FDBIEk-interior}) on the arc
$\Gamma_{j}=\{x=(\cos\theta,\sin\theta)^\top\in\R^2:
\theta\in(-\frac{\pi}{6},\frac{5\pi}{6})\}$ with $j=1$. (This is a
particular case, with parameter values $\omega_0=2$ and $R=1$, of the
configuration called ``Geometry 1'' in Section~\ref{sec:5}.) In
particular for the $j=1$ patch we have,
$\Gamma_{j-1,j}^{\mathrm{ov}}=\{x=(\cos\theta,\sin\theta)^\top\in\R^2:
\theta\in(-\frac{\pi}{6},\frac{\pi}{6})\}$ and
$\Gamma_{j,j+1}^{\mathrm{ov}}=\{x=(\cos\theta,\sin\theta)^\top\in\R^2:
\theta\in(\frac{\pi}{2},\frac{5\pi}{6})\}$. The boundary data used
equals the function $G_{j,m,q}$ that results from use of the multiple
scattering algorithm (Algorithm~\ref{alg-interior}) for $m=5$ and
$q=Q=1$ which, in particular, contains a number of induced
singularities. Figure~\ref{fig:density}(a) displays the solution
$\Psi_{j,m,q}$ of equation~\eqref{FDBIEk-interior} along the portion
$\{-\frac{\pi}{6}\leq \theta\leq \frac{\pi}{3}\}$ of the $j=1$ patch
$\Gamma_j$; the singularities of $\Psi_{j,m,q}$ near
$\theta=-\frac{\pi}{6}$ and $\theta=\frac{\pi}{6}$ are clearly
observed. (The density solution behaves similarly in the portion
$\{\frac{\pi}{3}\leq \theta\leq \frac{5\pi}{6}|\}$.) Figure
\ref{fig:density}(b), in turn, demonstrates that the solution
$\widetilde \Psi_{j,m,q}$ is smooth, at least up to the orders of
differentiation considered, and thus provides the desired numerical
evidence of conjectured singularity structure of the solution $\Psi_{j,m,q}$.

\begin{figure} \centering
\begin{tabular}{cc}
\includegraphics[width=0.45\linewidth]{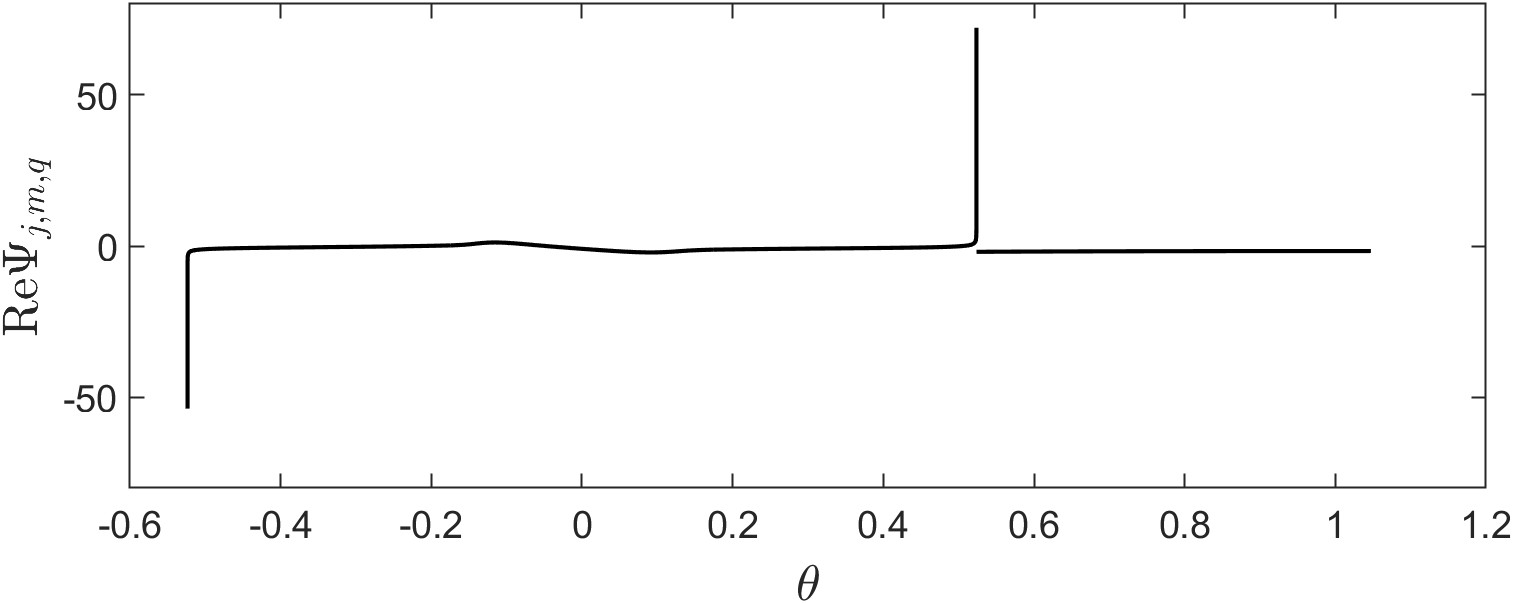} &
\includegraphics[width=0.45\linewidth]{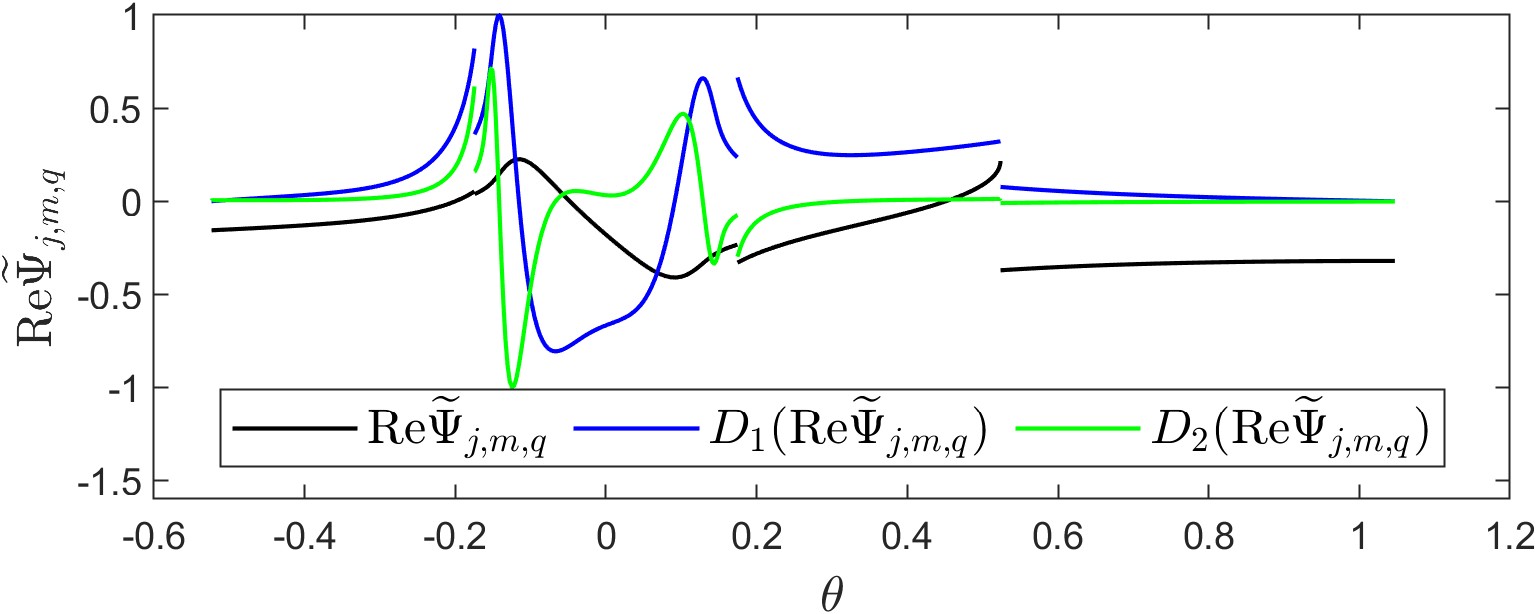}\\ (a)
$\Psi_{j,m,q}$ & (b) $\widetilde \Psi_{j,m,q}$
\end{tabular}
\caption{Comparison of the solutions $\Psi_{j,m,q}$ and $\widetilde
\Psi_{j,m,q}$ (normalized values).}
\label{fig:density}
\end{figure}

\section{Numerical results}
\label{sec:5}

This section presents a variety of numerical results, produced by
means of the FTH-MS integral equation solver introduced in this paper, which illustrate the efficiency
and accuracy of the proposed methods. With reference to
Section~\ref{sec:4.1}, we denote by
$\mathcal{F}=\{\omega_1,\cdots,\omega_J\}$ a set of frequencies used
for the Fourier transformation process, which includes equi-spaced
grids in the frequency intervals $[-W,-w_c]$ and $[w_c,W]$, as well as
Clenshaw-Curtis mesh points in the intervals
$(-\mu_{j+1},-\mu_j)$ and $(\mu_j,\mu_{j+1})$, $j=1,\cdots,P$. For
the necessary time-domain discretization, in turn, we use the mesh
$\mathcal{T}=\{t_n=n\Delta t\}_{n=1}^{N_T}$ of the time interval
$[0,s_Q+H]$ (cf. eq.~\eqref{eq:windowing}), where $\Delta t=(s_Q+H)/N_T$. In regard to the time windowing
strategy, in turn, we set $H=10$, $s_q=\frac{3}{2}(q-1)H$, $q\in\mathcal{Q}$, and
\begin{equation*} \sqcap_q(t)=\sqcap(t-s_q),\quad
\sqcap(s)=\begin{cases} \eta(s/H;1/2,1), & -H/2\le s\le H, \cr 1-
\eta(s/H+3/2;1/2,1), & -H< s<-H/2, \cr 0, & |s|\ge H,
\end{cases}
\end{equation*} 
where the function $\eta$ is defined in~\eqref{eta}. For simplicity, with exception of the long-time propagation problem illustrated in Figure~\ref{Example1.4}, all of the numerical examples presented in this paper utilize the value $Q=1$ (single incident-field window). Sufficiently fine frequency-independent meshes are
used on the boundaries $\Gamma_j$, $j=1,\cdots,N$ for all frequencies
considered, where, in accordance with Section~\ref{sec:4.2}, the
open-arc integral equations \eqref{FDBIEk-interior} are numerically
solved using the Chebyshev-based rectangular-polar solver described in
Section~\ref{FTH-op-arc}, while the closed-curve integral equations
\eqref{FDBIEk-closed} are solved using the Nystr\"om
method~\cite{CK98}.  The numerical errors $\varepsilon$ presented in
this section were calculated in accordance with the expression
\begin{equation}
  \label{eq:errors} \varepsilon= \max_{t\in[0,T]}|u^{s}_{\rm
num}-u^{s}_{\rm ref}|
\end{equation} where, with exception of the test cases considered in connection with 
Geometries 3 and 6, for which the closed-form solutions are known, the
reference solutions $u^{s}_{\rm ref}$ were obtained as numerical
solutions produced by means of sufficiently fine discretizations. All
of the numerical results presented in this paper were obtained on the
basis of \CC numerical implementations of the algorithms described in
the previous sections, parallelized using OpenMP, and run on a 8-core
Lenovo Laptop with an AMD Core processor R7-8845H.

Four different types of incident fields are considered in this section
(all of which are solutions of the wave equation), namely
\begin{enumerate}
\item A Gaussian-modulated point source $u_1^i(x,t)$ equal to the
Fourier transform of the function
\begin{eqnarray}
\label{pointsource}
U_1^i(x,\omega)=\frac{5i}{2}H_0^{(1)}(\omega|x-z_0|)
e^{-\frac{(\omega-\omega_0)^2}{\sigma^2}}e^{i\omega \tau_0}
\end{eqnarray} with respect to $\omega$, with $\sigma=2$ and $\tau_0=4$;
\item A Gaussian-modulated plane-wave incidence $u_2^i(x,t)$ equal to
the Fourier transform of the function
\begin{eqnarray}
\label{planewave2} U_2^i(x,\omega)=e^{i\omega x\cdot d^\textrm{inc}}
e^{-\frac{(\omega-\omega_0)^2}{\sigma^2}}e^{i\omega \tau_0}
\end{eqnarray} with respect to $\omega$, with $\sigma=\sqrt{2}$,
$\tau_0=6$, and with incident direction vector
$d^\textrm{inc}=(\cos\theta^\textrm{inc},
\sin\theta^\textrm{inc})^\top$;
  \item A pulse plane-wave incident field
\begin{eqnarray}
\label{planewave}
u_3^i(x,t)=-\sin(4\ell(x,t))e^{-1.6(\ell(x,t)-3)^2},\quad
\ell(x,t)=t-t_\textrm{lag}-x\cdot d^\textrm{inc}
\end{eqnarray} along the incident direction
$d^\textrm{inc}=(\cos\theta^\textrm{inc},
\sin\theta^\textrm{inc})^\top$ with $t_{lag}=2$; and
\item The multi-pulse incident field
  \begin{equation}
    \label{eq:multi} u_4^i=\sum_{j=0}^{4}u_3^i(x,t-10j).
  \end{equation}
\end{enumerate}

\noindent On the basis of these incident fields, a variety of
numerical examples are presented in what follows for five different
scattering geometries, as described below.  \vskip 0.2cm
\noindent {\bf Geometry 1.}  The domain $D$ in this case equals the
disc of radius $R>0$ centered at the origin. Equi-sized overlapping
patches are selected such that $\delta_{\mathrm{min}}\approx 0.35$ for
$R=1$, $N=3$ and $\delta_{\mathrm{min}}\approx 0.42$ for $R=2$,
$N=6$, where $\delta_{\mathrm{min}}$ is defined in~\eqref{dist}. Interior wave equation problems for two different time-domain
incident fields, namely, the plane wave $u^i_2(x,t)$ and the
multi-pulse incident field $u_4^i$, are considered for this
geometry. For the plane wave $u^i_2(x,t)$ incidence, the fixed
numerical frequency interval $\omega\in[5,25]$ is used for the necessary numerical Fourier transformation, and the parameter
values $\theta^{inc}=0$, $\omega_0=15$, $J=501$ and $\Delta t=0.01$
were set. For the multi-pulse incident field $u_4^i$, in turn, for which a
long time propagation problem is considered in what follows, the
parameter values $W=20$, $J=872$, $\Delta t=0.1$, and $M=40$ were
used. For both the plane-wave and multi-pulse incidence test cases the exact
solutions are given by $u(x,t)=-u_2^i(x,t)$ and $u(x,t)=-u_4^i(x,t)$
for $x\in D$, respectively. Solution values and numerical errors for
various values of $M$ at $x=(0.5,0)^\top$ as a function of $t$
resulting for the problems with plane wave $u^i_2(x,t)$ incidence are
displayed in Figures~\ref{Example1.2}--\ref{Example1.3}: clearly, high
accuracy and rapid convergence with respect to $M$ are observed. The
numerical and exact solutions at $x=(0.5,0)^\top$ for the case of
$u^i_4(x,t)$ incidence are displayed in Figure~\ref{Example1.4}; the errors in the numerical solution displayed are uniformly less than $10^{-6}$.

% \begin{figure}[htbp] % \centering
% \begin{tabular}{cc} %
% \includegraphics[scale=0.15]{Example1-psN3-sol} & %
% \includegraphics[scale=0.15]{Example1-psN3-error} \\ % (a) & (b)
% \end{tabular}
% \caption{Scattered field and errors obtained for the problem
% considered in Example 1 with $R=1$ and $N=3$. (a) Real and imaginary
% parts of scattered field at $x=(0.5,0)^\top$ resulting from the
% incident field $u_1^i$. (b) Numerical errors as functions of time $t$
% for various values of $M$.}
% \label{Example1.1}
% \end{figure}

\begin{figure}[htbp] \centering
\begin{tabular}{cc} \includegraphics[scale=0.08]{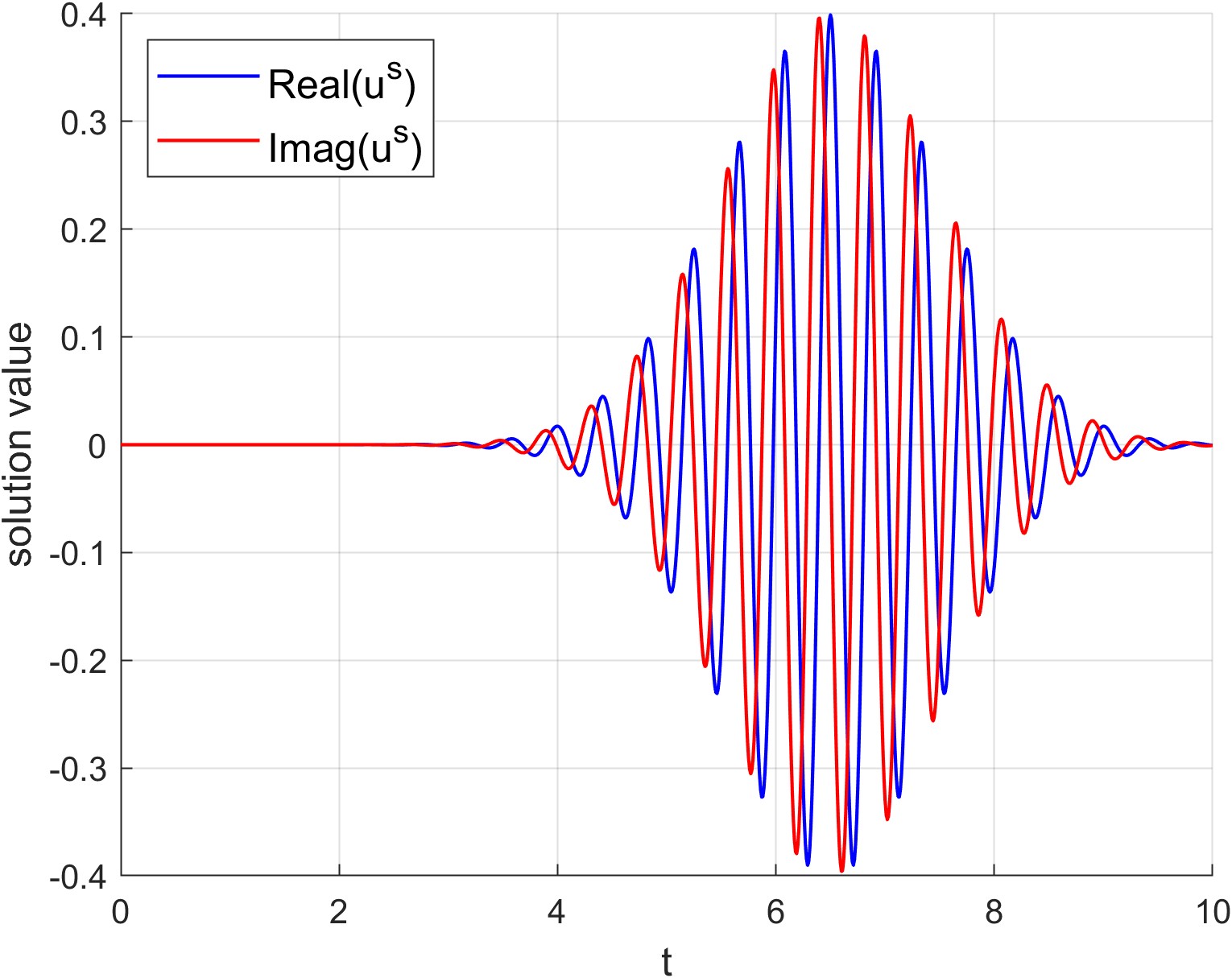}
& \includegraphics[scale=0.08]{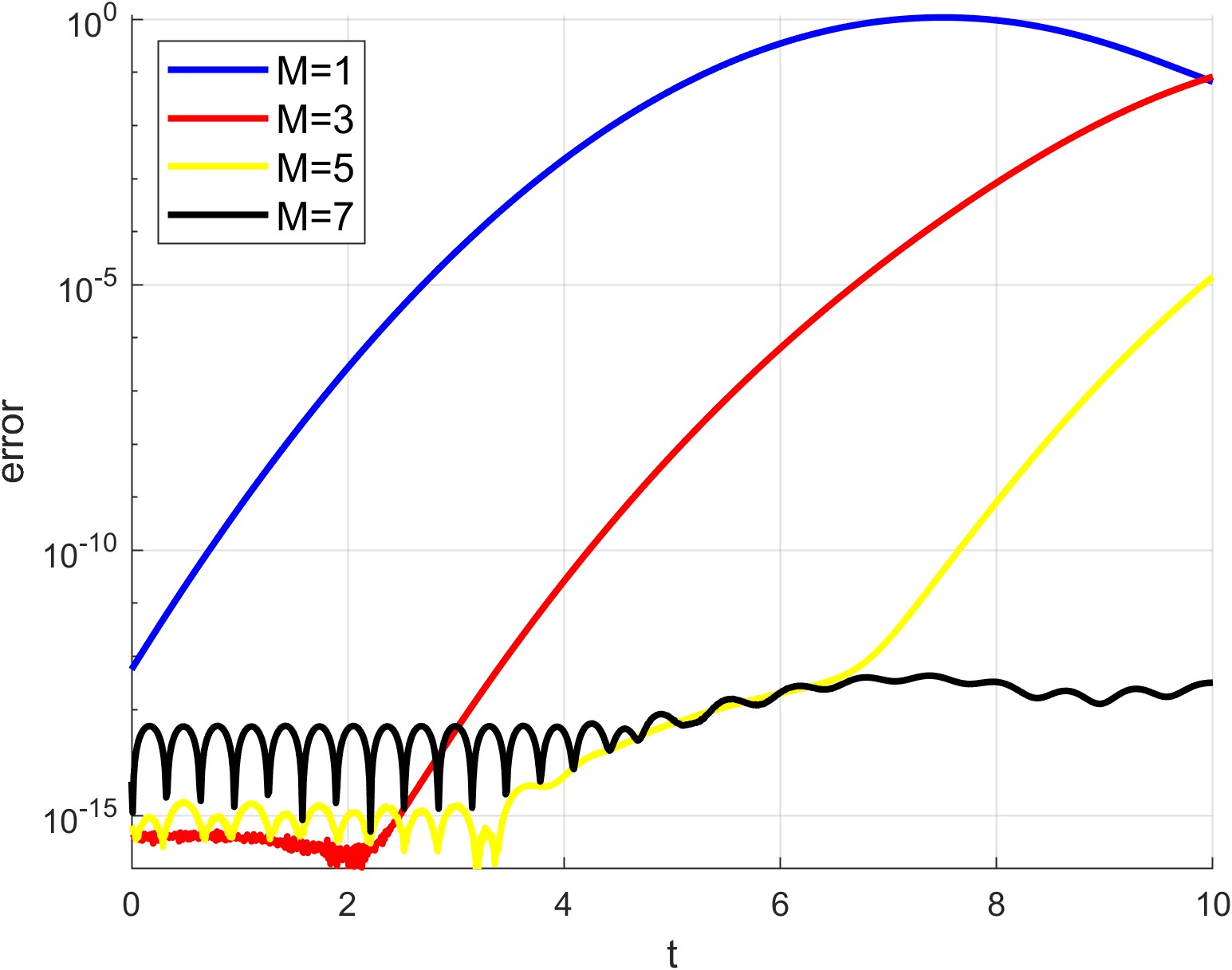} \\ (a) & (b)
\end{tabular}
\caption{Scattered field and errors obtained for the problem
considered in Geometry 1 with $R=1$ and $N=3$. (a) Real and imaginary
parts of scattered field at $x=(0.5,0)^\top$ resulting from the
incident field $u_2^i$. (b) Numerical errors as functions of time $t$
for various values of $M$.}
\label{Example1.2}
\end{figure}

\begin{figure}[htbp] \centering
\begin{tabular}{cc} \includegraphics[scale=0.08]{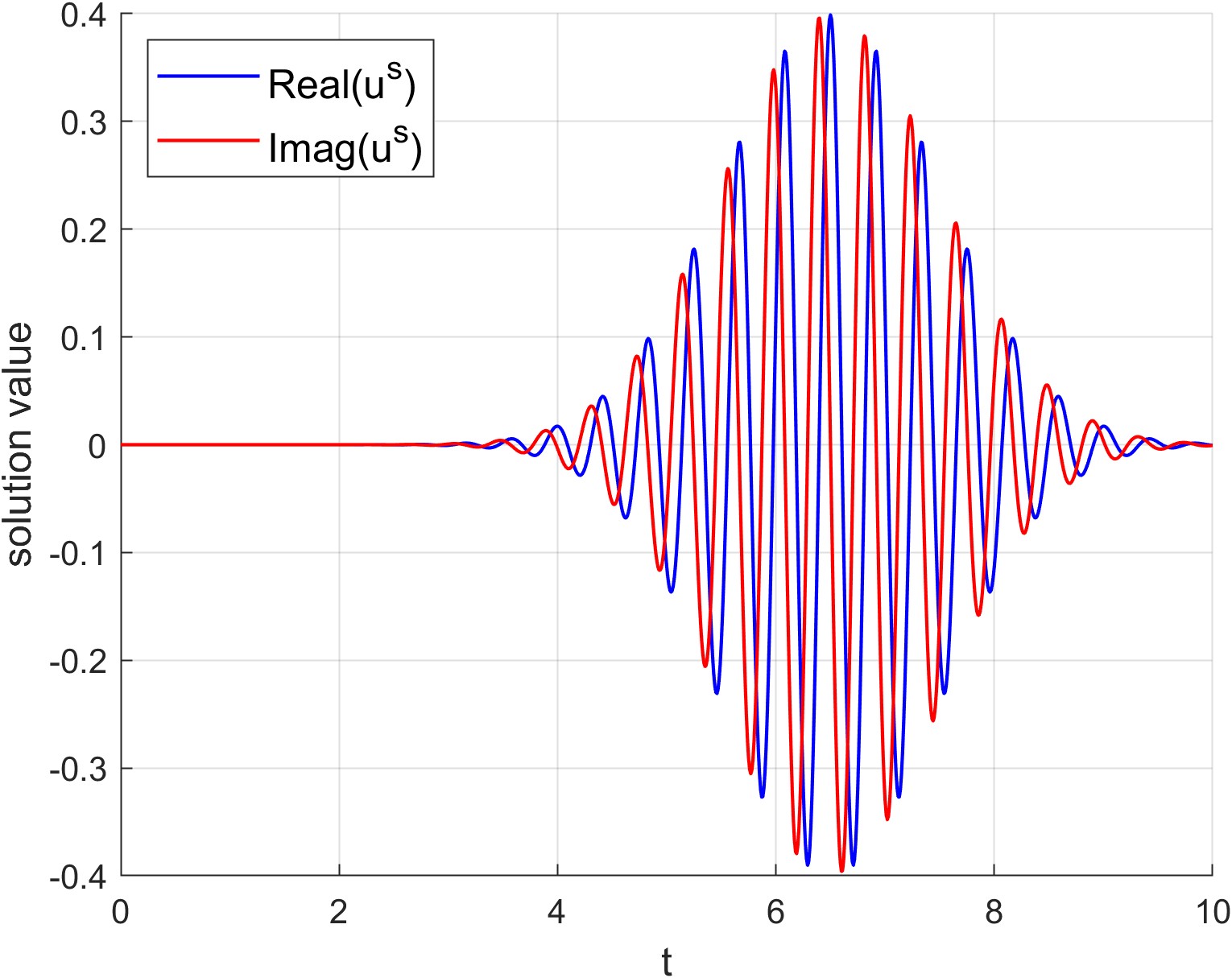}
& \includegraphics[scale=0.08]{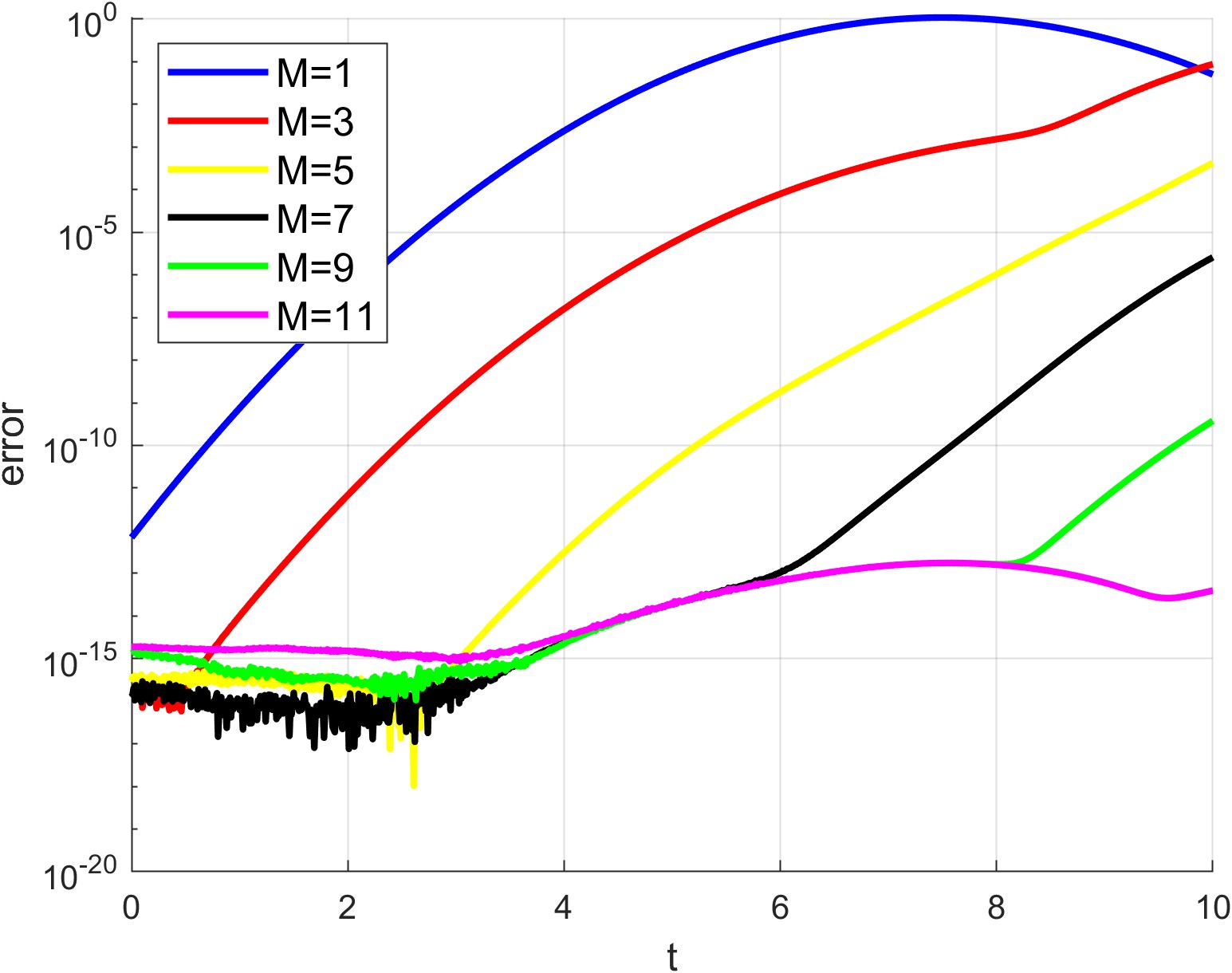} \\ (a) & (b)
\end{tabular}
\caption{Scattered field and errors obtained for the problem
considered in Example 1 with $R=2$ and $N=6$. (a) Real and imaginary
parts of scattered field at $x=(0.5,0)^\top$ resulting from the
incident field $u_2^i$. (b) Numerical errors as functions of time $t$
for various values of $M$.}
\label{Example1.3}
\end{figure}

\begin{figure}[htbp] \centering
\includegraphics[scale=0.12]{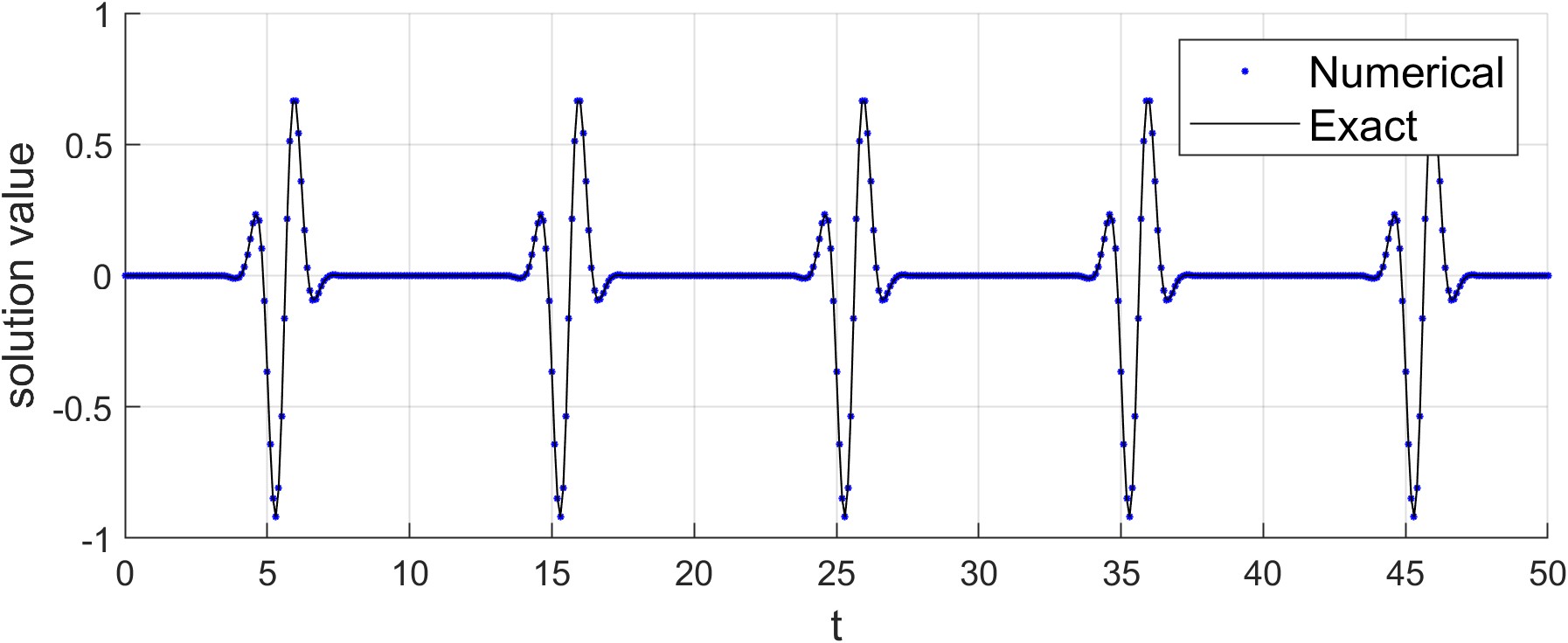}
\caption{Long-time propagation illustration. Real parts of the
numerical and exact solutions at $x=(0.5,0)^\top$ resulting from the
incident field $u_4^{i}$.}
\label{Example1.4}
\end{figure}

\vskip 0.2cm
\noindent {\bf Geometry 2.} Scattering problems in the interior of a
complex H-shaped domain centered at the origin are considered next,
utilizing  the point source incidence function $u_1^i$ with
$z_0=(0,0)$ and with two different frequencies: $\omega_0=15$ and
$\omega_0=50$. The Fourier-transforms utilized to produce the corresponding time-domain incident-fields (which are necessary to obtain the total time-domain fields $u^\mathrm{tot} = u^i +u$ displayed in Figure~\ref{Example2.1}) were evaluated by integrating in the frequency ranges $[5,25]$ and $[40,60]$ respectively---outside which the frequency-domain incident fields are vanishingly small. The parameter values $N=4$, $J=501, \Delta t=0.01$ and
$M=10$ were used. Figure~\ref{Example2.1} displays the total field $u^\mathrm{tot}$
within $D$ at various times.

\begin{figure}[htbp] \centering
\begin{tabular}{cccc} \includegraphics[scale=0.2]{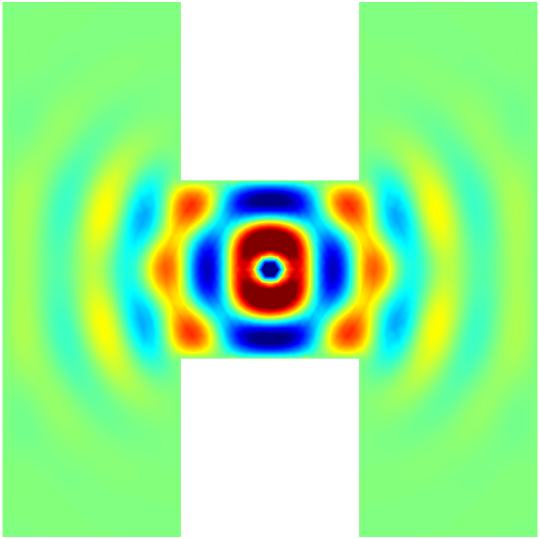}
& \includegraphics[scale=0.2]{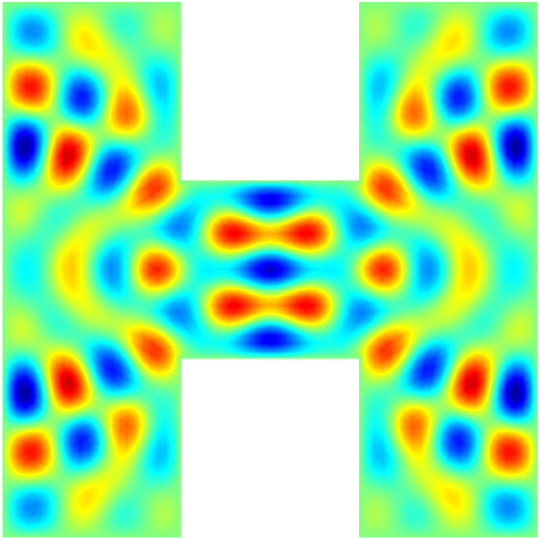} &
\includegraphics[scale=0.2]{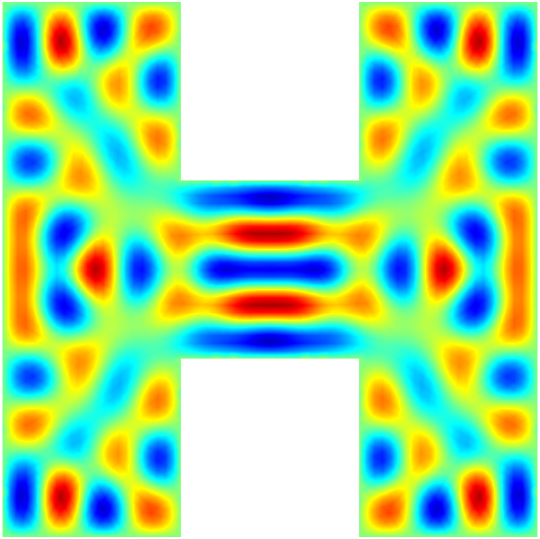} &
\includegraphics[scale=0.2]{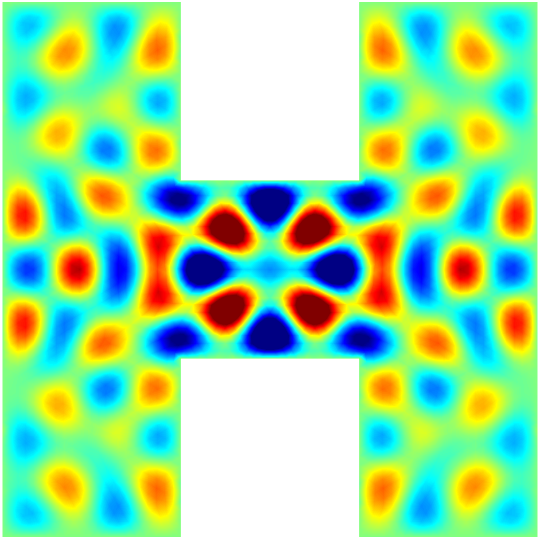} \\
\includegraphics[scale=0.2]{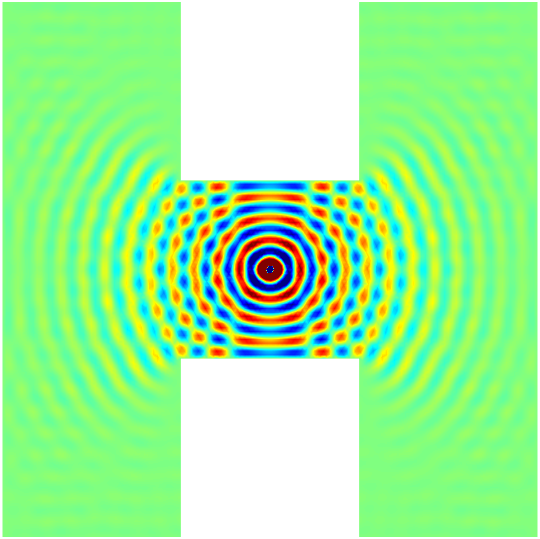} &
\includegraphics[scale=0.2]{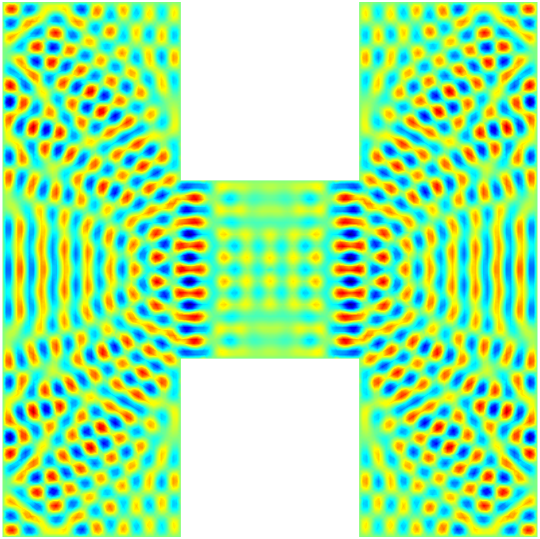} &
\includegraphics[scale=0.2]{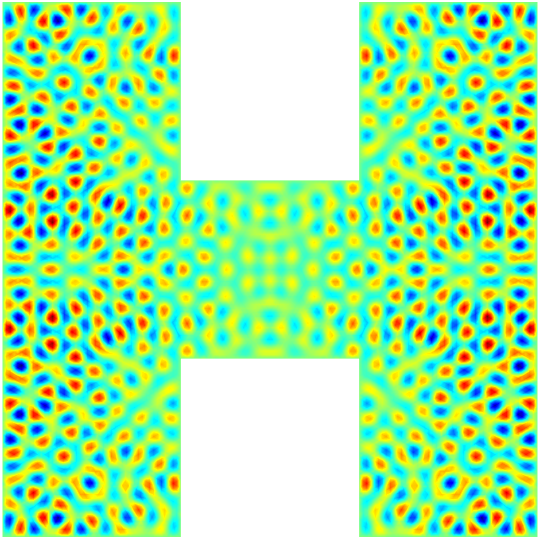} &
\includegraphics[scale=0.2]{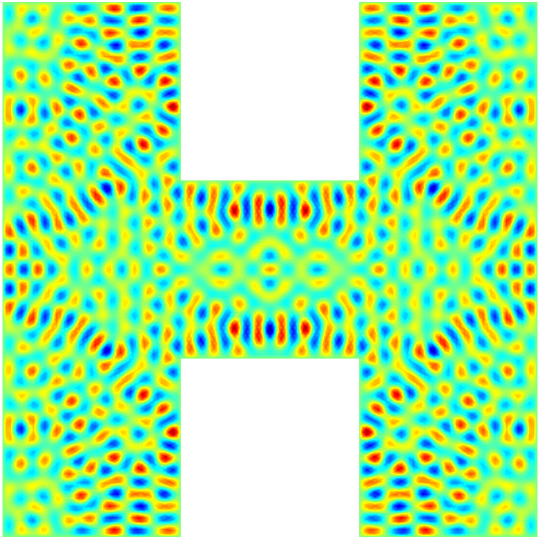}
\end{tabular}
\caption{Real part of the total field $u^\mathrm{tot}$ for the problems of scattering associated with the
Geometry 2, with point source $u_1^i$ located at
$z_0=(0,0)^\top$. Upper row:$w_0=15$, Lower row: $w_0=50$. Fields at
time $t=4,6,8$ and $10$ are displayed from left to right in each row.}
\label{Example2.1}
\end{figure}

%\begin{figure}[htbp] %\centering
%\begin{tabular}{cccc} %\includegraphics[scale=0.2]{Example2-z01-t3} &
%\includegraphics[scale=0.2]{Example2-z01-t5} &
%\includegraphics[scale=0.2]{Example2-z01-t7} &
%\includegraphics[scale=0.2]{Example2-z01-t9} \\
%\includegraphics[scale=0.2]{Example2-z02-t3} &
%\includegraphics[scale=0.2]{Example2-z02-t5} &
%\includegraphics[scale=0.2]{Example2-z02-t7} &
%\includegraphics[scale=0.2]{Example2-z02-t9}
%\end{tabular}
%\caption{Real part of the total fields for the problem considered in
% Example 2 with point source $u_3^i$. Upper row:$z_0=(0,0)^\top$, Lower
% row: $z_0=(-1,0)$. Fields at time $t=3,5,7$ and $9$ are displayed from
% left to right in each row.}
%\label{Example2.2}
%\end{figure}

\vskip 0.2cm
\noindent {\bf Geometries 3.} Problems of scattering posed in the
exterior of the groups of obstacles depicted in
Figures~\ref{Example3.1}(a) and \ref{Example3.2}(a) are now considered, with incident fields given by the functions $u_1^i$ and
$u_2^i$. For the incident function $u_1^i$ the point $z_0$ is located
within one of the obstacles, and, therefore, the exact solution for
the problem is given by $u(x,t)=-u_1^i(x,t)$ for $x\in D^\mathrm{e}$.
The numerical errors as a function of $t$ for various values of $M$
are displayed in Figures~\ref{Example3.1} and~\ref{Example3.2} which
clearly demonstrate the high accuracy of the proposed
solver. Figure~\ref{Example3.3} displays the numerical total field
resulting from the scattering of the plane incidence $u_2^i$ for
various values of $\theta^{inc}$. The parameter values $\omega_0=15$,
$W=25$, $J=800$ and $\Delta t=0.0125$ were used for all of the
Geometry 3 cases.

\begin{figure}[htbp] \centering
\begin{tabular}{ccc} \includegraphics[scale=0.12]{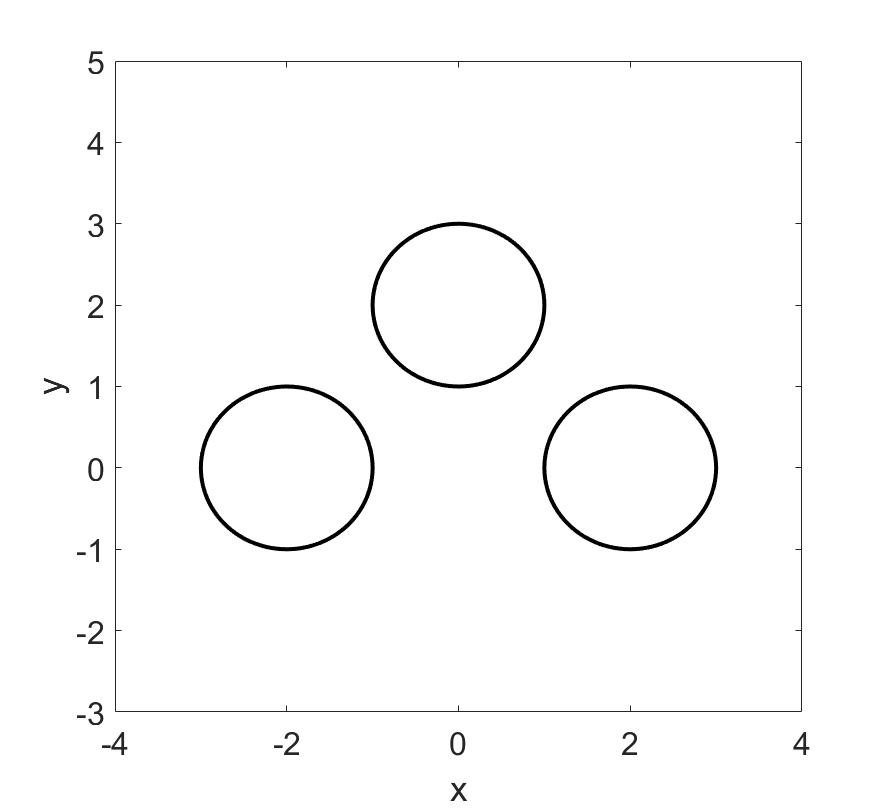}
& \includegraphics[scale=0.15]{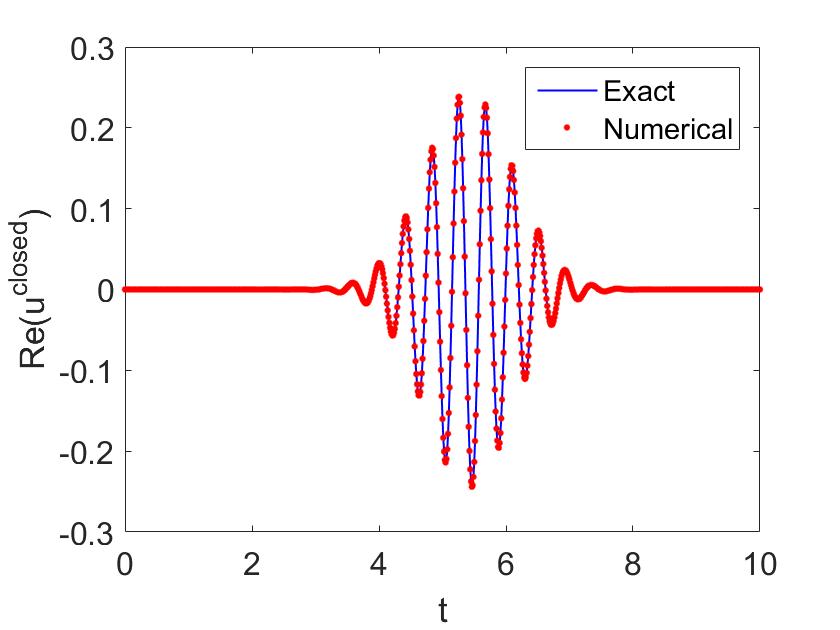} &
\includegraphics[scale=0.15]{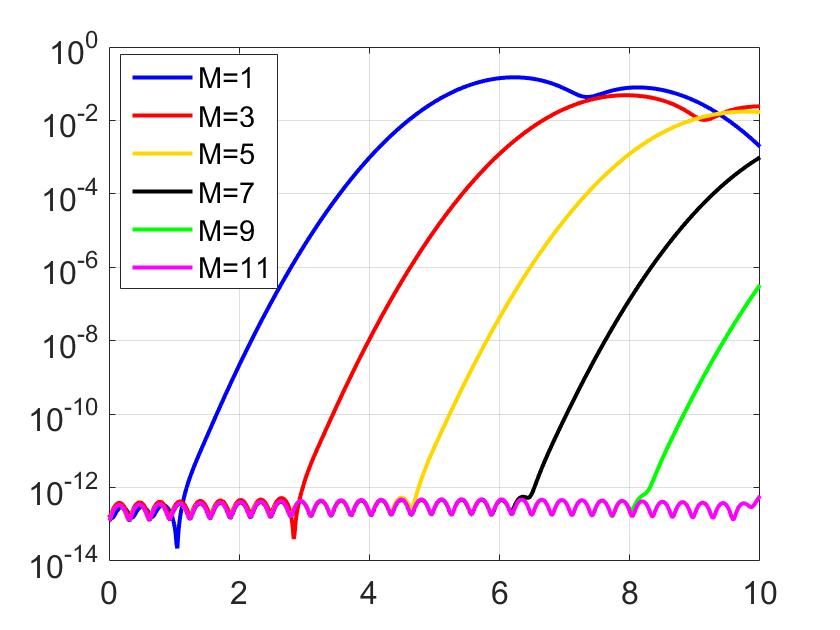} \\ (a) & (b) & (c)
\end{tabular}
\caption{Real part of the scattered field (b) at $x=(-1,1)^\top$
resulting from the incident field $u_1^i$ with $z_0=(0,2)^\top$ and
numerical errors (c) as functions of time $t$ for various values of
$M$ obtained for the exterior wave equation problem with 3 bounded
obstacles (a) considered in Geometry 3.}
\label{Example3.1}
\end{figure}

\begin{figure}[htbp] \centering
\begin{tabular}{ccc} \includegraphics[scale=0.12]{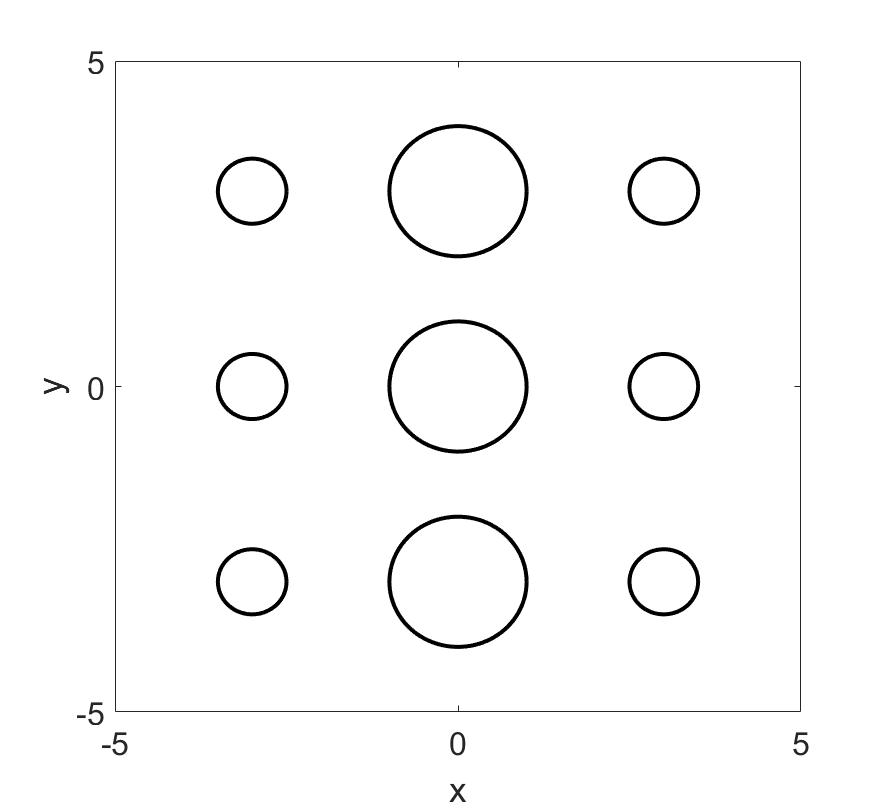}
& \includegraphics[scale=0.15]{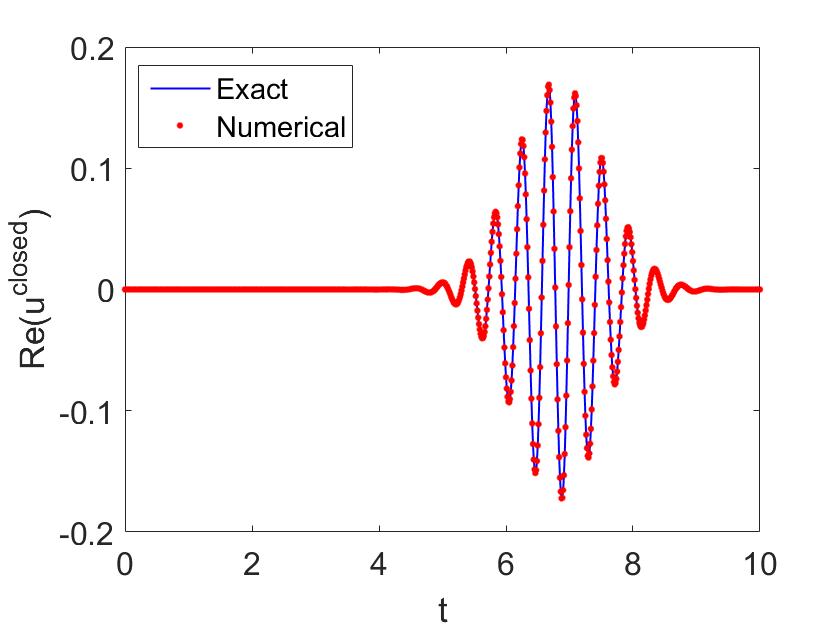} &
\includegraphics[scale=0.15]{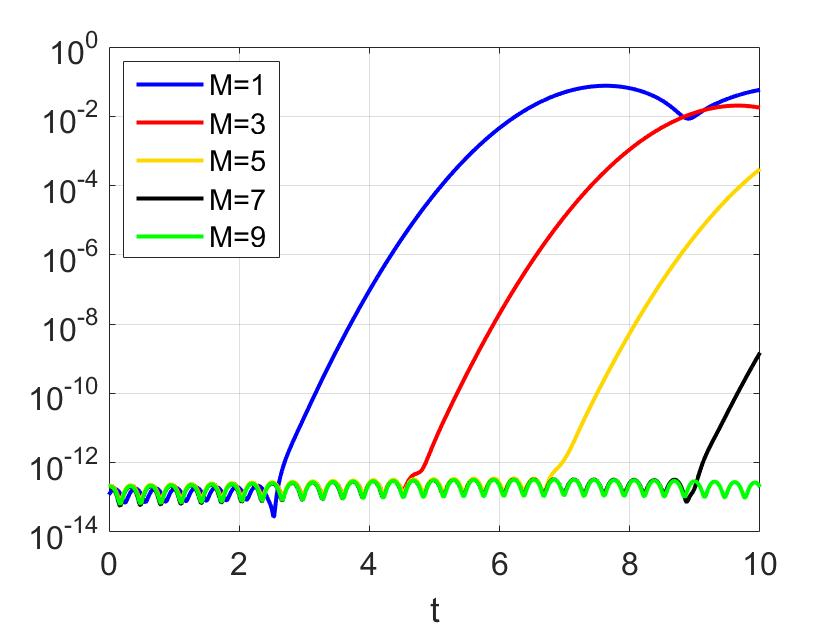} \\ (a) & (b) & (c)
\end{tabular}
\caption{Real part of the scattered field (b) at $x=(2,2)^\top$
resulting from the incident field $u_1^i$ with $z_0=(0,0)^\top$ and
numerical errors (c) as functions of time $t$ for various values of
$M$ obtained for the exterior wave equation problem with 9 bounded
obstacles (a) considered in Geometry 3.}
\label{Example3.2}
\end{figure}

\begin{figure}[htbp] \centering
\begin{tabular}{cccc}
\includegraphics[scale=0.06]{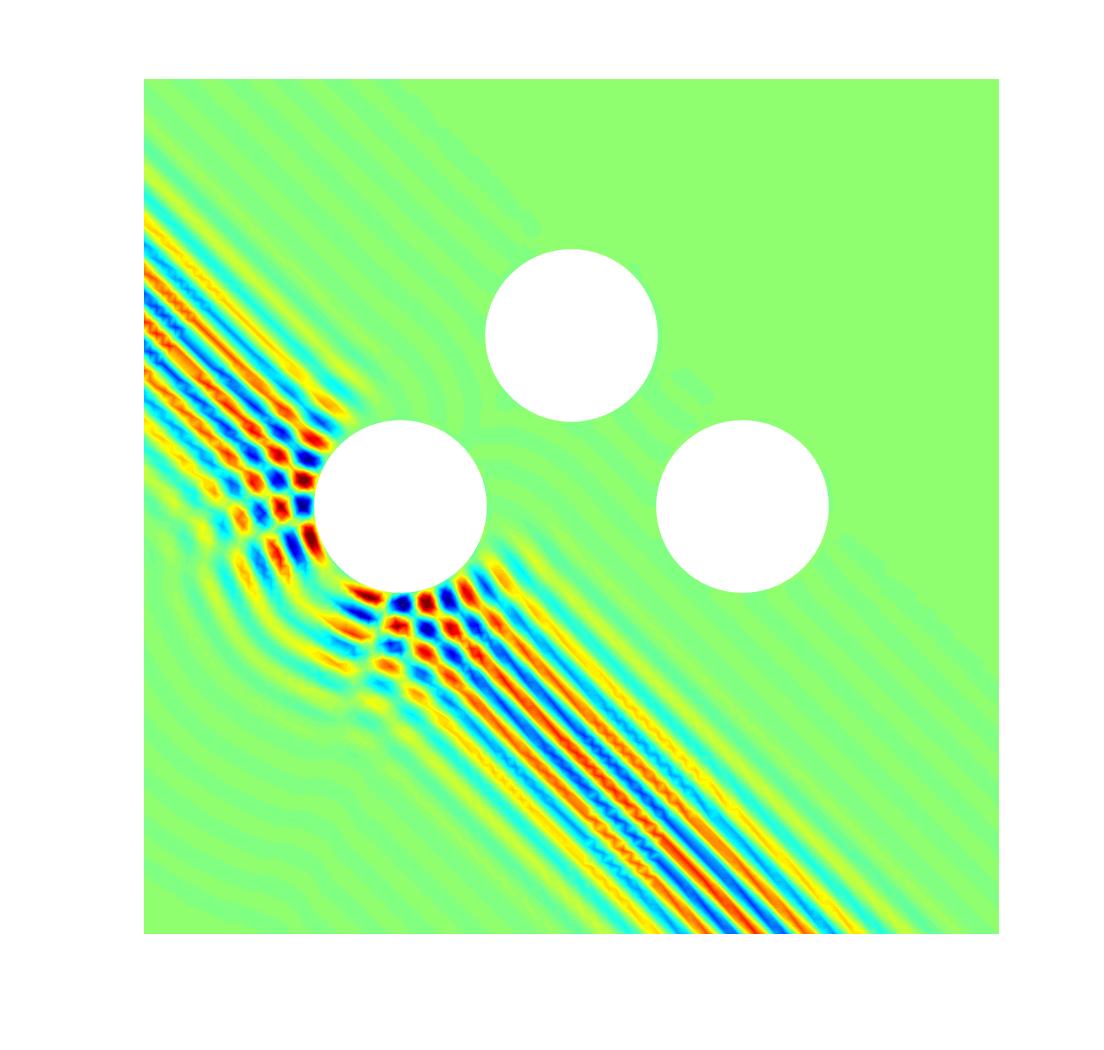} &
\includegraphics[scale=0.06]{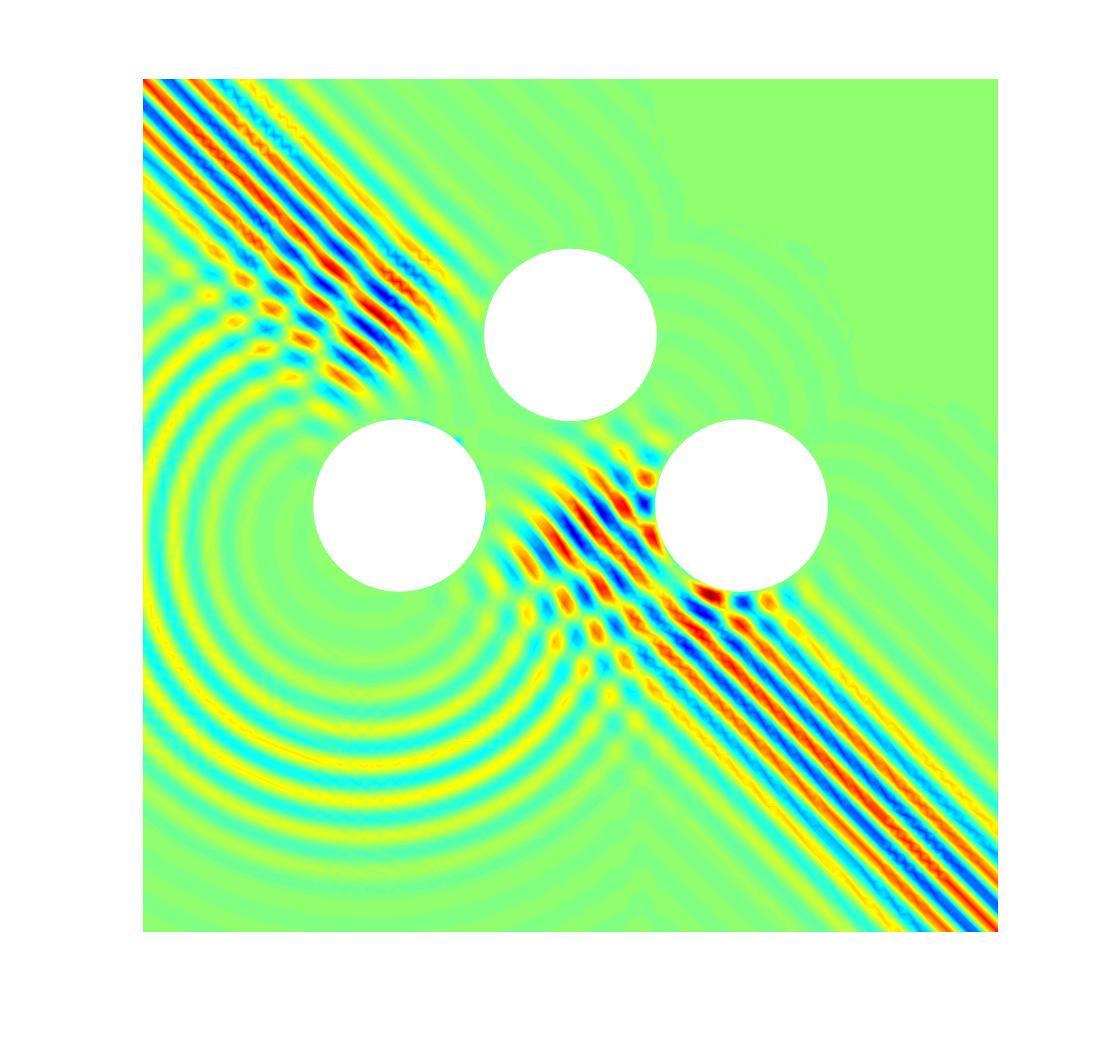} &
\includegraphics[scale=0.06]{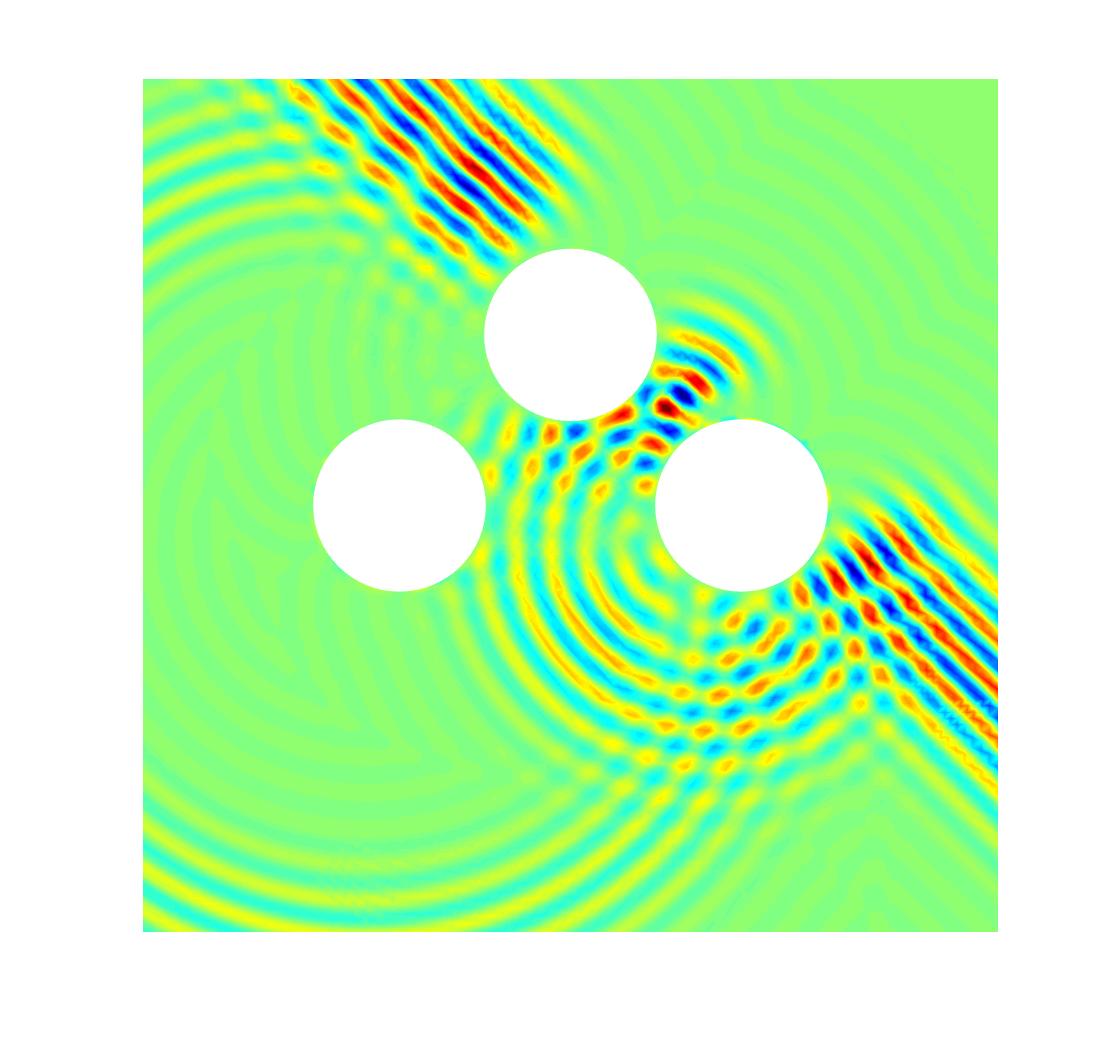} &
\includegraphics[scale=0.06]{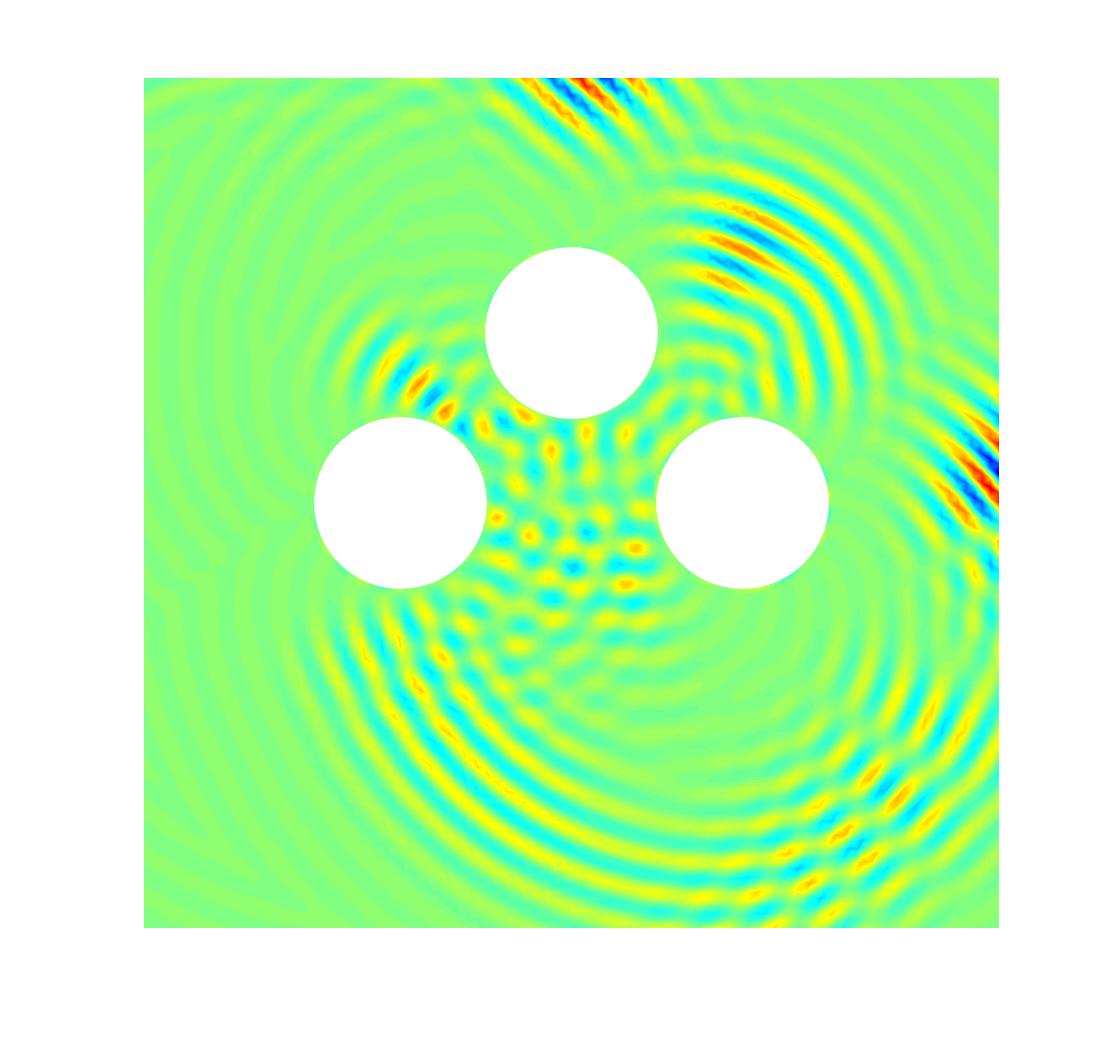} \\
\includegraphics[scale=0.06]{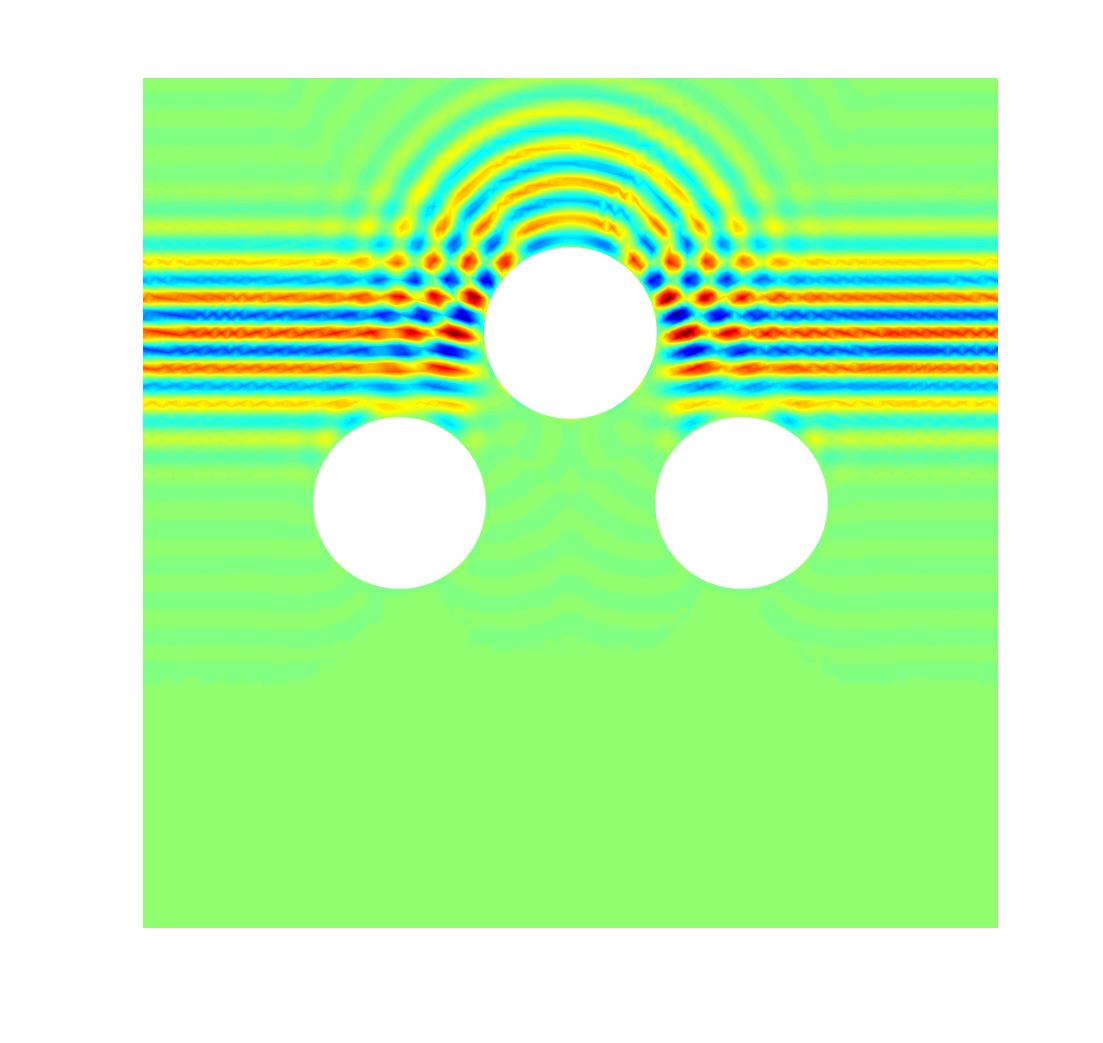} &
\includegraphics[scale=0.06]{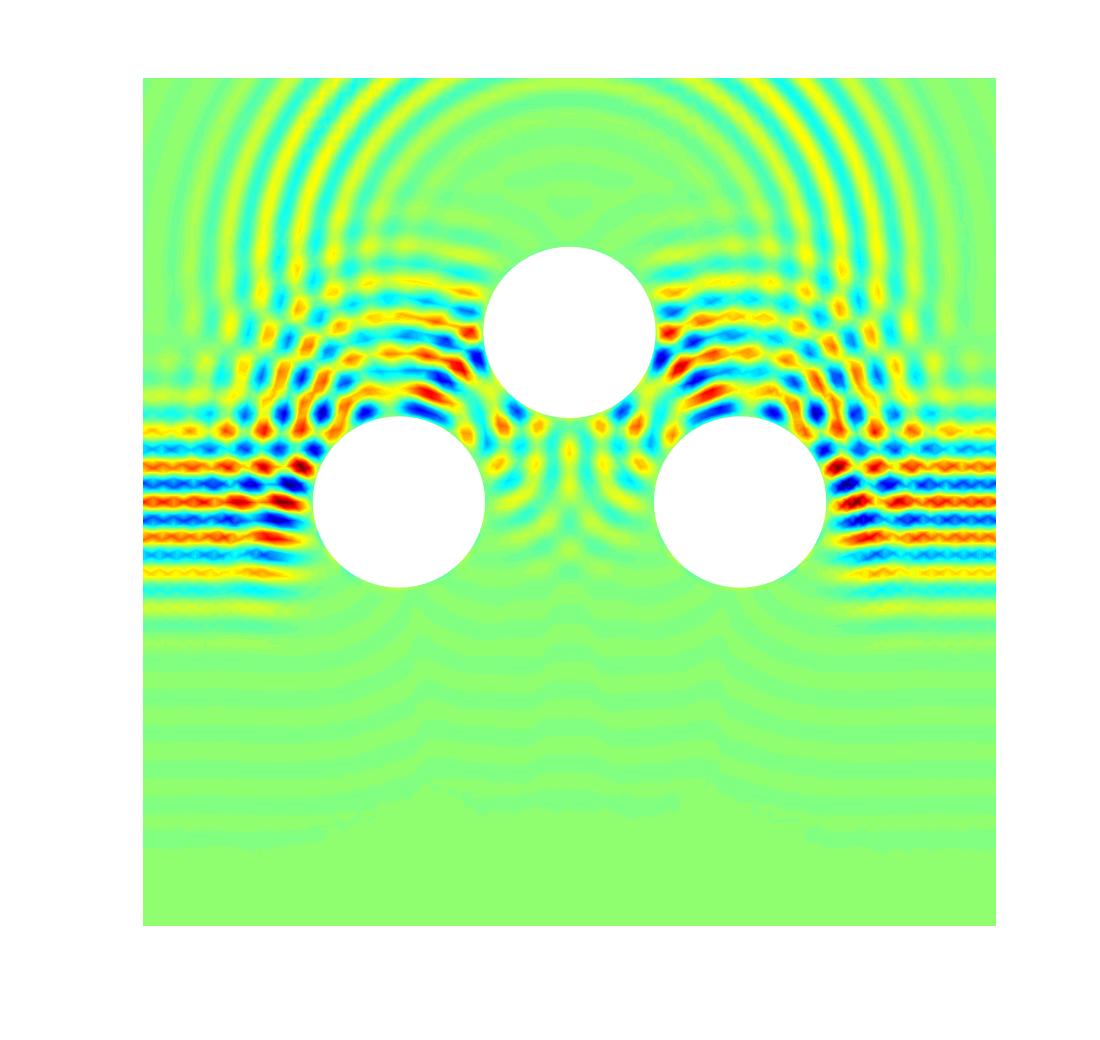} &
\includegraphics[scale=0.06]{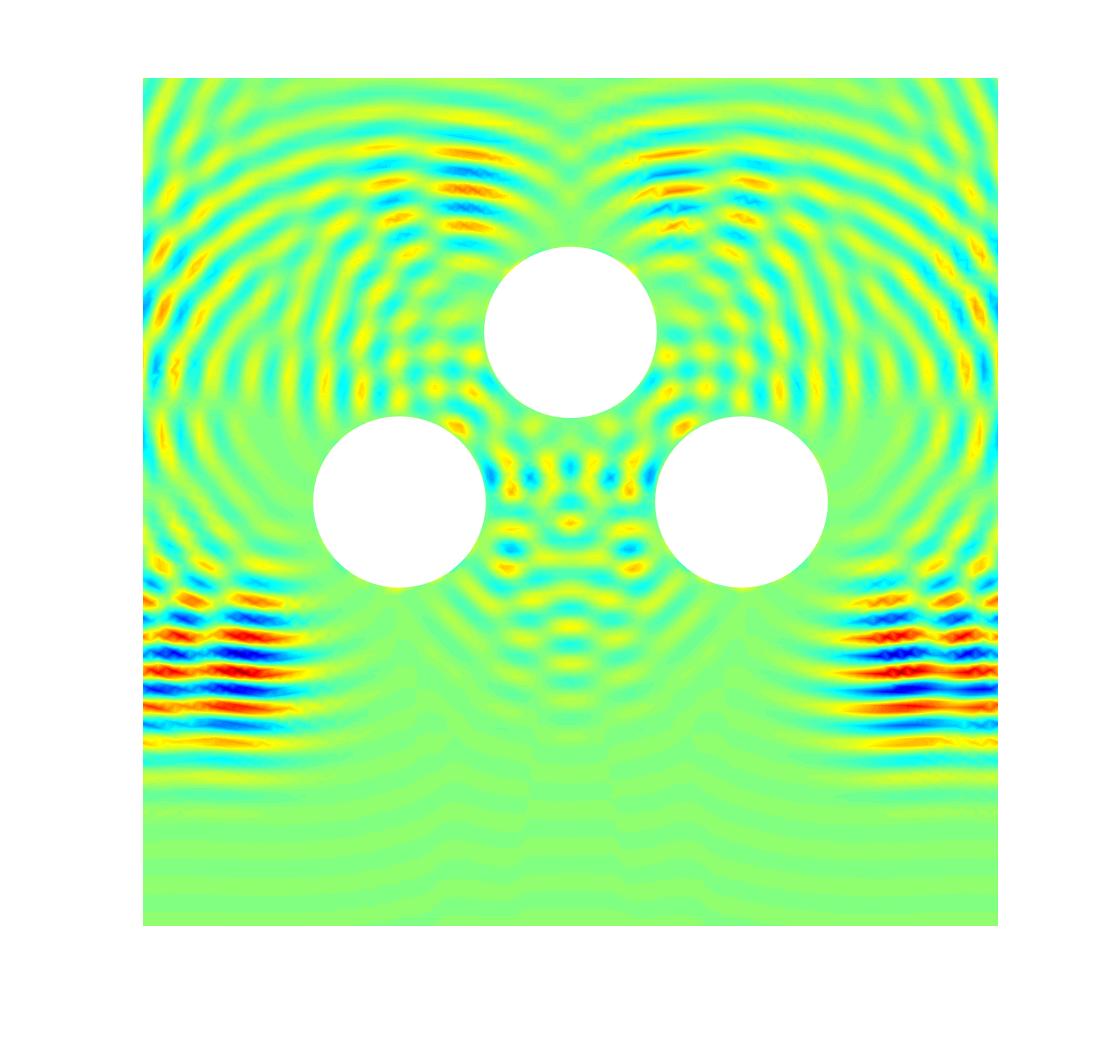} &
\includegraphics[scale=0.06]{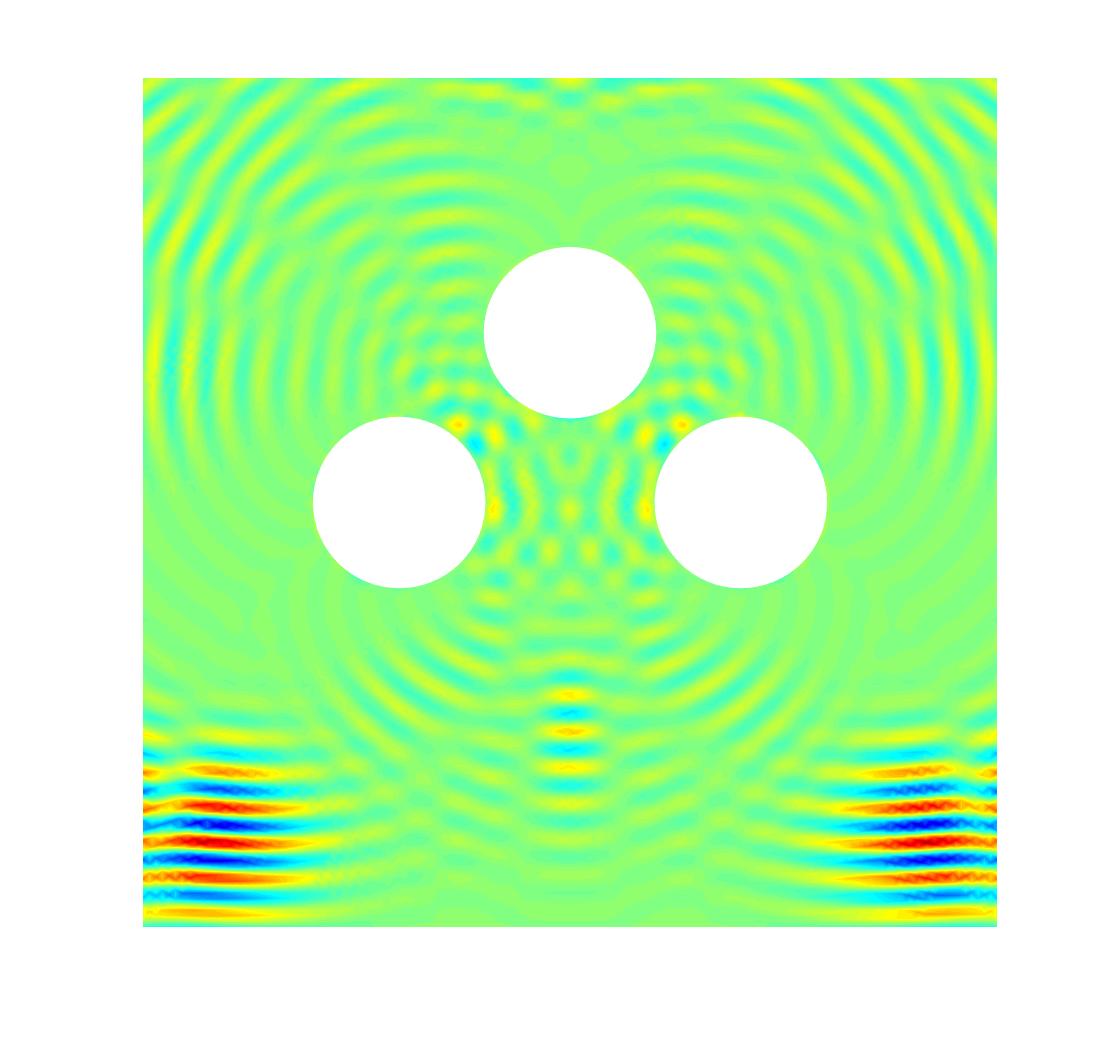}
\end{tabular}
\caption{Real part of the total fields for the scattering of plane
incidence $u^i_1$ by three bounded obstacles considered in Geometry
3. Upper row: $\theta^\textrm{inc}=\frac{\pi}{4}$. Lower row:
$\theta^\textrm{inc}=-\frac{\pi}{2}$. Fields at times $t = 2, 4, 6$
and 8 are displayed from left to right in each row.}
\label{Example3.3}
\end{figure}

\vskip 0.2cm
\noindent {\bf Geometry 4.} We now consider  scattering problems
in the exterior of the group of three open-arcs
displayed in Figure~\ref{Example4.1}(a). A point source $u_1^i$ with
$z_0=(0,0)^\top$ and $\omega_0=15$ is considered for this example, for
which we used the parameter values $W=25$, $J=800$ and $\Delta
t=0.0125$. The resulting numerical errors as a function of $t$ for
various values of $M$ are presented in Figure~\ref{Example4.1} which
also demonstrates the high accuracy of the proposed solver. The
numerical total fields resulting for this geometry from the scattering
of the plane incidence $u_2^i$ with $\theta^{inc}=-\frac{3}{4}\pi$ are
displayed in Figure~\ref{Example4.2}.

\begin{figure}[htbp] \centering
\begin{tabular}{ccc} \includegraphics[scale=0.15]{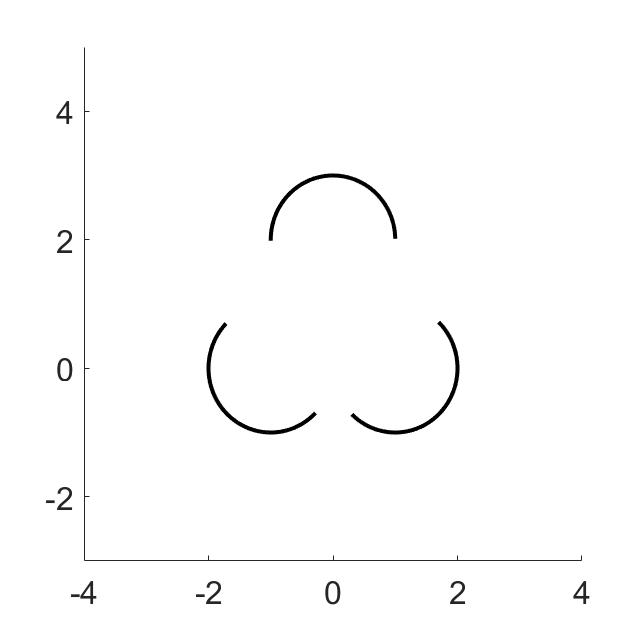}
& \includegraphics[scale=0.15]{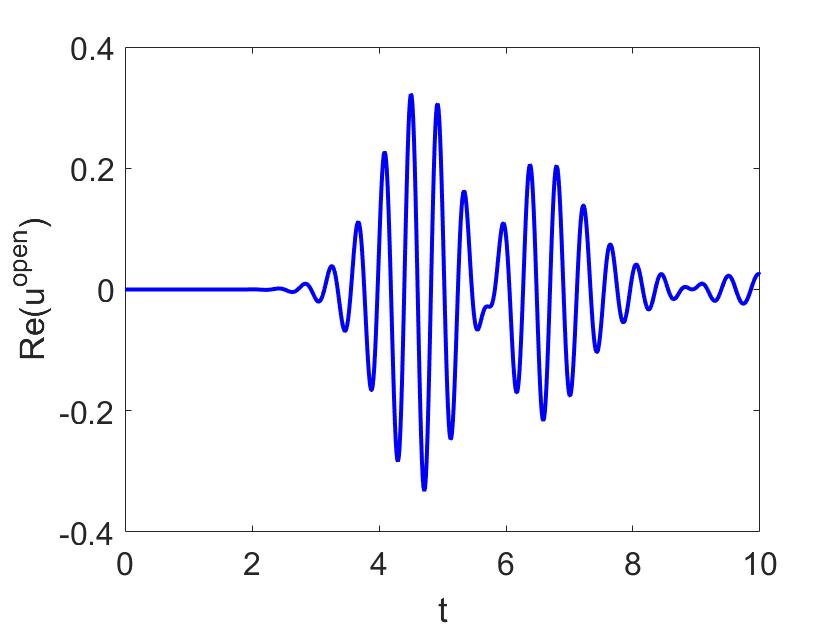} &
\includegraphics[scale=0.15]{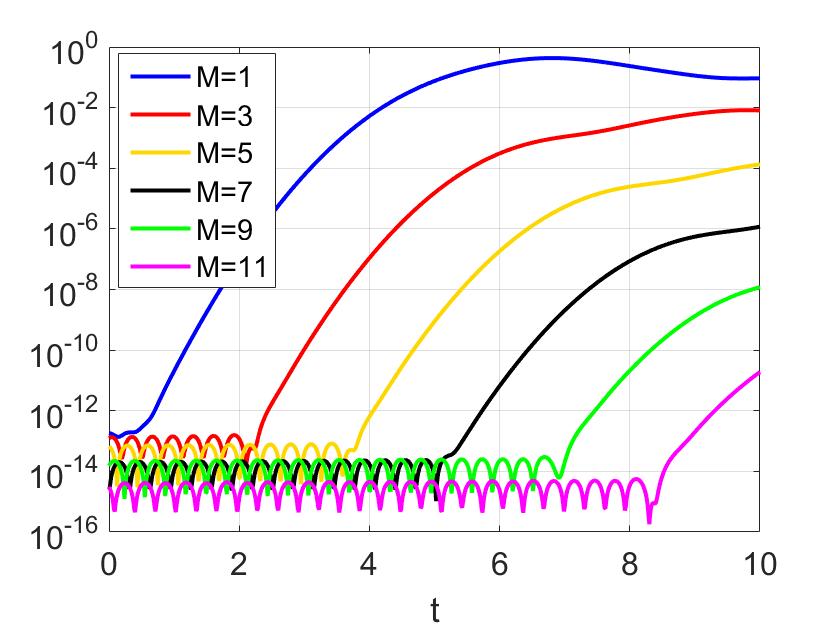} \\ (a) & (b) & (c)
\end{tabular}
\caption{Real part of the scattered field (b) at $x=(-1,0)^\top$
resulting from the incident field $u_2^i$ and numerical errors (c) as
functions of time $t$ for various values of $M$ obtained for the
exterior wave equation problem with 3 open-arcs (a) considered in
Geometry 4.}
\label{Example4.1}
\end{figure}

\begin{figure}[htbp] \centering
\begin{tabular}{cccc} 
\includegraphics[scale=0.06]{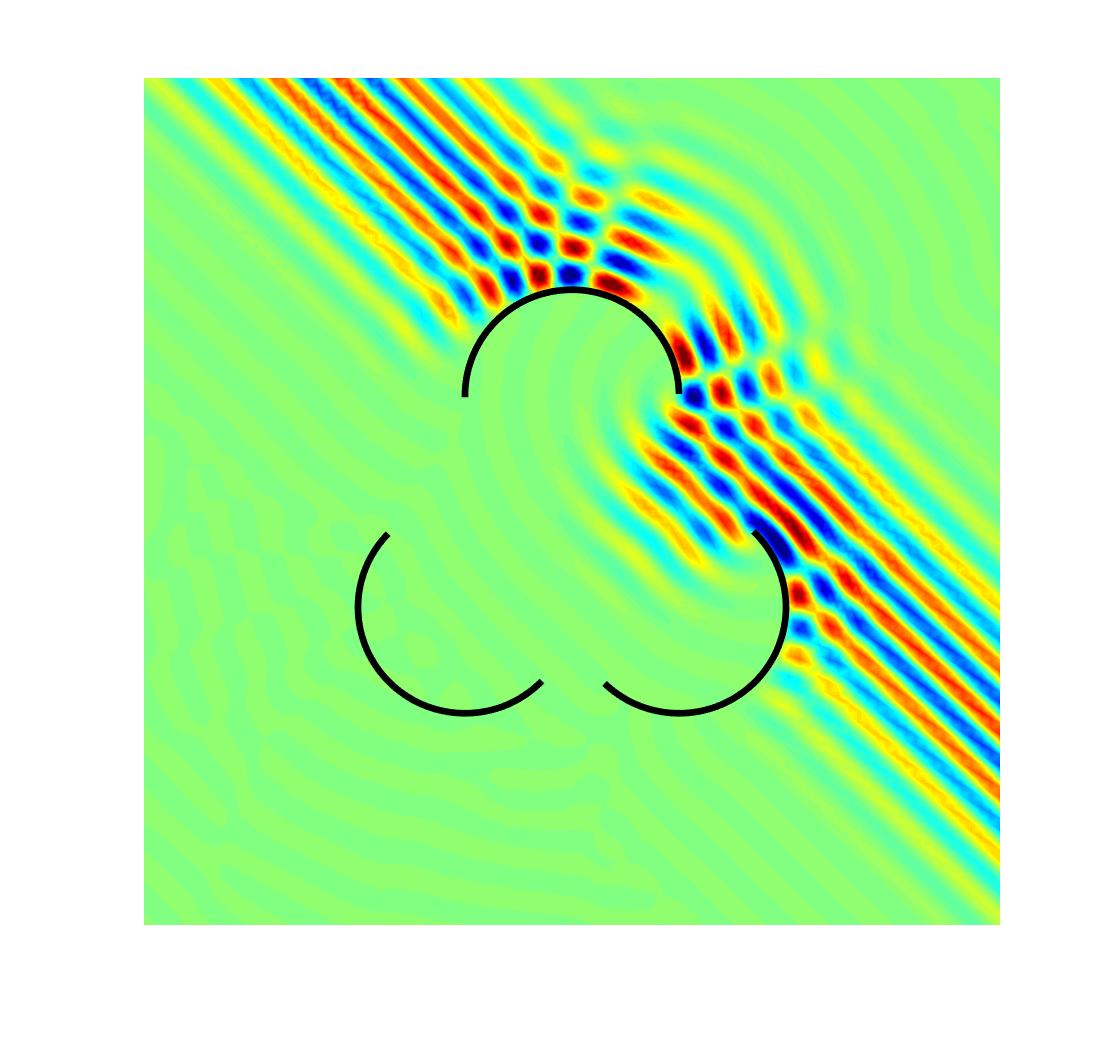} &
\includegraphics[scale=0.06]{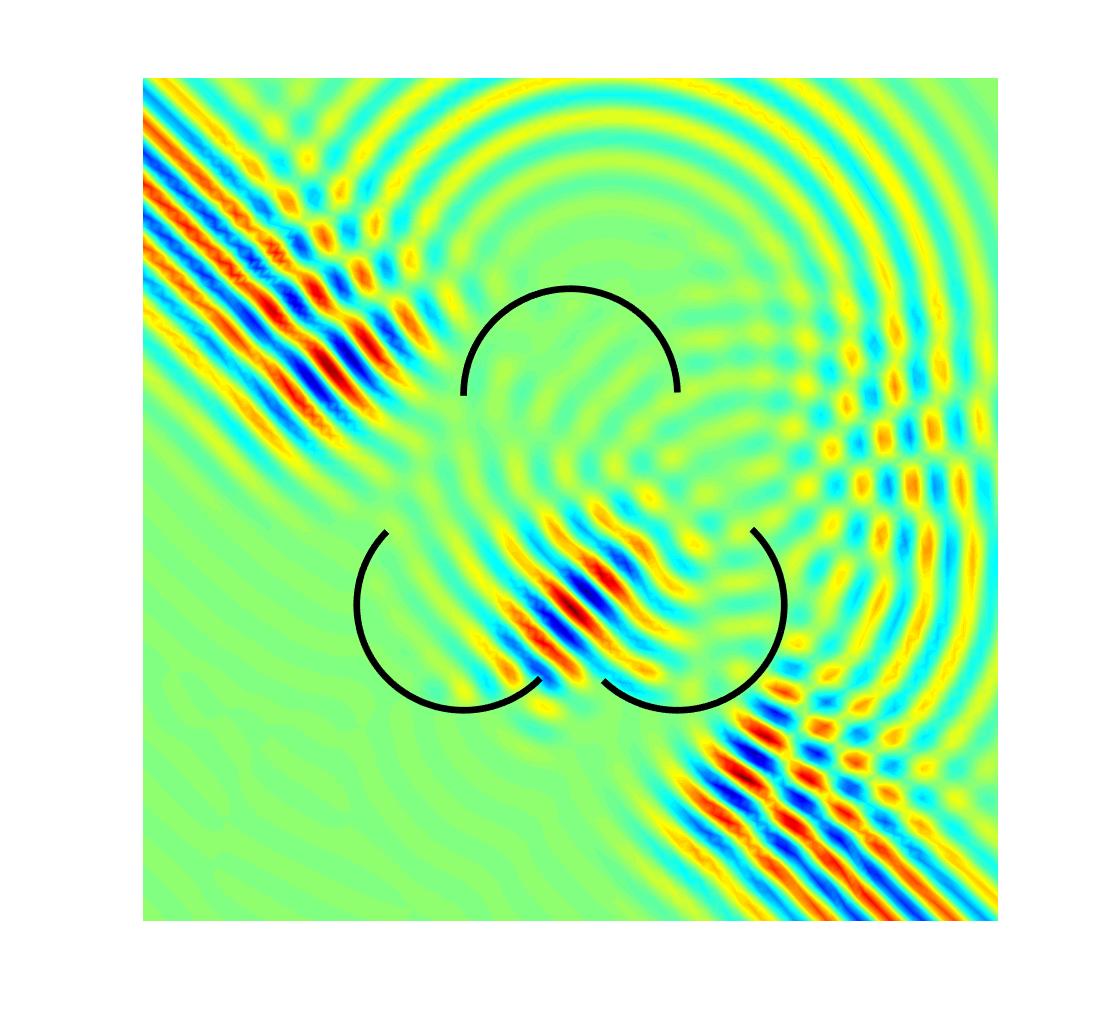} &
\includegraphics[scale=0.06]{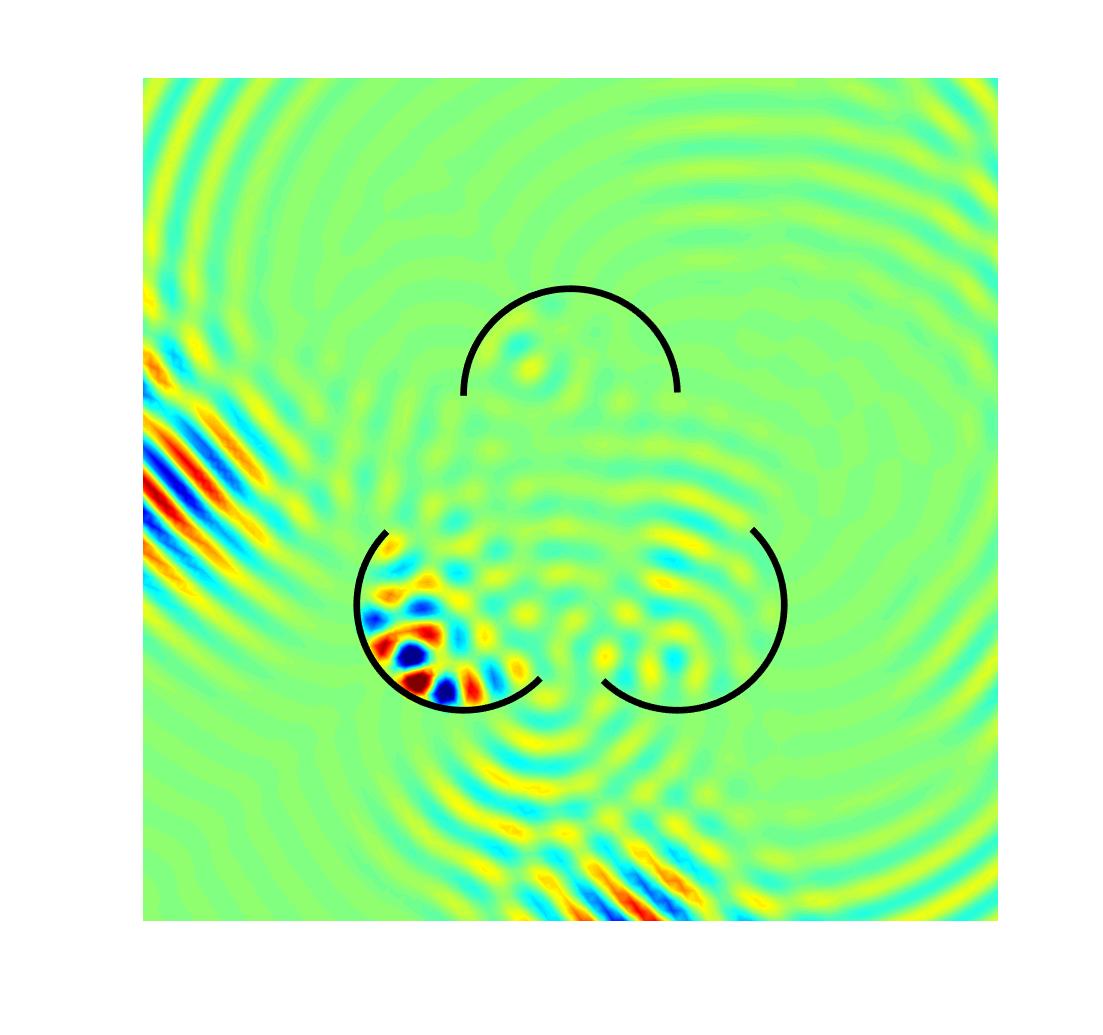} &
\includegraphics[scale=0.06]{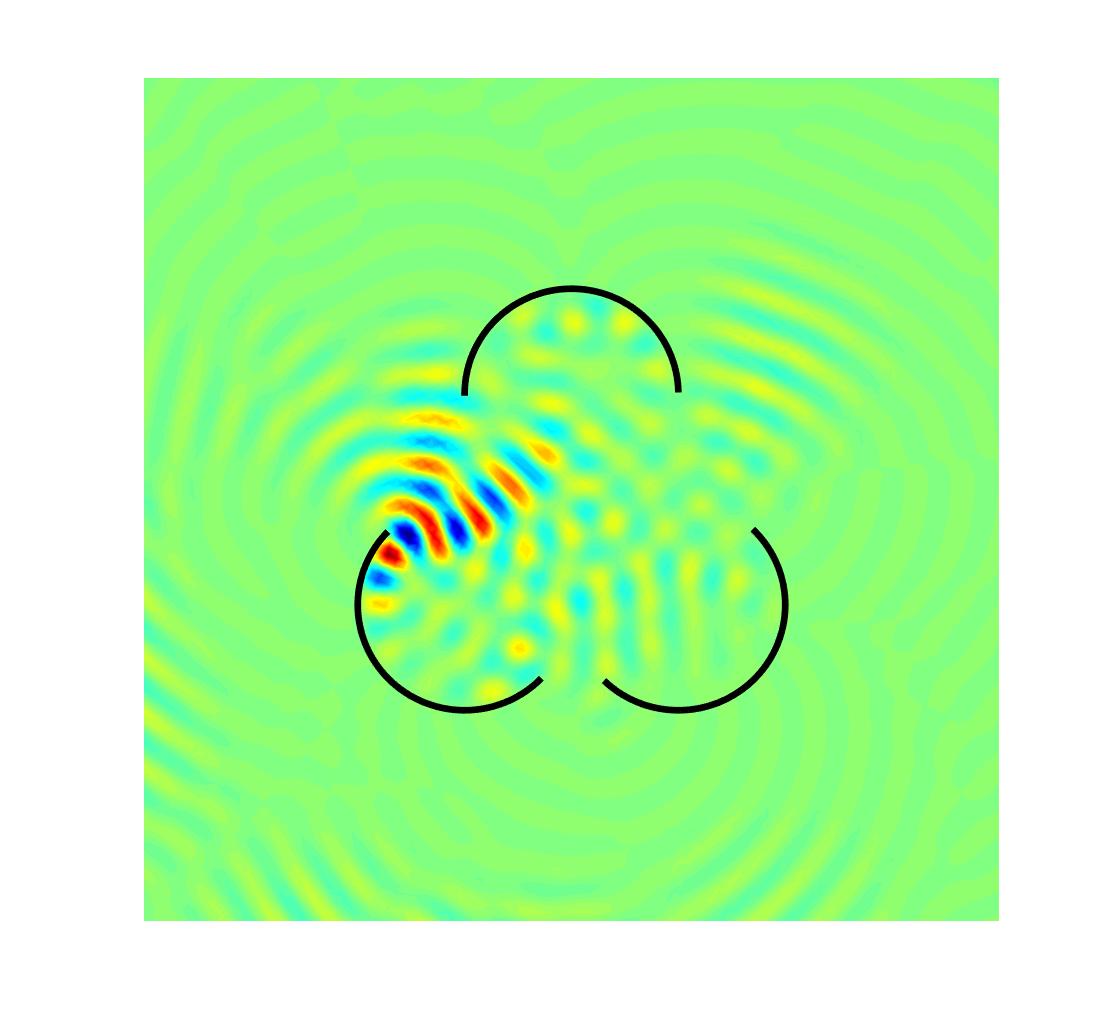}
\end{tabular}
\caption{Real part of the total fields for the scattering of plane
wave $u^i_2$ with $\theta^\textrm{inc}=-\frac{3}{4}\pi$ by three
open-arcs considered in Geometry 4. Fields at times $t = 2, 4, 6$ and
8 are displayed from left to right.}
\label{Example4.2}
\end{figure}

\vskip 0.2cm
\noindent {\bf Geometries 5.} Finally we consider problems of
scattering by a highly trapping open-arc cavity structures with a
small openings, whose numerical solution can be challenging on account
of wave trapping and associated large numbers of GMRES
frequency-domain iterations.  We consider two cavity geometries,
namely, a circular cavity
$\Gamma=\{x=(\cos\theta,\sin\theta)^\top\in\R^2:
\theta\in(0.02\pi,1.98\pi )\}$ and a rocket-like open cavity (see
Figure~\ref{Example5.4}) for which we choose $N=6$ and $N=10$,
respectively.  We first consider the problem of a circular cavity
under plane wave incidence, $u^i_2(x,t)$, with an incidence angle
$\theta^\text{inc} = -\pi$, and using the parameter values: $\omega_0
= 15$, $W = 25$, $J = 401$, and $\Delta t = 0.01$. Solution values at
$x=(-0.5,0)^\top$ as a function of $t$ together with numerical errors
resulting for various values of $M$ are displayed in
Figure~\ref{Example5.1}. The total fields for the higher-frequency
case $\omega_0=50$ are displayed in Figure~\ref{Example5.2}. We
emphasize the substantial benefits of boundary decomposition, which,
as illustrated in Figure~\ref{Example5.3}, can significantly reduce
the number of GMRES iterations required for the numerical evaluation
of the frequency-domain BIEs. The total fields for the rocket-like
open cavity under the $u_2^i(x, t)$ incident field with $\omega_0=50$
and $M=50$ are presented in Figure~\ref{Example5.4}.

\begin{figure}[htbp] \centering
\begin{tabular}{ccc} \includegraphics[scale=0.08]{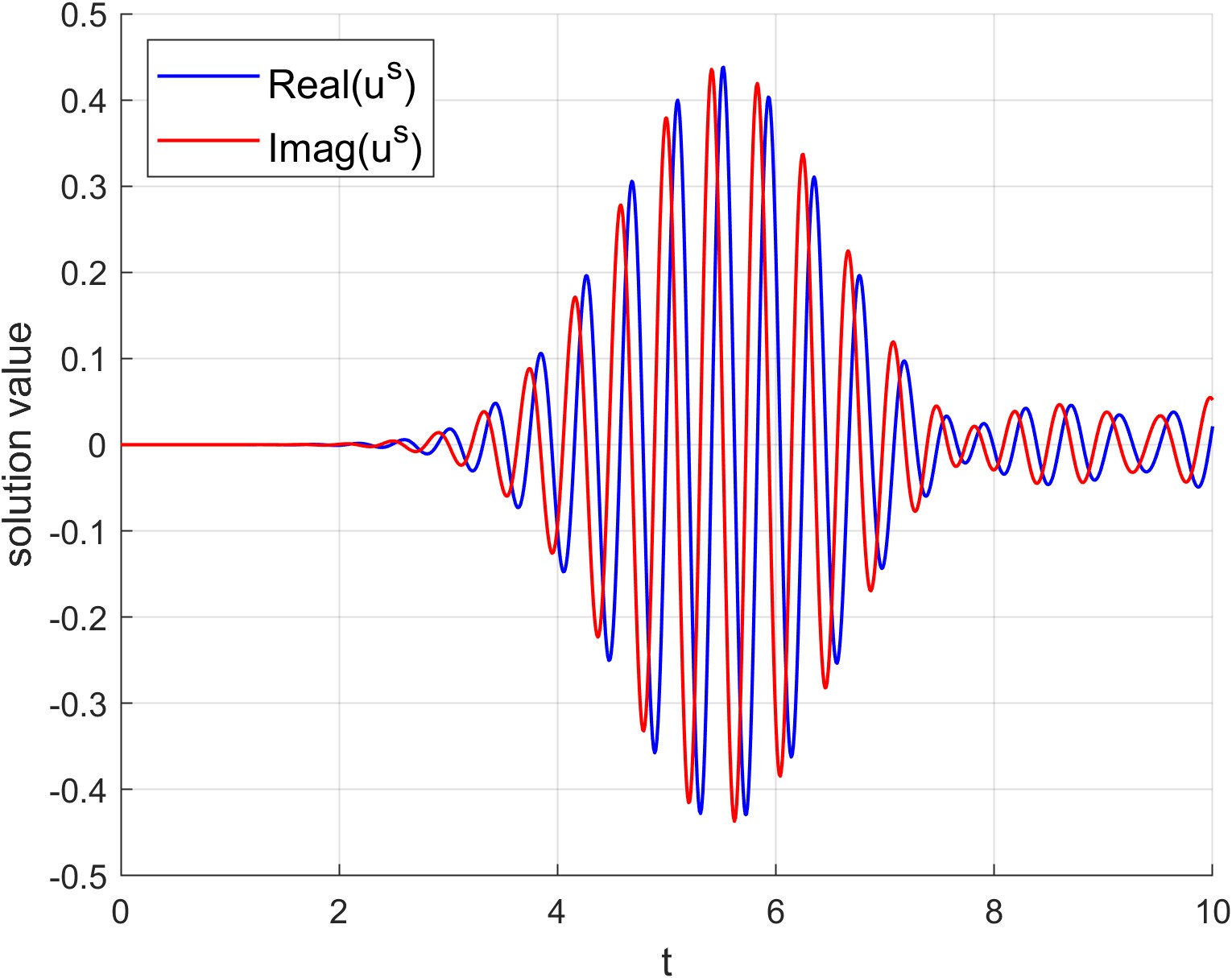} &
\includegraphics[scale=0.08]{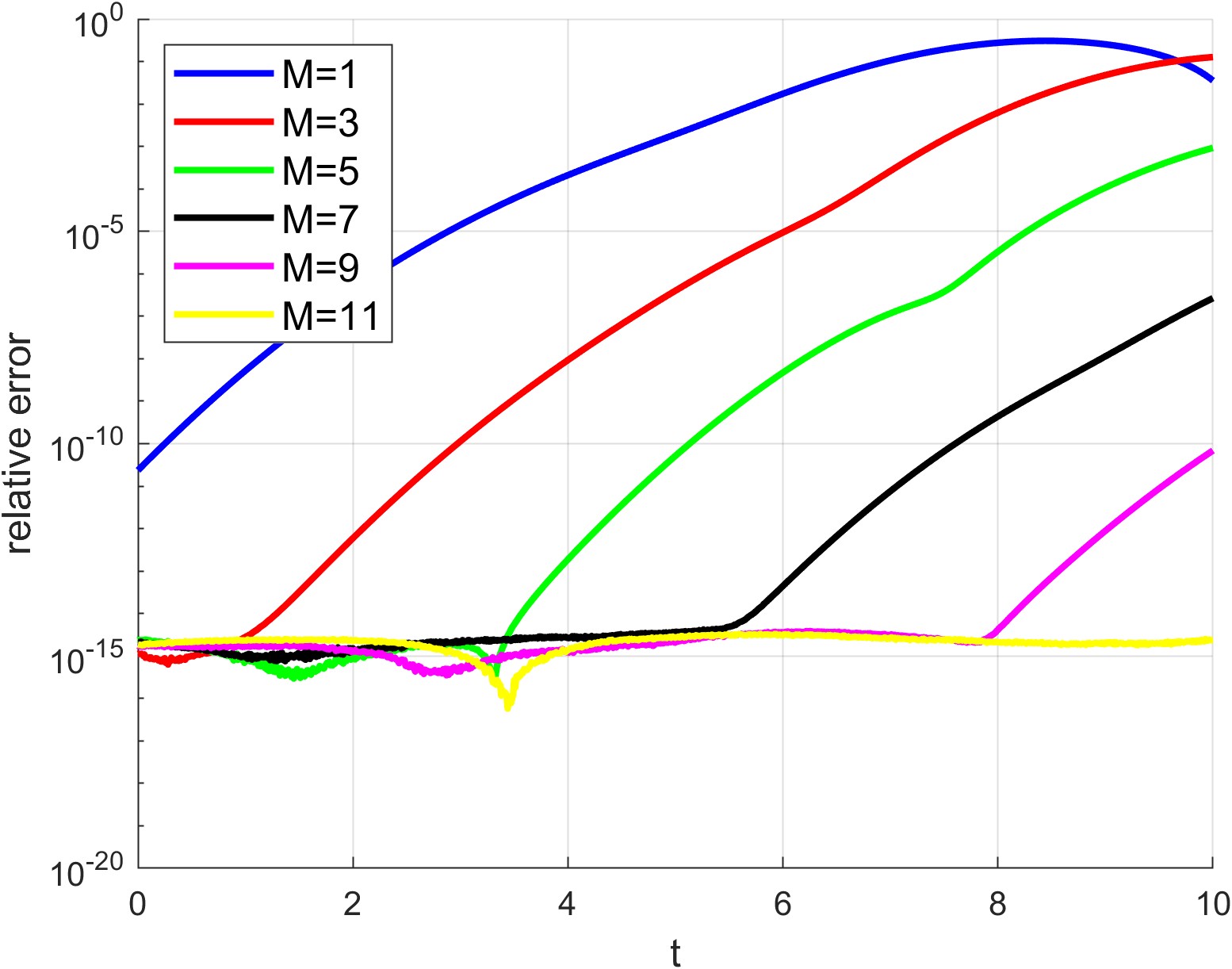} \\ (a) & (b)
\end{tabular}
\caption{Real part of the scattered field (a) at $x=(0.5,0)^\top$
resulting from the incident field $u_2^i$ with
$\theta^\textrm{inc}=-\pi$. and numerical errors (b) as functions of
time $t$ for various values of M.}
\label{Example5.1}
\end{figure}

\begin{figure}[htbp] \centering
  \begin{tabular}{cccc}
\includegraphics[scale=0.07]{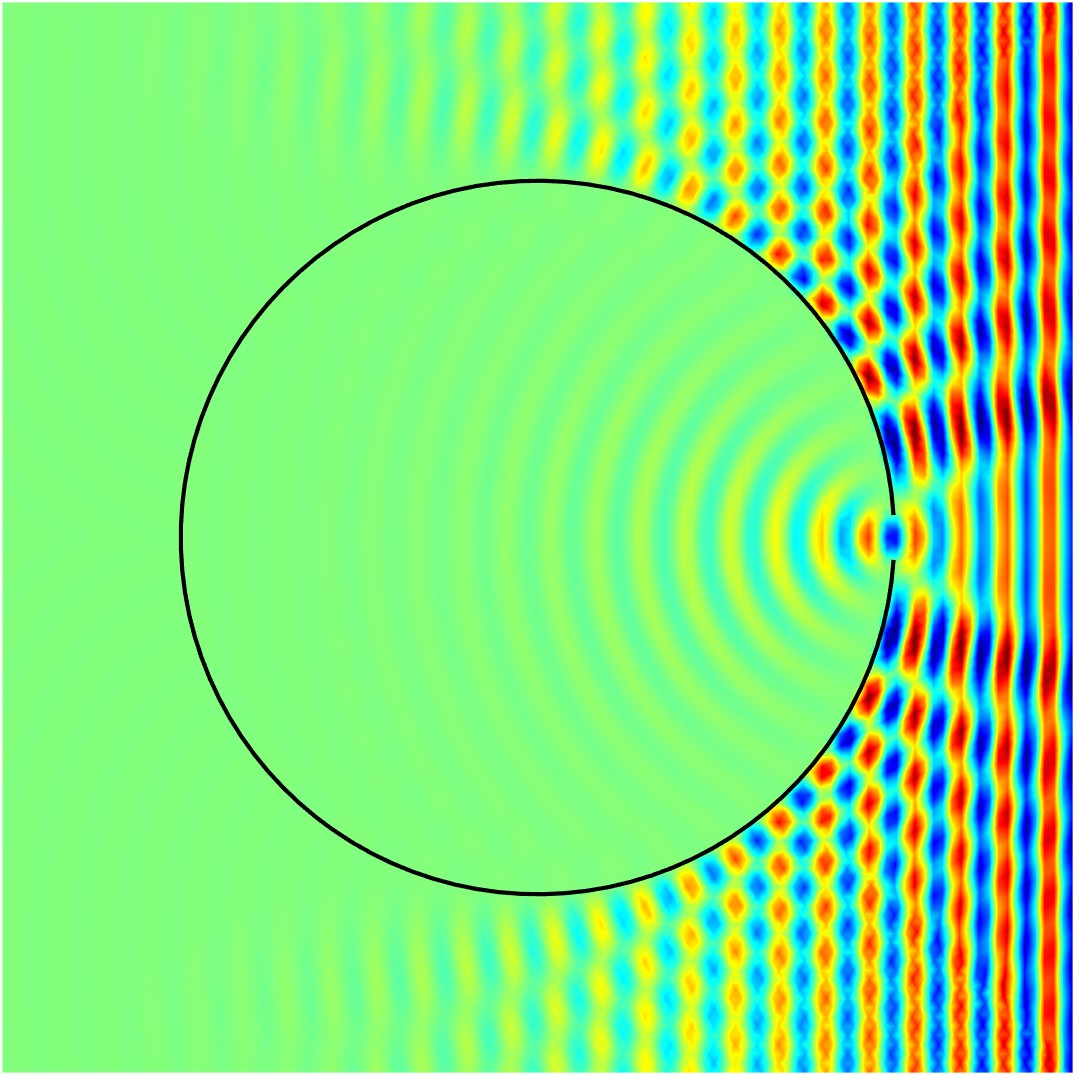} &
\includegraphics[scale=0.07]{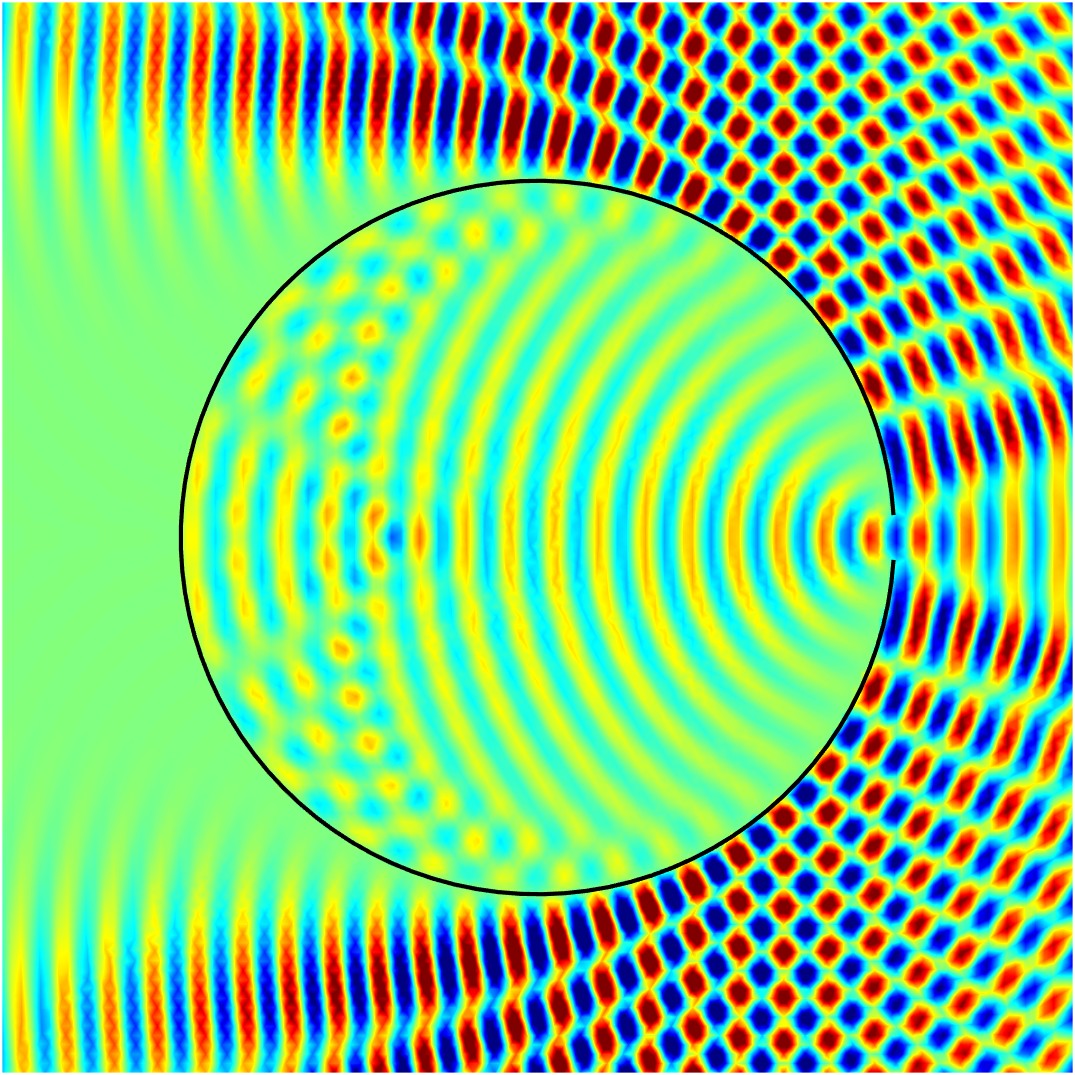} &
\includegraphics[scale=0.07]{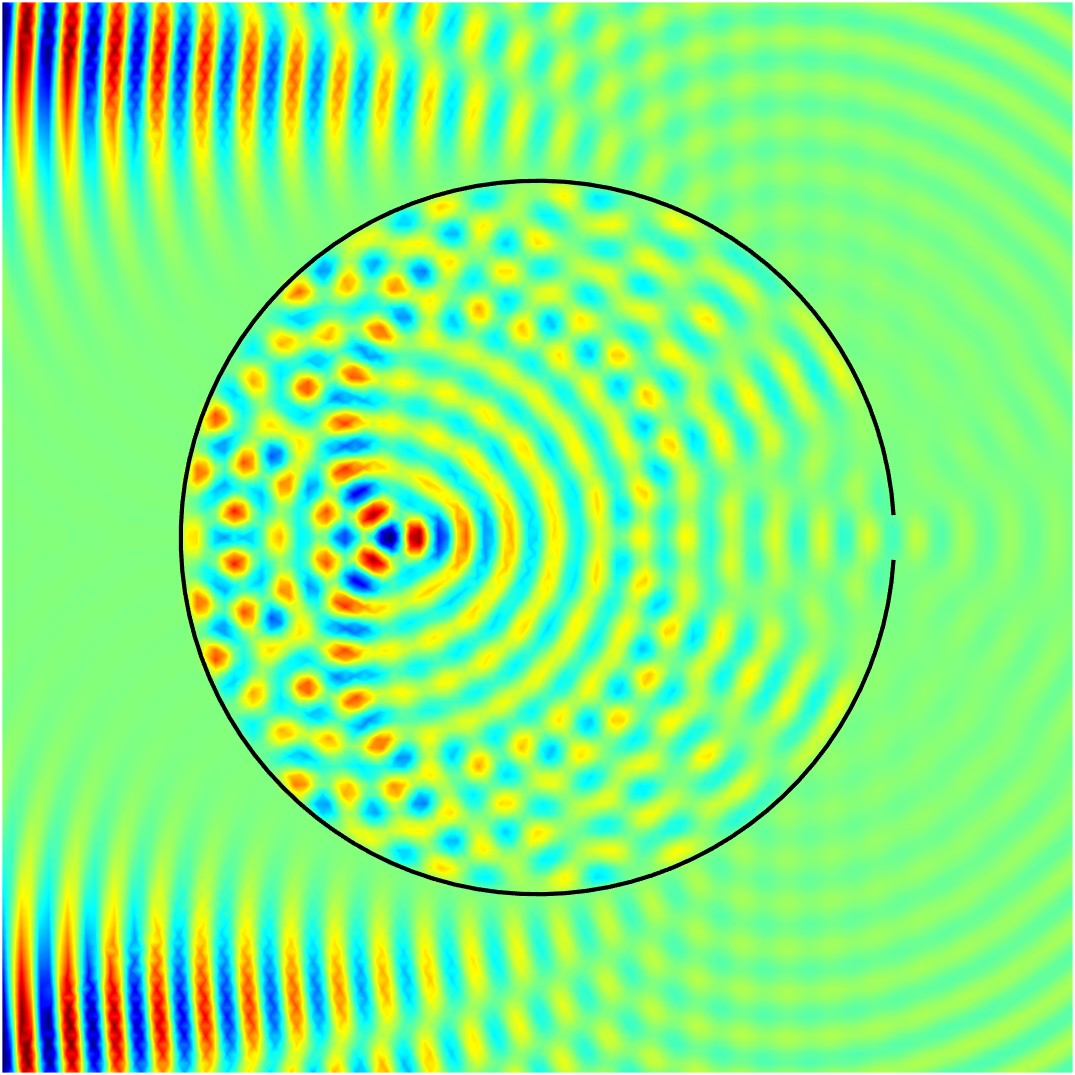} &
\includegraphics[scale=0.07]{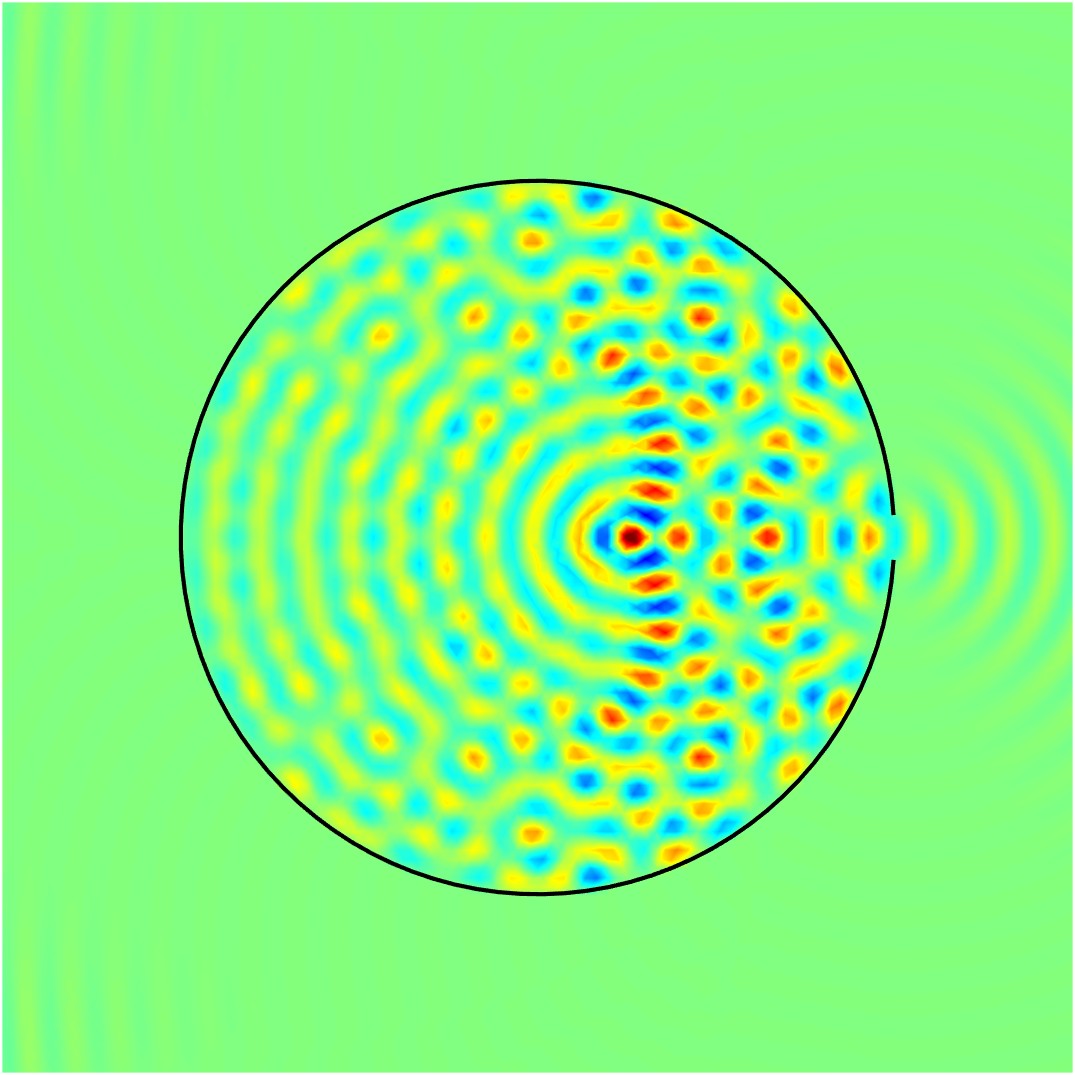}
  \end{tabular}
  \caption{Real part of the total fields for the scattering of plane
wave $u^i_2$ with $\theta^\textrm{inc}=-\pi$ by considered in Example
5. Fields at times $t = 4, 6, 8$ and 10 are displayed from left to
right.}
  \label{Example5.2}
  \end{figure}

\begin{figure}[htbp] \centering
\includegraphics[scale=0.1]{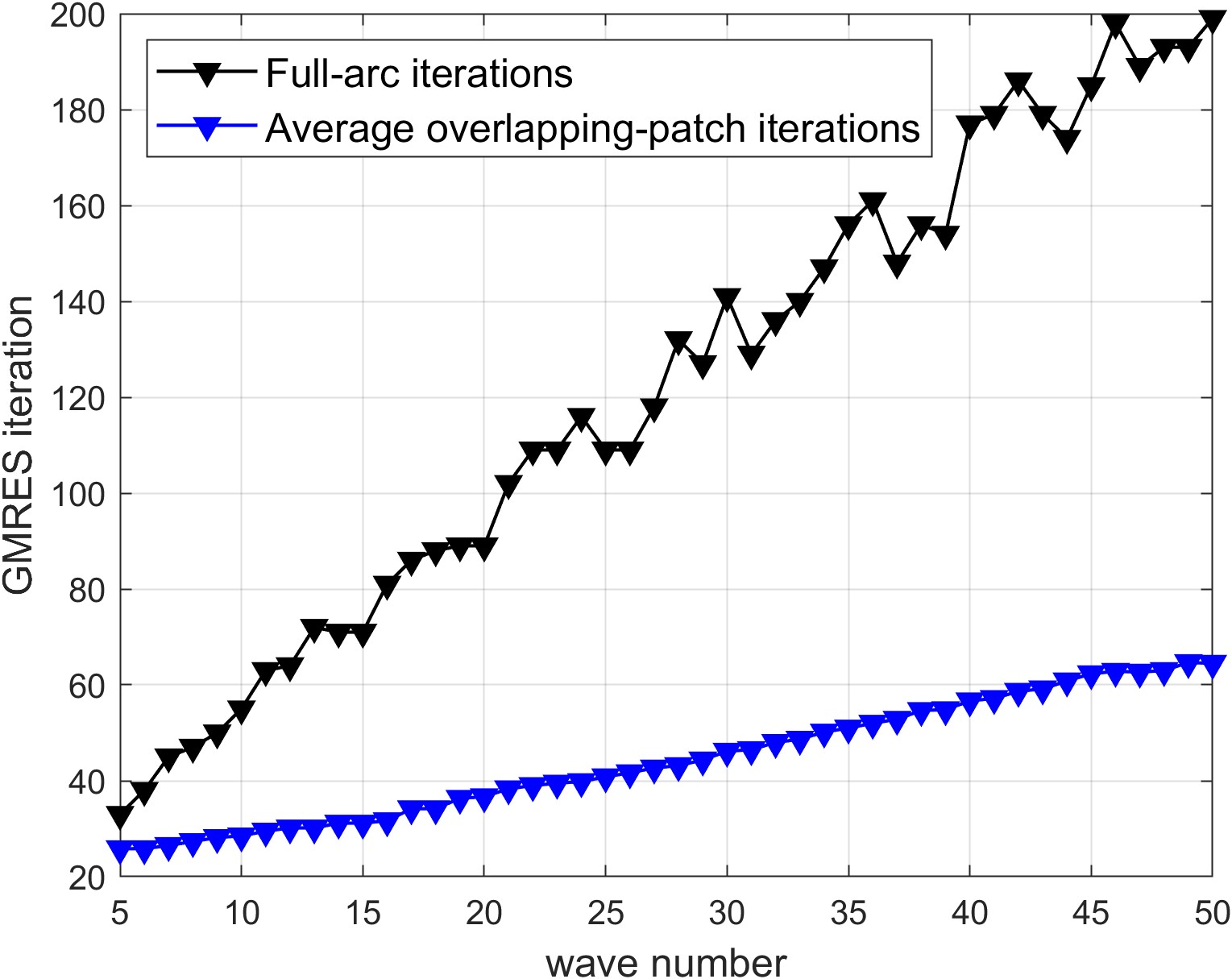}
\caption{GMRES iteration numbers for different frequencies while
solving the integral equation~(\ref{FDBIEk-interior}) on the full-arc
(circular cavity) or 6 overlapping patches such that the relative
residual is smaller than $10^{-6}$.}
\label{Example5.3}
\end{figure}

\begin{figure}[htbp] \centering
  \begin{tabular}{cccc}
\includegraphics[scale=0.07]{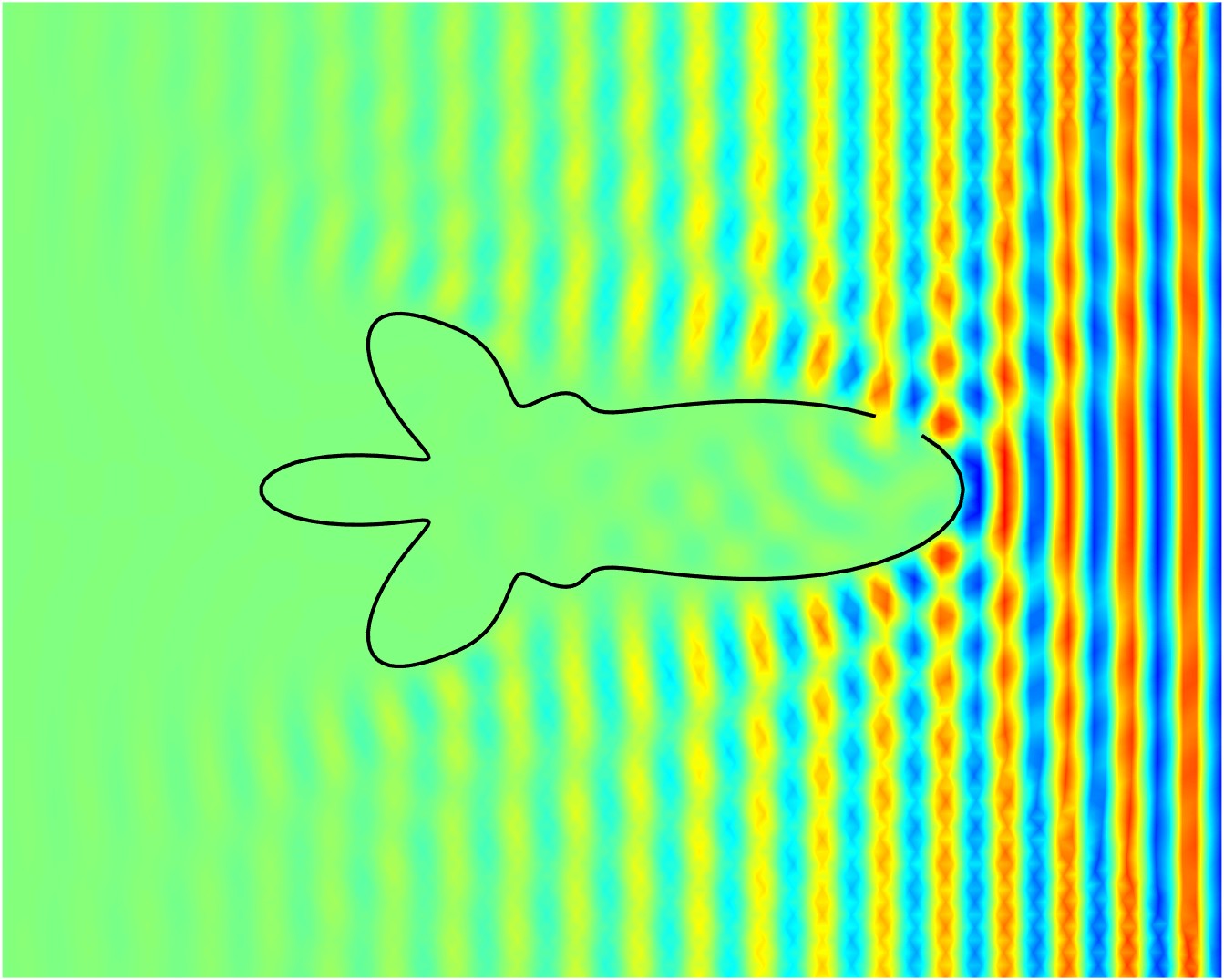} &
\includegraphics[scale=0.07]{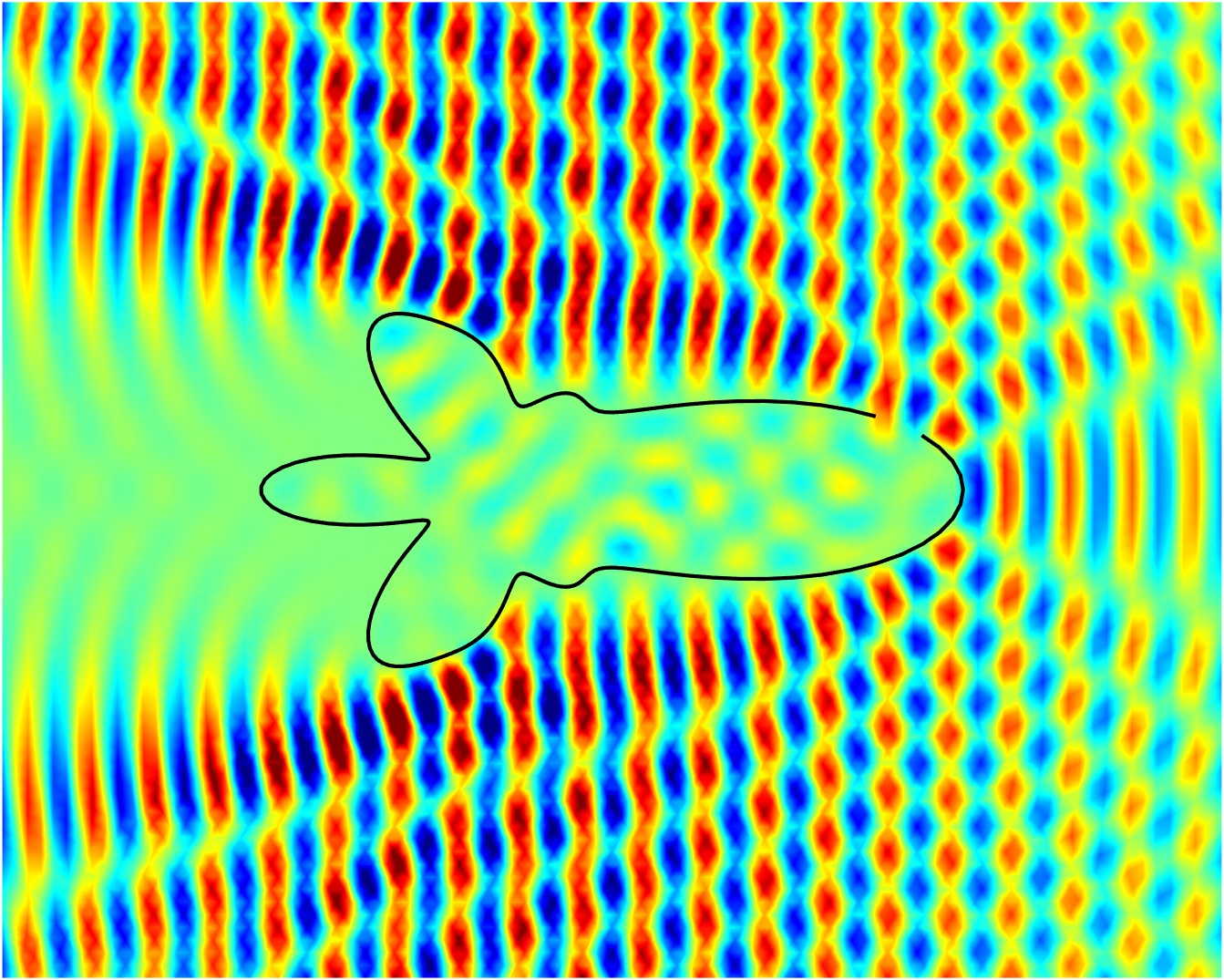} &
\includegraphics[scale=0.07]{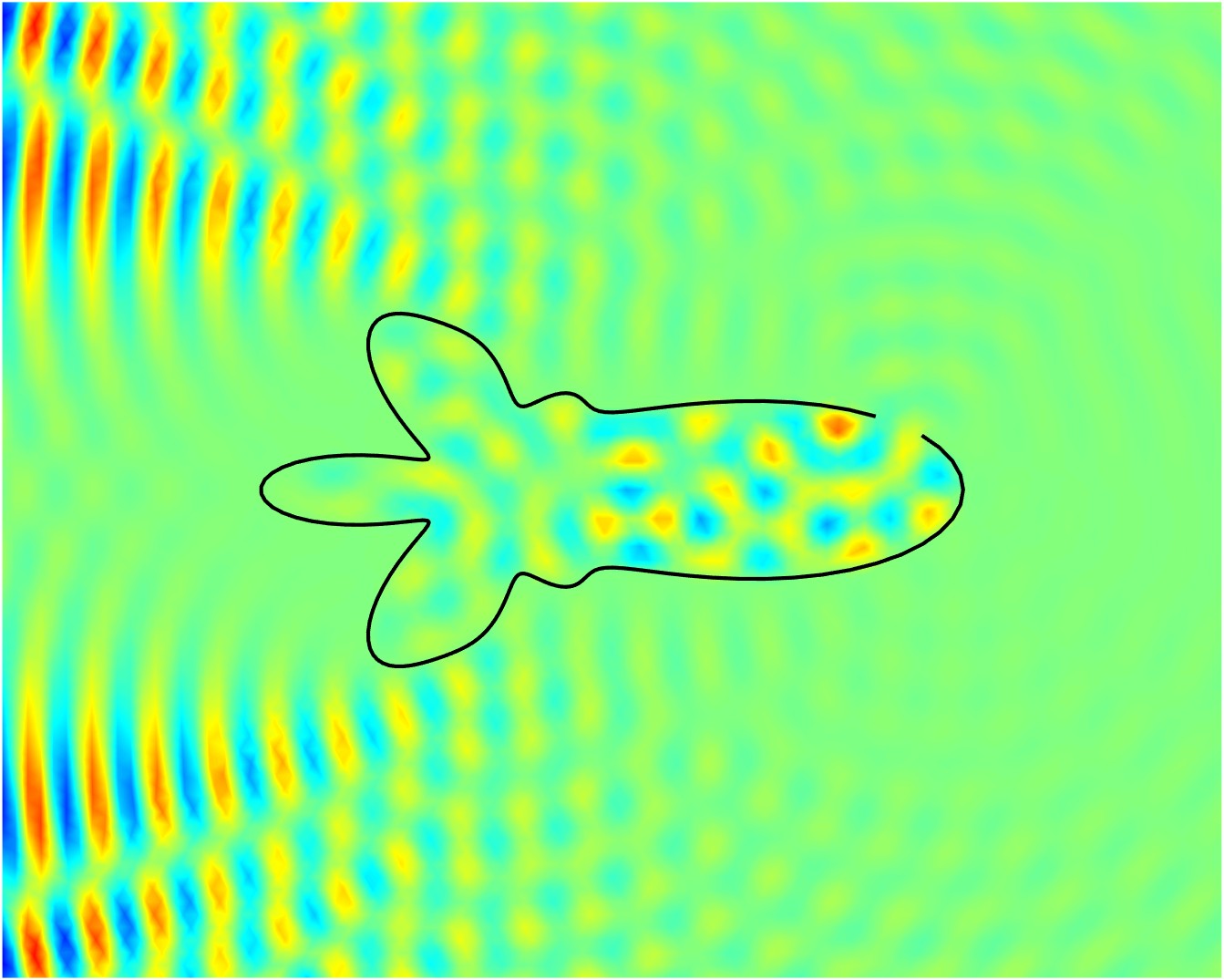} &
\includegraphics[scale=0.07]{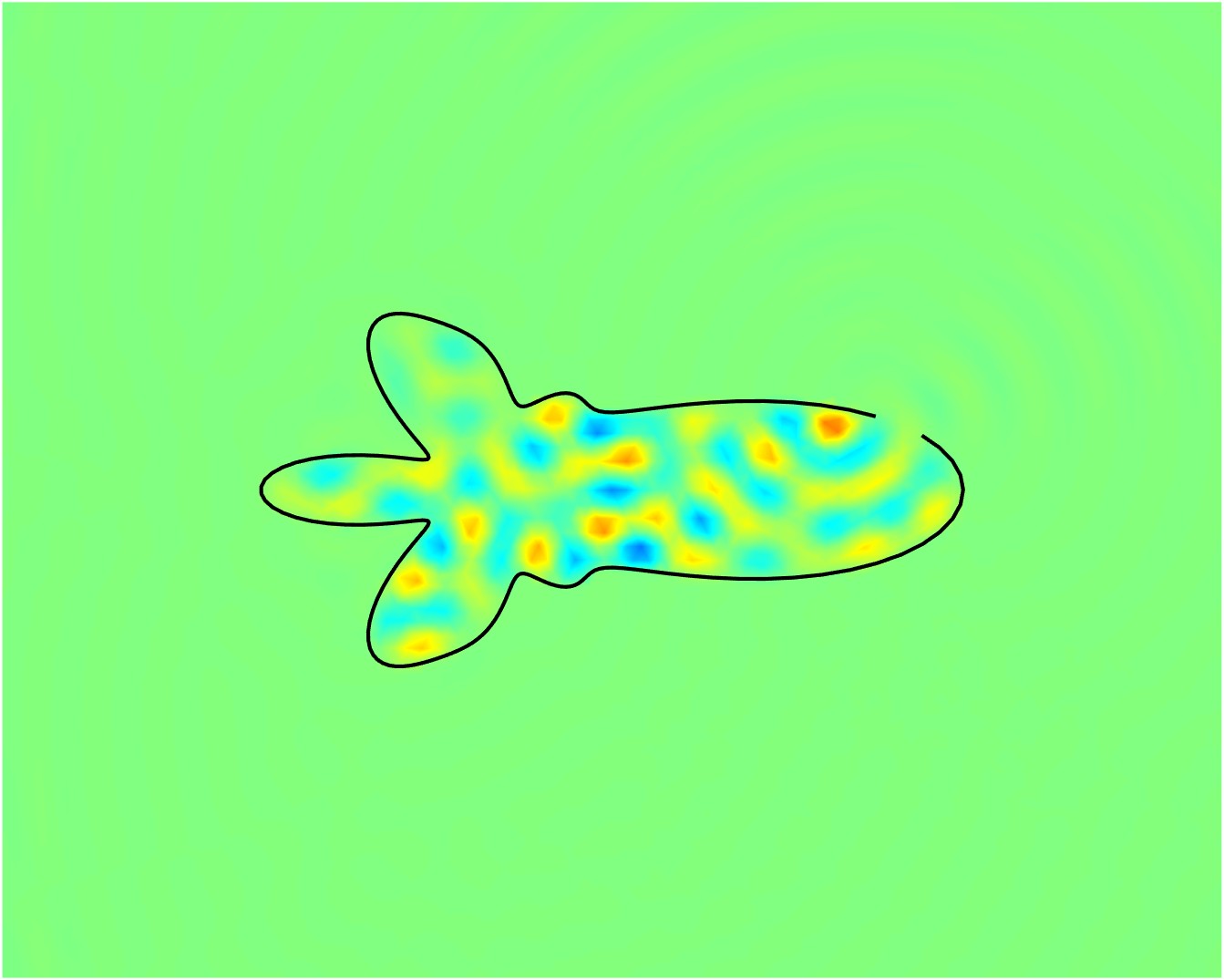}
  \end{tabular}
  \caption{Real part of the total fields for the scattering of plane
wave $u^i_2$ with $\theta^\textrm{inc}=-\pi$ by rocket-like open
cavity. Fields at times $t = 4, 6, 8$ and 10 are displayed from left
to right.}
  \label{Example5.4}
  \end{figure}

\section{Concluding remarks}
\label{sec:6}

In significant extension of a previous two-patch interior multiple-scattering wave equation solver~\cite{BY23}, and on the basis of a
newly introduced spatially-windowed multi-patch decomposition
strategy, this paper proposed new multi-patch multiple-scattering FTH-MS
methodologies for the solution of both interior and exterior wave
equation problems. In particular, on the basis of the Huygens-like
domain-of-influence condition, this paper established rigorously that
the sum of the overlapping-arc multiple scattering solutions coincides
with the solution to the original wave equation problem. Two key innovations distinguish the present contribution from the previous multi-scattering method~\cite{BY23}, namely
(i)~The use of a POU along the boundary now enables the treatment of an arbitrary number of overlapping arcs, whereas the earlier algorithm was limited to a maximum of two; and
(ii)~The new change-of-variables-based handling of density singularities arising from the overlapping-arc method replaces the ad-hoc technique used in~\cite{BY23}---which, relying on extension of the  scatterer boundaries along the normal direction, is rather cumbersome and  fails in the presence of open arcs on the scattering surface. In
all cases the solver transforms the original problem into a sequence
of scattering problems involving the given overlapping open arcs and
closed curves, each one of which is solved on the basis of the
``smooth time-windowing and recentering'' methodology and
high-frequency Fourier transform algorithms~\cite{ABL20}. The solves
required for the sequence of smooth temporal windows can clearly be
performed in parallel, allowing for implementations that leverage
parallelization of the time evolution. As demonstrated through various
numerical examples, the proposed methodologies enable long-time
simulations, including incident signals of extended duration, while
maintaining high accuracy and negligible dispersion errors. The
extension of these ideas to three-dimensional problems, as well as
their application to elastic and electromagnetic wave phenomena and
multi-layered media, is left for future work.

\section*{Acknowledgments} This work is partially supported by the
National Key R\&D Program of China, Grants No. 2024YFA1016000 and 2023YFA1009100,  the Strategic Priority Research Program of the Chinese Academy of Sciences, Grant No. XDB0640000, the Key Project of Joint Funds for Regional Innovation and Development (U21A20425) Province, NSFC under grants 12171465 and a Key Laboratory of Zhejiang Province. OB gratefully acknowledges support for NSF and AFOSR under contracts DMS-2109831, FA9550-21-1-0373 and FA9550-25-1-0015.

\end{document}